\documentclass[12pt]{memo-l}
\usepackage{color,amssymb,cite,amsmath,latexsym,comment,epic,eepic,euscript,mathtools,amsthm,amsfonts,bm,graphicx,float, caption, subcaption,yfonts,cite,xfrac}
\usepackage{graphicx}
\usepackage{enumerate}
\usepackage{float}
\usepackage[initials]{amsrefs}
\usepackage[all]{xy}
\usepackage{nomencl}
\usepackage{tikz}
\usetikzlibrary{positioning}
\usetikzlibrary{decorations.pathreplacing}
\theoremstyle{plain}

\newtheorem{thm}{Theorem}[chapter]
\newtheorem{example}[thm]{Example}
\newtheorem{cor}[thm]{Corollary} 
\newtheorem{lemma}[thm]{Lemma} 
\newtheorem{prop}[thm]{Proposition}

\numberwithin{section}{chapter}
\numberwithin{equation}{chapter}
\date{\today}
\makeindex

\theoremstyle{plain}

\newtheorem{remark}[thm]{Remark}

\theoremstyle{plain}
\newtheorem{defi}[thm]{Definition}
\newtheorem{definition}[thm]{Definition}
\newtheorem{corollary}[thm]{Corollary}
\newtheorem{proposition}[thm]{Proposition}

\newtheorem{notation}[thm]{Notation}

\newcount\theTime
\newcount\theHour
\newcount\theMinute
\newcount\theMinuteTens
\newcount\theScratch
\theTime=\number\time
\theHour=\theTime
\divide\theHour by 60
\theScratch=\theHour
\multiply\theScratch by 60
\theMinute=\theTime
\advance\theMinute by -\theScratch
\theMinuteTens=\theMinute
\divide\theMinuteTens by 10
\theScratch=\theMinuteTens
\multiply\theScratch by 10
\advance\theMinute by -\theScratch
\def\timeHHMM{{\number\theHour:\number\theMinuteTens\number\theMinute}}

\def\today{{\number\day\space
 \ifcase\month\or
  January\or February\or March\or April\or May\or June\or
  July\or August\or September\or October\or November\or December\fi
 \space\number\year}}

\def\timeanddate{{\timeHHMM\space o'clock, \today}}
\newcommand{\beq}{\begin{equation}}
\newcommand{\eeq}{\end{equation}}
\newcommand{\beqn}{\begin{equation*}}
\newcommand{\eeqn}{\end{equation*}}
\newcommand{\bre}{\begin{remark}}
\newcommand{\ere}{\end{remark}}
\newcommand{\beqno}[1]{\begin{equation}\label{#1}}

\newcommand{\mref}[1]{(\ref{#1})}

\newcommand\X{\mathcal{X}}

\newcommand\R{{\mathbb{R}}}
\newcommand\T{{\mathbb{T}}}

\newcommand\bZ{\mathbb Z}

\newcommand\N{\mathbb N}

\newcommand\cV{\mathcal V}
\newcommand\cF{\mathcal F}
\newcommand\Z{\mathbb Z}
\newcommand*\di{\mathop{}\!\mathrm{d}}

\DeclareMathOperator{\sign}{sign}
\DeclareMathOperator{\supp}{supp}
\DeclareMathOperator{\dist}{dist}
\DeclareMathOperator{\Hess}{Hess}
\DeclareMathOperator{\Dom}{Dom}
\DeclareMathOperator{\Ker}{Ker}
\DeclareMathOperator{\Image}{Im}
\DeclareMathOperator{\Domp}{Dom\cprime}
\DeclareMathOperator{\Index}{ind}
\DeclareMathOperator{\codim}{codim}
\DeclareMathOperator{\Span}{span}
\DeclareMathOperator{\vol}{vol}

\newcommand{\floor}[1]{\lfloor #1 \rfloor}

\newenvironment{prsubexp-solvability}{\paragraph{\textit{Proof of Corollary \ref{subexp-solvability}}}}{\hfill$\square$}
\begin{document}
\frontmatter

\title[Liouville-Riemann-Roch]
{Liouville-Riemann-Roch theorems on abelian coverings}

\author[Kha]{Minh Kha$^{*}$}
\address{M.K., Department of Mathematics, University of Arizona,
Tucson, AZ 85721-0089, USA}
\email{minhkha@math.arizona.edu}
\thanks{The work of both authors was partially supported by the NSF grant DMS-1517938. M. Kha was also partially supported by the AMS-Simons Travel Grant.}

\author[Kuchment]{Peter Kuchment$^{*}$}
\address{P.K., Department of Mathematics, Texas A\&M University,
College Station, TX 77843-3368, USA}
\email{kuchment@math.tamu.edu}


\date{\timeanddate}



\dedicatory{
\begin{center}
Dedicated to the memory of dear friends and wonderful mathematicians\\ Carlos Berenstein and Misha Boshernitzan\
\end{center}
}
\maketitle

\tableofcontents

\begin{abstract}
The classical Riemann-Roch theorem has been extended by N. Nadirashvili and then M. Gromov and M. Shubin to computing indices of elliptic operators on compact (as well as non-compact) manifolds, when a divisor mandates a finite number of zeros and allows a finite number of poles of solutions.

On the other hand, Liouville type theorems count the number of solutions that are allowed to have a ``pole at infinity.'' Usually these theorems do not provide the exact dimensions of the spaces of such solutions (only finite-dimensionality, possibly with estimates or asymptotics of the dimension, see e.g. \cite{ColdingMin,ColdingMin_excursion,ColdMinICM,Yau87,LiBook,Yau93}). An important case has been discovered by M. Avellaneda and F. H. Lin and advanced further by J. Moser and M. Struwe. It pertains periodic elliptic operators of divergent type, where, surprisingly, exact dimensions can be computed. This study has been extended by P. Li and Wang and brought to its natural limit for the case of periodic elliptic operators on co-compact abelian coverings by P. Kuchment and Pinchover.

Comparison of the results and techniques of Nadirashvili and Gromov and Shubin with those of Kuchment and Pinchover shows significant similarities, as well as appearance of the same combinatorial expressions in the answers. Thus a natural idea was considered that possibly the results could be combined somehow in the case of co-compact abelian coverings, if the infinity is ``added to the divisor.''

This work shows that such results indeed can be obtained, although they come out more intricate than a simple-minded expectation would suggest. Namely, the interaction between the finite divisor and the point at infinity is non-trivial.

\end{abstract}

\chapter{Introduction}\label{C:intro}
Counting the number (i.e., dimension of the space), or even confirming existence of solutions of an elliptic equation on a compact manifold (or in a bounded domain in $\R^n$) is usually a rather impossible task, unstable with respect to small variations of parameters. On the other hand, the Fredholm index of the corresponding operator, as it was conjectured by I.~Gel'fand \cite{Gelfand_index} and proven by M.~Atiyah and I.~Singer \cite{AtSi1,AtSi3,AtSi4,AtSi5,AtSi_indexBAMS}, can be computed in topological terms. In particular, if the index of an operator $L$ happens to be positive, this implies non-triviality of its kernel.

One might also be interested in index formulas in the case when the solutions are allowed to have some prescribed poles and have to have some mandatory zeros. Probably the first result of this kind was the century older classical Riemann-Roch theorem \cite{Riemann,Roch}, which in an appropriate formulation provides the index of the $\overline{\partial}$-operator on a compact Riemannian surface, when a devisor of zeros and poles is provided. This result was extended by N.~Nadirashvili \cite{nadirashvili} to the Laplace-Beltrami operator on a complete compact\footnote{A version for non-compact manifolds with ``appropriate'' conditions at infinity was also given.} Riemannian manifold with a prescribed divisor. This result has been generalized by M.~Gromov and M.~Shubin \cite{GSadv,GSinv} to computing indices of elliptic operators in vector bundles over compact (as well as non-compact) manifolds, when a divisor mandates a finite number of zeros and allows a finite number of poles of solutions.

On the other hand, Liouville type theorems count the number of solutions that are allowed to have a ``pole at infinity.'' Solution of an S.-T.~Yau's problem \cite{Yau87,Yau93}, given by C.~Colding and W.Minicozzi II \cite{ColdingMin,ColdingMin_excursion} shows that on a Riemannian manifold of non-negative Ricci curvature, the spaces of harmonic functions of fixed polynomial growths are finite-dimensional. The result also applies to the Laplace-Beltrami operator on an nilpotent covering of a compact Riemannian manifold. No explicit formulas for these dimensions are available. However, an interesting case has been discovered by M.~Avellaneda and F.~H.~Lin \cite{AvLin} and J.~Moser and M.~Struwe \cite{MoserStruwe}. It pertains periodic elliptic operators of divergent type, where exact dimensions can be computed and coincide with those for the Laplacian. This study has been extended by P.~Li and J.~Wang \cite{LiWang01} and brought to its natural limit in the case of periodic elliptic operators on co-compact abelian coverings by P.~Kuchment and Y.~Pinchover \cite{KP1,KP2}.

A comparison of the works of Nadirashvili and Gromov and Shubin (N-G-S) and of Kuchment and Pinchover shows significant similarities in the techniques, as well as appearance of the same combinatorial expressions in the answers. Thus a natural idea was considered that possibly the results could be combined somehow in the case of co-compact abelian coverings, if the infinity is ``added to the divisor.'' In fact, the results of Nadirashvili, Gromov, and Shubin allowed having some conditions at infinity, if these lead to some kind of ``Fredholmity.'' When one considers polynomial growth conditions at infinity, i.e. Liouville property, the results of  \cite{KP1,KP2} show that such ``Fredholmity'' can be expected at the edges of the spectrum (in more technical terms, when the Fermi surface is discrete). Outside of the spectrum this also works, vacuously. On the other hand, inside the spectrum the space of polynomially growing solutions is infinitely dimensional, which does not leave much hope for Liouville-Riemann-Roch type results. We thus concentrate on the spectral edge case (the results of \cite{AvLin,MoserStruwe} correspond to the case of the bottom of the spectrum).

This work shows that such results indeed can be obtained, although they come out more intricate than a simple-minded expectation would suggest. The interaction between the finite divisor and the point at infinity turns out to be non-trivial. Also, there is dependence upon the $L_p$ norm with respect to which the growth is measured (in the works \cite{KP1,KP2} only the case when $p=\infty$ was considered).

The structure of the text is as follows: Chapter \ref{C:preliminaries-LRR} contains a survey of the required information about periodic elliptic operators, Liouville type theorems (including some new observations there), and Nadirashvili-Gromov-Shubin version of the Riemann-Roch theorem. Chapter \ref{C:main} contains formulations of the main results of this work. In some cases, the Riemann-Roch type equalities cannot be achieved (counterexamples exist), while inequalities still hold. These inequalities, however, can be applied the same way the equalities are for proving existence of solutions of elliptic equations with prescribed zeros, poles, and growth at infinity. Proofs of these results are mostly contained in Chapter \ref{C:proofs-LRR}, while proofs of some technical auxiliary statements are delegated to Chapter \ref{C:auxiliary}. Chapter \ref{C:applications-LRR} provides applications to some specific examples. The work ends with Chapter \ref{C:remarks} containing final remarks, acknowledgments, and bibliography.

\mainmatter
\chapter{Preliminaries}
\label{C:preliminaries-LRR}

\section[Periodic operators]{Periodic elliptic operators on abelian coverings}
Let $X$ be a noncompact smooth Riemannian manifold of dimension $n$ equipped with an isometric, properly discontinuous, free, and co-compact\footnote{I.e., its quotient (orbit) space is compact.}
action of a finitely generated abelian discrete group $G$. We denote by $g\cdot x$ the action of an element $g \in G$ on $x \in X$.

Consider the (compact) orbit space $M:=X/G$.
We thus are dealing with a regular abelian covering $\pi$ of a compact manifold:
$$
X \xrightarrow{\pi} M,
$$
with $G$ as its {\bf deck group}.

\begin{remark}
Not much harm will be done, if the reader assumes that $X=\R^d$ and $M$ is the torus $\T^d=\R^d/\Z^d$. The results are new in this case as well. The only warning is that in this situation the dimension of $X$ and the rank of the group $\Z^d$ coincide, while this is not required in general. As we will show, the distinction between the rank $d$ of the deck group and dimension $n$ of the covering manifold $X$ pops up in some results.
\end{remark}

Let  $\mu_M$ be the Riemannian measure of $M$ and $\mu_X$ be its lifting to $X$. Then $\mu_X$ is a $G$-invariant Riemannian measure on $X$. We denote by $L^2(X)$ the space of $L^2$-functions on $X$ with respect to $\mu_X$.

We also consider the $G$-invariant bilinear\footnote{One can also consider the sesquilinear form to obtain analogous results.} duality
\beq
\label{L2inner}
\langle\cdot, \cdot\rangle: C^{\infty}_c(X) \times C^{\infty}(X) \rightarrow \mathbb{C}, \quad
\langle f,g \rangle=\int_X f(x)g(x)\di\mu_X.
\eeq
It extends by continuity to a $G$-invariant bilinear non-degenerate duality
\beq
\langle\cdot, \cdot\rangle: L^2(X) \times L^2(X)  \rightarrow \mathbb{C}.
\eeq

Let $A$ be an elliptic\footnote{Here ellipticity means that the principal symbol of the operator $A$ does not vanish on the cotangent bundle with the zero section removed
$T^*X \setminus (X \times \{0\})$.} operator of order $m$ on $X$ with smooth coefficients. We will be assuming that $A$ is \textbf{$G$-periodic}, i.e.,  $A$ commutes with the
action of $G$ on $X$. Then $A$ can be pushed down to an elliptic operator $A_M$ on $M$. Equivalently, $A$ is the lifting of $A_M$ to $X$.
We will assume in most cases (except unfrequent non-self-adjoint considerations) that the operator $A$ is bounded below.

The formal adjoint $A^{*}$ (transpose with respect to the bilinear duality \mref{L2inner}) to $A$ is also a periodic elliptic operator of order $m$ on $X$.

Note that since $G$ is a finitely generated abelian group, it is the direct sum of a finite abelian group and $\bZ^d$, where $d$ is the \textbf{rank} of the torsion free subgroup of $G$. One can always eliminate the torsion part of $G$ by switching to a sub-covering $X \rightarrow X/\bZ^d$. In what follows, without any effect on the results, we could replace $M$ by the compact Riemannian manifold $X/\bZ^d$ and thus, we can work with $\bZ^d$ as our new deck group. Therefore,
\begin{center}
\textbf{we assume henceforth that $G=\bZ^d$, where $d \in \mathbb{N}$.}
\end{center}

The \textbf{reciprocal lattice} $G^*$ for the deck group $G=\bZ^d$ is $(2\pi \mathbb{Z})^d$ and we choose $B=[-\pi,\pi]^{d}$ as its fundamental domain (\textbf{Brillouin zone} in physics).
The quotient $\mathbb{R}^d/G^*$ is a torus, denoted by $\mathbb{T^*}^d$. So, $G^*$-periodic functions on $\mathbb{R}^d$ can be naturally identified with functions
on the torus $\mathbb{T^*}^d$.

For any \textbf{quasimomentum} $k \in \mathbb{C}^d$, let $\gamma_k$ be the \textbf{character} of the deck group $G$ defined as $\gamma_k(g)=e^{ik\cdot g}$ (a quasimomentum is defined modulo the reciprocal lattice).
If $k$ is real, $\gamma_k$ is unitary and vice versa. Abusing the notations slightly, we will sometimes identify a unitary character $\gamma_k$, which belongs to the dual group $\mathbb{T}^d$ of $\bZ^d$, with its quasimomentum $k\in B$.
\begin{definition}
We denote by $L^{2}_k(X)$ the space of all $\gamma_k$-automorphic function $f(x)$ in $L^2_{loc}(X)$, i.e. such that
\beq
\label{quasiperiodic}
f(g\cdot x)=\gamma_k(g)f(x), \text{ for a.e }  x\in X \text{ and } g\in G.
\eeq
\end{definition}
It is convenient at this moment to introduce, given a quasimomentum $k$, the following line (i.e., one-dimensional) vector bundle $E_k$ over $M$:

\begin{defi}
\label{assbundle}
Given any $k \in \mathbb{C}^d$, we consider the free
left action of $G$ on the Cartesian product $X \times \mathbb{C}$ (a trivial linear bundle over $X$) given by
$$
g \cdot (x,z)=(g\cdot x,\gamma_{k}(g)z), \quad (g,x,z) \in G \times X \times \mathbb{C}.
$$
Now $E_k$ is defined as the orbit space of this action. Then the canonical projection $X \times \mathbb{C} \rightarrow M$ descends to the surjective mapping $E_k \rightarrow M$, thus defining
a linear bundle $E_k$ over $M$  (see e.g., \cite{spin}).
\end{defi}
\begin{remark}
The space $L^{2}_k(X)$ can be naturally identified with the space of $L^2$-sections of the bundle  $E_k$.
\end{remark}
This construction can be easily generalized to Sobolev spaces:
\begin{defi}
\label{sobolevsection}
 For a quasimomentum $k \in \mathbb{C}^d$ and a real number $s$, we denote by $H^s_k(X)$ the closed subspace of $H^s_{loc}(X)$ consisting of $\gamma_k$-automorphic (i.e., satisfying (\ref{quasiperiodic})) functions.
 Then $H^s_k(X)$ is a Hilbert space, when equipped with the natural inner product induced by the inner product in the Sobolev space $H^s(\overline{\mathcal{F}})$,
 where $\mathcal{F}$ is any fixed fundamental domain for the action of the group  $G$ on $X$.

 Equivalently, the space $H^s_k(X)$ can be identified with the space $H^s(E_k)$ of all $H^s$-sections of the bundle $E_k$.
\end{defi}
For any $k$, the periodic operator $A$ maps continuously $H^m_k(X)$ into $L^2_k(X)$.
This defines an elliptic operator $A(k)$ on the spaces of sections of the
bundles $E_k$ over the compact manifold $M$. When $A$ is self-adjoint and $k$ is real, the operator $A(k)$, with
the space $H^m(E_k)$ as the domain, is an unbounded, bound below self-adjoint operator in $L^2(E_k)$. Thus its spectrum is discrete
and eigenvalues can be labeled in non-decreasing order as
\begin{equation}
 \lambda_1(k)\leq \lambda_2(k)\leq \dots \rightarrow \infty.
\end{equation}
A simple application of perturbation theory shows that $\lambda_j(k)$ are continuous piecewise-analytic $G^*$-periodic functions of $k$.
Their ranges over the torus $\T^d$ are closed intervals $I_j$ of the real axis, called \textbf{spectral bands}. The spectral bands tend to infinity, when $j\to\infty$ and may overlap (with any point being able to belong to at most finitely many bands), while they might leave some \textbf{spectral gaps} uncovered (e.g., \cite{Ksurvey,Wilcox}.

\section{Floquet transform}\label{S:Floquetrans}

Fourier transform is a major tool of studying linear constant coefficient PDEs, due to their invariance with respect to all shifts. The periodicity of the operator $A$ suggests that it is natural to apply the Fourier transform with respect to the period group $G$ to block-diagonalize $A$.

The group Fourier transform we have just mentioned is the so called \textbf{Floquet transform} $\textbf{F}$ (see e.g., \cite{Kbook, Ksurvey, Gelfand}):
\beq
\label{fl}
f(x) \mapsto \textbf{F}f(k,x)=\sum_{g \in G}f(g \cdot x)\overline{\gamma_k(g)}=\sum_{g \in G}f(g \cdot x)e^{-ik\cdot g}, \quad k \in \mathbb{C}^d.
\eeq
For reader's convenience, we collect some basic properties of Floquet transform in Chapter \ref{C:auxiliary}.

This transform, as one would expect, decomposes the original operator $A$ on the non-compact manifold $X$ into a direct integral of operators $A(k)$ acting on sections of line bundles $E_k$ over $k$ in torus $\mathbb{T}^d$:
\beq
A=\int\limits_{\mathbb{T^*}^d}^{\oplus}A(k)\di{k}, \quad L^2(X)=\int\limits_{\mathbb{T^*}^d}^{\oplus}L^2(E_k)\di{k}.
\eeq
Here the measure $\mbox{d}k$ is the normalized Haar measure on the torus $\mathbb{T^*}^d$, which can be also considered as the normalized Lebesgue measure on the Brillouin zone $B$.\footnote{In fact, we are mixing up the quasimomenta and characters here, and more appropriate formulas would be
\beq
A=\int\limits_{B}^{\oplus}A(k)\di{k}, \quad L^2(X)=\int\limits_{B}^{\oplus}L^2(E_k)\di{k}.
\eeq
We will keep periodically abusing notations this way.}

One of the important consequences of this decomposition via the Floquet transform is the following well-known result, which in particular justifies the names ``spectral band'' and ``spectral gap'':
\begin{thm}
\label{specA}
\cite{KOS,Kbook, RS4}
The union of the spectra of operators $A(k)$ over the torus $\mathbb{T}^d$ is the spectrum of the periodic operator $A$: In other words, we have
\begin{equation}
\label{fl_spectrum}
\sigma(A)=\bigcup_{k \in \mathbb{T}^d}\sigma(A(k)).
\end{equation}
In the self-adjoint case, this can be rewritten as
\begin{equation}
\label{fl_spectrum2}
\sigma(A)=\bigcup_{j\in\N}I_j=\bigcup_{j\in\N, k \in \mathbb{T}^d}\{\lambda_j(k)\},
\end{equation}
where $\lambda_j(k)$ are the eigenvalues of the operator $A(k)$, listed in non-decreasing order, and the finite closed segment $I_j$ is the range of the function $\lambda_j(k)$.
\end{thm}

\section{Bloch and Fermi varieties}\label{S:Bloch}

We now recall a notion that plays a crucial role in studying periodic PDEs (see e.g., \cite{Kbook, Ksurvey}).
\begin{defi}
\label{blochfermi}
\begin{enumerate}[a.]
\item
The (complex) \textbf{Bloch variety} $B_{A}$ of the operator $A$ consists of all pairs $(k,\lambda) \in \mathbb{C}^{d+1}$ such that $\lambda$ is an eigenvalue of the operator $A(k)$.
\begin{equation*}
\label{disp_relation}
B_A=\{(k,\lambda) \in \mathbb{C}^{d+1}: \lambda \in \sigma(A(k))\}.
\end{equation*}
Thus, the Bloch variety can be seen as the graph of the multivalued function $\lambda(k)$, which is also called the \textbf{dispersion relation}\footnote{In other words, it is the graph of the function $k\mapsto \sigma(A(k))$.}.
\item
The (complex) \textbf{Fermi surface} $F_{A,\lambda}$ of the operator $A$ at the energy level $\lambda \in \mathbb{C}$ consists of all quasimomenta $k \in \mathbb{C}^{d}$ such that the equation $A(k)u=\lambda u$ has a nonzero solution\footnote{In physics, the name ``Fermi surface'' is reserved only to a specific value of energy $\lambda_F$, called \textbf{Fermi energy} \cite{AM}. For other values of $\lambda$, the names \textbf{equi-energy}, or \textbf{constant energy surface} are used. For our purpose, the significance of the Fermi energy evaporates, and so we extend the name Fermi surface to all (even complex) values of $\lambda$.}. Equivalently, the Fermi surface $F_{A,\lambda}$ is the $\lambda$-level set of the dispersion relation. By definition, $F_{A,\lambda}$ is $G^*$-periodic.
\item
We denote by $B_{A,\mathbb{R}}$ and $F_{A,\lambda,\mathbb{R}}$ the \textbf{real Bloch variety} $B_A \cap \mathbb{R}^{d+1}$ and the \textbf{real Fermi surface} $F_{A, \lambda} \cap \mathbb{R}^d$, respectively.
\item
Whenever $\lambda=0$, we will write $F_A$ and $F_{A,\mathbb{R}}$ instead of $F_{A,0}$ and $F_{A,0,\mathbb{R}}$, correspondingly.
This is convenient, since being at the spectral level $\lambda$, we could consider the operator $A-\lambda I$ instead of $A$ and thus, $F_{A,\lambda}=F_{A-\lambda}$ and $F_{A,\lambda,\mathbb{R}}=F_{A-\lambda,\mathbb{R}}$. In other words, we will be able to assume w.l.o.g. that $\lambda=0$.
\end{enumerate}
\end{defi}
Some important properties of Bloch variety and Fermi surface are stated in the next proposition (see e.g., \cite{Kbook, Ksurvey, KP2}).
\begin{prop}
\label{Panalytic}\indent
\begin{enumerate}[a.]
\item
The Fermi surface and the Bloch variety are the zero level sets of some entire ($G^*$-periodic in $k$) functions of finite orders on $\mathbb{C}^d$ and $\mathbb{C}^{d+1}$ respectively.
\item
The Bloch (Fermi) variety is a $G^*$-periodic, complex analytic, codimension one subvariety of $\mathbb{C}^{d+1}$ (correspondingly $\mathbb{C}^{d}$).
\item
The real Fermi surface $F_{A,\lambda}$ either has zero measure in $\mathbb{R}^d$ or coincides with the whole $\mathbb{R}^d$.
\item
$(k,\lambda) \in B_A$ if and only if $(-k,\bar{\lambda}) \in B_{A^*}$. In other words, $F_{A,\lambda}=-F_{A^*,\bar{\lambda}}$ and $F_{A,\mathbb{R}}=-F_{A^*,\mathbb{R}}$.
\end{enumerate}
\end{prop}
The Fermi and Bloch varieties encode much of crucial spectral information about the periodic elliptic operator. For example, the absolute continuity of the spectrum of a self-adjoint periodic elliptic operator (which is true for a large class of periodic Schr\"odinger operators) can be reformulated as the absence of flat components in its Bloch variety, which is also equivalent to a seemingly stronger fact that the Fermi surface at each energy level has zero measure (due to Proposition \ref{Panalytic}).

\section[Bloch functions]{Floquet-Bloch functions and solutions}\label{S:Bolch_funct}

In this section, we introduce the notions of Bloch and Floquet solutions of periodic PDEs and then state the Liouville theorem of \cite{KP2}.
\begin{defi}
\label{Floquetsol}
For any $g \in G$ and quasimomentum $k \in \mathbb{C}^d$, we denote by $\Delta_{g;k}$ the ``$k$-twisted'' version of the first difference operator acting on functions on the covering $X$ as follows:
\beq
\Delta_{g;k}u(x)=e^{-ik\cdot g}u(g \cdot x)-u(x).
\eeq
The iterated ``twisted'' finite differences of order $N$ with quasimomentum $k$ are defined as
\beq
\Delta_{g_1, \ldots, g_N;k}=\Delta_{g_1;k}\ldots \Delta_{g_N;k}, \quad \mbox{for} \quad g_1, \ldots, g_N\in G.
\eeq
\end{defi}
\begin{defi}
A function $u$ on $X$ is a \textbf{Floquet function of order $N$ with quasimomentum $k$} if any twisted finite difference of order $N+1$ with quasimomentum $k$ annihilates $u$.
Also, a \textbf{Bloch function with quasimomentum $k$} is a Floquet function of order $0$ with quasimomentum $k$.
\end{defi}
According to this definition, a Bloch function $u(x)\in L^2_{loc}(X)$ with quasimomentum $k$ is a $\gamma_k$-automorphic function on $X$, i.e., $u(g\cdot x)=e^{ik\cdot g}u(x)$ for any $g \in G$. If the quasimomentum is \textit{real} quasimomentum, then for any compact subset $K$ of $X$, the sequence $\{\|u\|_{L^2(gK)}\}_{g \in G}$ is bounded (i.e., belongs to $\ell^{\infty}(G)$).

It is also known \cite{KP2} that $u(x)$ is a Floquet function of order $N$ with quasimomentum $k$ if and only if $u$ can be represented in the form
$$u(x)=e_k(x)\left(\sum_{|j| \leq N}[x]^j p_j(x)\right),$$
where $j=(j_1, \ldots, j_d) \in \bZ^{d}_{+}$,
 and the functions $p_j$ are $G$-periodic. Here for any $j\in \bZ^d$, we define
\beq\label{E:mod_g}
|j|:=|j_1|+ \ldots |j_d|,
 \eeq
 while $e_k(x)$ and $[x]^j$ are analogs of the exponential $e^{ikx}$ and the monomial $x^j$ on $\mathbb{R}^d$ (see \cite{KP2} for details). For convenience, in Section \ref{S:auxfloquet}, we collect some basic facts of Floquet functions on abelian coverings. In the flat case $X=\mathbb{R}^d$, a Floquet function of order $N$ with quasimomentum $k$ is the product of the plane wave $e^{ikx}$ and a polynomial of degree $N$ with $G$-periodic coefficients.

An important consequence of this representation is that any Floquet function $u(x) \in L^2_{loc}(X)$ of order $N$ with a real quasimomentum satisfies the following $L^2$-growth estimate
$$
\|u\|_{L^2(gK)} \leq C(1+|g|)^N, \quad \forall g \in G \quad \mbox{and} \quad K \Subset X.
$$
Here $|g|$ is defined according to (\ref{E:mod_g}), and we have used the \v{S}varc-Milnor lemma from geometric group theory (see e.g., \cite[Lemma 2.8]{Luck}) to conclude that on a Riemannian co-compact covering $X$, the Riemannian distance between any compact subset $K$ and its $g$-translation $gK$ is comparable with $|g|$.

If $u$ is continuous, the above $L^2$-growth estimate can be replaced by the corresponding $L^{\infty}$-growth estimate.

\section[Liouville theorem]{Liouville theorem on abelian coverings}\label{S:Liouville}

We now need to introduce the spaces of polynomially growing solutions of the equation $Au=\lambda u$.
\begin{center}\textbf{To simplify the notations, we will assume from now on that $\lambda=0$, since, as we discussed before, we can deal with the operator $A-\lambda$ instead of $A$ (see also Definition \ref{blochfermi} d).}
\end{center}
\begin{defi}\label{D:Vpn}
\label{polyspace}
Let $K \Subset X$ be a compact domain such that $X$ is the union of all $G$-translations of $K$, i.e.,
\beq
\label{fundamental}
X=\bigcup_{g \in G} gK.
\eeq
For any $s, N \in \mathbb{R}$ and $1 \leq p \leq \infty$, we define the vector spaces
\beqn
V^p_{N}(X):=\left\{u \in C^{\infty}(X) \mid \{\|u\|_{L^2(gK)}\cdot \langle g \rangle^{-N}\}_{g \in G} \in \ell^p(G)\right\},
\eeqn
and
\beqn
V^{p}_{N}(A):=\left\{u \in V^p_{N}(X) \mid Au=0 \right\}.
\eeqn
\end{defi}
Here $\langle g \rangle:=(1+|g|^2)^{1/2}$.

It is not hard to show that these spaces are independent of the choice of the compact subset $K$ satisfying \mref{fundamental}. In particular, one can take as $K$ a fundamental domain for $G$-action on $X$.
Moreover, we have $V^{p_1}_{N_1}(X) \subseteq V^{p_2}_{N_2}(X)$ and $V^{p_1}_{N_1}(A) \subseteq V^{p_2}_{N_2}(A)$ whenever $N_1 \leq N_2$ and $p_1 \leq p_2$.
\begin{defi}
For $N \geq 0$, we say that the \textbf{Liouville theorem of order $(p,N)$ holds for A}, if the space $V^{p}_{N}(A)$ is finite dimensional.
\end{defi}
Now we can restate one of the main results in \cite{KP2} as follows
\begin{thm} \cite{KP2}
\label{Liouvillethm}
$ $

\begin{enumerate}[(i)]
\item
The following statements are equivalent:
\begin{enumerate}[(1)]
\item
The cardinality of the real Fermi surface $F_{A,\mathbb{R}}$ is finite modulo $G^*$-shifts, i.e., Bloch solutions exist for only finitely many unitary characters $\gamma_k$.

\item
The Liouville theorem of order $(\infty, N)$ holds for $A$ for some $N \geq 0$.

\item
The Liouville theorem of order $(\infty, N)$ holds for $A$ for all $N \geq 0$.
\end{enumerate}

\item
Suppose that the Liouville theorem holds for $A$. Then for any $N \in \mathbb{N}$, each solution $u \in V^{\infty}_N(A)$ can be represented as a finite sum of Floquet solutions:
\beq
\label{floquetrep}
u(x)=\sum_{k \in F_{A,\mathbb{R}}}\sum_{0\leq j \leq N}u_{k,j}(x),
\eeq
where each $u_{k,j}$ is a Floquet solution of order $j$ with a quasimomentum $k$.

\item A crude estimate of the dimension of $V^{\infty}_N(A)$:
$$\displaystyle \dim V^{\infty}_N(A) \leq \binom{d+N}{N}\cdot \sum_{k \in F_{A,\mathbb{R}}}\dim \Ker{A(k)}<\infty.$$
\end{enumerate}
\end{thm}
Due to the relation between $F_{A,\mathbb{R}}$ and $F_{A^*,\mathbb{R}}$ (see Proposition \ref{Panalytic}), the Liouville theorem holds for $A$ if and only if it also holds for $A^*$.

\section{Some properties of spaces $V^p_N(A)$}

\begin{notation}
For a real number $r$, we denote by $\floor{r}$ the largest integer that is \textbf{strictly} less than $r$, while $[r]$ denotes the largest integer that is less or equal than $r$.
\end{notation}
The following statement follows from Theorem \ref{Liouvillethm} (ii):
\begin{lemma}
$V^{\infty}_N(A)=V^{\infty}_{[N]}(A)$ for any non-negative real number $N$.
\end{lemma}
The proofs of the next two theorems are delegated to the section \ref{technical}.
\begin{thm}
\label{dimvp}
For each $1\leq p<\infty$ such that $pN>d$, one has
$$V^{\infty}_{\floor{N-d/p}}(A) \subseteq V^{p}_N(A).$$
If, additionally, the Fermi surface $F_{A, \mathbb{R}}$ is finite modulo $G^*$-shifts, then
$$V^{\infty}_{\floor{N-d/p}}(A)=V^{p}_N(A).$$
\end{thm}
\begin{corollary}
If the Fermi surface $F_{A, \mathbb{R}}$ is finite modulo $G^*$-shifts, then Liouville theorem of order $(p,N)$ with $pN>d$ holds for $A$ if and only if the Liouville theorem of order $(\infty,N)$ holds for $A$ for some $N \geq 0$, and thus, according to Theorem \ref{Liouvillethm}, for all $N\geq 0$.
\end{corollary}

The following theorem could be regarded as a version of the unique continuation property at infinity for the periodic elliptic operator $A$.
\begin{thm}
\label{UCinfty}
Assume that $F_{A,\mathbb{R}}$ is finite (modulo $G^*$-shifts). Then the space $V^p_{N}(A)$ is trivial if either one of the following conditions holds:
\begin{enumerate}[(a)]
\item
$p \neq \infty, pN \leq d$.
\item
$p=\infty, N<0$.
\end{enumerate}
\end{thm}

In fact, a more general version of these results holds:
\begin{thm}
\label{general}
Let $p \in [1,\infty)$. Let also $\phi$ be a continuous, positive function defined on $\mathbb{R}^+$ such that
\beqn
N_{p,\phi}:=\sup\left\{N \in \bZ : \int\limits_{0}^{\infty}\phi(r)^{-p}\cdot \langle r \rangle^{pN+d-1}\di{r}<\infty\right\}<\infty.
\eeqn
We define $\cV^{p}_{\phi}(A)$ as the space of all solutions $u$ of $A=0$ satisfying the condition
$$\sum\limits_{g \in \bZ^d}\|u\|_{L^2(gK)}^p\cdot \phi(|g|)^{-p}<\infty$$
holds for some compact domain $K$ satisfying \mref{fundamental}.

If $F_{A,\mathbb{R}}$ is finite (modulo $G^*$-shifts), then one has
\begin{itemize}
\item
If $N_{p, \phi} \geq 0$, then $\displaystyle \cV^{p}_{\phi}(A)=V^{\infty}_{N_{p,\phi}}(A)$.
\item
If $N_{p, \phi}<0$, then $\displaystyle \cV^{p}_{\phi}(A)=\{0\}$.
\end{itemize}
\end{thm}

Note that if $\phi(r)=\langle r \rangle^{N}$, then $\cV^{p}_{\phi}(A)=V^{p}_N(A)$.

The proofs of Theorems \ref{dimvp} and \ref{UCinfty} (provided in section \ref{technical}) easily transfer to this general version. We thus skip the proof of Theorem \ref{general} (never used later on in this text), leaving this as an exercise for the reader.

\begin{remark}
It is worthwhile to note that results of this section did not require the assumption of discreteness of spectra of the operators $A(k)$. This is useful, in particular, when considering overdetermined problems.
\end{remark}

\section[Formulas for dimensions]{Explicit formulas for dimensions of spaces $V^{\infty}_N(A)$}

In order to describe explicit formulas for the dimensions of $V^{\infty}_N(A)$, we need to introduce some notions from \cite{KP2}.
Recall that, for each quasimomentum $k$, $A(k)$ belongs to the space $\mathcal{L}(H^m_k(X), L^2_k(X))$ of bounded linear operators acting from $H^m_k(X)$ to $L^2_k(X)$.
For a real number $s$, the spaces $H^s_k(X)$ are the fibers of the following analytic $G^*$-periodic Hilbert vector bundle over $\mathbb{C}^d$:
\beq
\mathcal{E}^s:=\bigcup_{k\in \mathbb{C}^d} H^s_k(X)=\bigcup_{k\in \mathbb{C}^d} H^s(E_k).
\eeq
Consider a quasimomentum $k_0$ in $F_{A,\mathbb{R}}$. We can locally trivialize\footnote{In fact, $\mathcal{E}^s$ is globally analytically trivializable (see e.g., \cite{Kbook, ZKKP}) although we do not need this fact here.} the vector bundle $\mathcal{E}^s$, so that in a neighborhood of $k_0$, $A(k)$ becomes an analytic family of bounded operators from $H^s_{k_0}$ to $L^2_{k_0}$ (see Subsection \ref{a(k)}).
Suppose that the spectra of operators $A(k)$ are discrete for any value of the quasimomentum $k$.

Assume now that zero is an eigenvalue of the operator $A(k_0)$ with algebraic multiplicity $r$. Let $\Upsilon$ be a contour in $\mathbb{C}$ separating zero from the rest of the spectrum of $A(k_0)$. According to the perturbation theory (see Proposition \ref{semicont}), we can pick a small neighborhood of $k_0$ such that the contour $\Upsilon$ does not intersect with $\sigma(A(k))$ for any $k$ in this neighborhood.
We denote by $\Pi(k)$ the $r$-dimensional Riesz spectral projector \cite[Ch. III, Theorem 6.17]{Ka} for the operator $A(k)$, associated with the contour $\Upsilon$.
Now one can pick an orthonormal basis $\{e_j\}_{1 \leq j \leq r}$ in the range of $\Pi(k_0)$ and define $e_{j}(k):=\Pi(k)e_{j}$. Then let us consider the $r \times r$ matrix $\lambda(k)$ of the operator $A(k)\Pi(k)$ in the basis $\{e_{j}(k)\}$, i.e.,
\beq
\lambda_{ij}(k)=\langle A(k)e_j(k), e_i(k)\rangle=\langle A(k)\Pi(k)e_j, e_i\rangle.
\eeq
\begin{remark}\indent
\begin{itemize}
\item An important special case is when $r=1$ near $k_0$. Then $\lambda(k)$ is just the band function that vanishes at $k_0$.
\item
It will be sometimes useful to note that our considerations in this part will not change if we multiply $\lambda(k)$ by an invertible matrix function analytic in a neighborhood of $k_0$.
\end{itemize}
\end{remark}
Now, using the Taylor expansion around $k_0$, we decompose $\lambda(k)$ into the series of homogeneous matrix valued polynomials:
\beq\label{Taylor}
\lambda(k)=\sum_{j \geq 0} \lambda_{j}(k-k_0),
\eeq
where each $\lambda_j$ is a $\mathbb{C}^{r \times r}$-valued homogeneous polynomial of degree $j$ in $d$ variables.

For each quasimomentum $k_0 \in F_{A,\mathbb{R}}$, let $\ell_0(k_0)$ be the order of \textbf{the first non-zero term} of the Taylor expansion (\ref{Taylor}) around $k_0$ of the matrix function $\lambda(k)$.

The next result of \cite{KP2} provides explicit formulas for dimensions of the spaces $V^{\infty}_N(A)$.

In order to avoid misunderstanding the formulas below, we adopt the following agreement:
\begin{defi}\label{D:neg_binom}
If in some formulas throughout this text one encounters a binomial coefficient $\binom{A}{B}$, where $A<B$ (in particular, when $A<0$) we define its value to be equal to zero.
\end{defi}

\begin{thm}\cite{KP2}
\label{LiouvilleDim}
Suppose that the real Fermi surface $F_{A,\mathbb{R}}$ is finite (modulo $G^*$-shifts) and the spectrum $\sigma(A(k))$ is discrete for any quasimomentum $k$. Then,
\begin{enumerate}[a.]
\item
For each integer $0 \leq N<\min\limits_{k \in F_{A,\mathbb{R}}}\ell_0(k)$, we have
\beq
\label{dimliouville}
\dim V^{\infty}_N(A)=\sum_{k \in F_{A, \mathbb{R}}} m_k \left[\binom{d+N}{d}-\binom{d+N-\ell_0(k)}{d}\right],
\eeq
where $m_k$ is the algebraic multiplicity of the zero eigenvalue of the operator $A(k)$.
\item
If for every $k \in F_{A,\mathbb{R}}$, $\det\lambda_{\ell_0(k)}$ is not identically equal to zero, formula \mref{dimliouville} holds for any $N \geq 0$.
\end{enumerate}
\end{thm}

It is worthwhile to note that the positivity of $\ell_0(k)$ is equivalent to the fact that both algebraic and geometric multiplicities of the zero eigenvalue of the operator $A(k)$ are the same. Also, the non-vanishing of the determinant of $\lambda_{\ell_0(k)}$ implies that $\ell_0(k)>0$.
\section[Gromov-Shubin theory]{The Nadirashvili-Gromov-Shubin version of the Riemann-Roch theorem for elliptic operators on noncompact manifolds}
\label{GSRR}

It will be useful to follow closely the paper \cite{GSinv} by M.~Gromov and M.~Shubin, addressing its parts that are relevant for our considerations.

\subsection{Some notions and preliminaries}
\label{RRsetting}
Through this section, $P$ will denote a linear elliptic differential expression with smooth coefficients on a non-compact Riemannian manifold $\mathcal{X}$ (later on, $\mathcal{X}$ will be the space of an abelian co-compact covering). We denote by $P^*$ its transpose (also an elliptic differential operator), defined via the identity
\beqn
\langle Pu, v\rangle=\langle u, P^*v\rangle, \quad \forall u,v \in C^{\infty}_c(\mathcal{X}),
\eeqn
where $\langle \cdot, \cdot \rangle$ is the bilinear duality \mref{L2inner}.

We notice that both $P$ and $P^*$ can be applied as differential expressions to any smooth function on $\X$ and these operations keep the spaces $C^\infty(\X)$ and $C^\infty_c(\X)$ invariant.

We assume that $P$ and $P^*$ are defined as operators on some domains $\Dom{P}$ and $\Dom{P^*}$, such that
\begin{eqnarray}
\label{domainP}
C^{\infty}_c(\mathcal{X}) \subseteq \Dom{P} \subseteq C^{\infty}(\mathcal{X}),\\
C^{\infty}_c(\mathcal{X}) \subseteq \Dom{P^*} \subseteq C^{\infty}(\mathcal{X}).
\end{eqnarray}

\begin{defi}
We denote by $\Image{P}$ and $\Image{P^*}$ the ranges of $P$ and $P^*$ on their corresponding domains, i.e.
\beq
\Image{P}=P(\Dom{P}),\quad \Image{P^*}=P^*(\Dom{P^*}).
\eeq
As usual, $\Ker{P}$ and $\Ker{P^*}$ denote the spaces of solutions of the equations $Pu=0$, $P^*u=0$ in $\Dom{P}$ and $\Dom{P^*}$ respectively.
\end{defi}

We also need to define some auxiliary spaces\footnote{Most of the complications in definitions here and below come from non-compactness of the manifold.}.
Namely, assume that we can choose linear subspaces $\Domp{P}$ and $\Domp{P^*}$ of $C^\infty(\X)$ so that
\begin{enumerate}[($\mathcal{P}$1)]
\item
\begin{eqnarray}
C^{\infty}_c(\mathcal{X}) \subseteq \Domp{P} \subseteq C^{\infty}(\mathcal{X}),\\
C^{\infty}_c(\mathcal{X}) \subseteq \Domp{P^*} \subseteq C^{\infty}(\mathcal{X}),
\end{eqnarray}
and
\item
\beqn
\Image{P^*} \subseteq \Domp{P}, \quad \Image{P} \subseteq \Domp{P^*}.
\eeqn
\item
The bilinear pairing $\int_{\mathcal{X}} f(x)g(x) \di\mu_X$ (see (\ref{L2inner})) makes sense for functions from the relevant spaces, to define the pairings
\beqn
\langle \cdot, \cdot \rangle: \Domp{P^*} \times \Dom{P^*} \mapsto \mathbb{C}, \quad \langle \cdot, \cdot \rangle: \Dom{P} \times \Domp{P} \mapsto \mathbb{C},
\eeqn
so that
\item
The duality (``integration by parts formula'')
\beqn
\langle Pu, v \rangle=\langle u, P^*v \rangle, \quad \forall u \in \Dom{P}, \hspace{3pt} v \in \Dom{P^*}
\eeqn
holds.
\end{enumerate}

\begin{remark}
The notation $\Domp$ might confuse the reader, leading her to thinking that this is a different domain of the same  differential expression. It is rather an object \underline{\textbf{dual}} to the domain $\Dom$.
\end{remark}
We also need an appropriate notion of a polar (annihilator) to a subspace:
\begin{defi}
\label{annihilator}
For a subspace $L\subset\Dom{P}$, its \textbf{annihilator} $L^{\circ}$ is the subspace of $\Domp{P}$ consisting of all elements of  $\Domp{P}$ that are orthogonal to $L$ with respect to the pairing $\langle \cdot, \cdot \rangle$ :
\beqn
L^{\circ}=\{ u \in \Domp{P}\mid \langle v, u \rangle=0, \mbox{ for any } v \in L\}.
\eeqn
Analogously, $M^{\circ}$ is the annihilator in $\Domp{P^*}$ of a linear subspace $M \subset \Dom{P^*}$ with respect to $\langle \cdot, \cdot \rangle$.
\end{defi}

Following \cite{GSinv}, we now introduce an appropriate for our goals notion of Fredholm property.
\begin{defi}
\label{Fredholm}
The operator $P$, as above, is a \textbf{Fredholm operator on $\mathcal{X}$} if the following requirements are satisfied:
\begin{enumerate}[(i)]
\item
\beqn
\dim{\Ker{P}}<\infty, \dim{\Ker{P^*}}<\infty
\eeqn
and
\item
\beqn
\Image{P}=\left(\Ker{P^*}\right)^{\circ}.
\eeqn
\end{enumerate}
Then the \textbf{index} of $P$  is defined as
$$\Index{P}=\dim{\Ker{P}}-\codim{\Image{P}}=\dim{\Ker{P}}-\dim{\Ker{P^*}}.$$
\end{defi}

\subsection{Point divisors}
We will need to recall the rather technical notion of a \textbf{rigged divisor} from \cite{GSinv}. However, for reader's sake, we start with more familiar and easier to comprehend particular case of a \textbf{point divisor}, which appeared initially in Nadirashvili and Gromov-Shubin papers \cite{nadirashvili,GSadv}.

\begin{defi}\label{D:divisor_point}
A \textbf{point divisor} $\mu$ on $X$ consists of two finite disjoint subsets of $X$
\beq\label{E:Dpm_point}
D^+=\{x_1, \ldots, x_r\}, D^-=\{y_1, \ldots, y_s\}
\eeq
and two tuplets $0<p_1, \ldots, p_r$ and $q_1, \ldots, q_s<0$ of integers. The \textbf{support of the point divisor} $\mu$ is $D^+\bigcup D^-$. We will also write $$\mu:=x_1^{p_1}\ldots x_r^{p_r}\cdot y_1^{q_1}\ldots y_s^{q_s}.$$
\end{defi}

In other words, $\mu$ is an element of the free abelian group generated by points of $\mathcal{X}$.

In \cite{nadirashvili, GSadv}, such a divisor is used to allow solutions $u(x)$ of an elliptic equation $Pu=0$ of order $m$ on $n$-dimensional manifold $X$ to have poles up to certain orders at the points of $D^+$ and enforce zeros on $D^-$. Namely,
\begin{enumerate}[(i)]
\item
For any $1 \leq j \leq r$, there exists an open neighborhood $U_j$ of $x_j$ such that on $U_j \setminus \{x_j\}$, one has $u=u_s+u_r$, where $u_r \in C^{\infty}(U_j)$, $u_s \in C^{\infty}(U_j \setminus \{x_j\})$ and when $x \rightarrow x_j$,
\begin{equation*}
u_s(x)=o(|x-x_j|^{m-n-p_j}).
\end{equation*}
\item
For any $1 \leq j \leq s$, as $x \rightarrow y_j$, one has
\begin{equation*}
u(x)=O(|x-y_j|^{|q_j|}).
\end{equation*}
\end{enumerate}

\subsection{Rigged divisors}

The notion of a ``rigged'' divisor comes from the desire to allow for some infinite sets $D^\pm$, but at the same time to impose only finitely many conditions (``zeros'' and ``singularities'') on the solution.

So, let us take a deep breath and dive into it. First, let us define some distribution spaces:
\begin{defi}\label{D:distr_spaces}
For a closed set $C\subset\X$, we denote by $\mathcal{E}'_{C}(\mathcal{X})$ the space of distributions on $\mathcal{X}$, whose supports belong to $C$ (i.e., they are zero outside $C$).
\end{defi}

\begin{defi}
\label{rigged}
\indent
\begin{enumerate}
\item
A \textbf{rigged divisor associated with $P$} is a tuple $\mu=(D^+, L^+; D^-, L^-)$, where $D^{\pm}$ are \textit{compact nowhere dense disjoint subsets} in $\mathcal{X}$ and $L^{\pm}$ are \textit{finite-dimensional} vector spaces of distributions on $X$ supported in $D^{\pm}$ respectively, i.e.,
\beqn
L^+ \subset \mathcal{E}'_{D^+}(\mathcal{X}), \quad L^- \subset \mathcal{E}'_{D^-}(\mathcal{X}).
\eeqn
\item The \textbf{secondary spaces $\tilde{L}^{\pm}$} associated with $L^{\pm}$ are defined as follows:
\beqn
\tilde{L}^{+}=\{u \mid u \in \mathcal{E}'_{D^+}(\mathcal{X}), Pu \in L^+\}, \quad \tilde{L}^{-}=\{u \mid u \in \mathcal{E}'_{D^-}(\mathcal{X}), P^{*}u \in L^-\}.
\eeqn
\item Let $\ell^{\pm}=\dim L^{\pm}$ and $\tilde{\ell}^{\pm}=\dim \tilde{L}^{\pm}$. The \textbf{degree} of $\mu$ is defined as follows:
\beq
\deg_P{\mu}=(\ell^+-\tilde{\ell}^+)-(\ell^--\tilde{\ell}^-).
\eeq

\item The \textbf{inverse of $\mu$} is the rigged divisor $\mu^{-1}:=(D^-, L^-;D^+,L^+)$ associated with $P^*$.
\end{enumerate}
\end{defi}
\begin{remark} \indent
   \begin{itemize}
   \item Notice that the degree of the divisor involves the operator $P$, so it would have been more prudent to call it ``degree of the divisor with respect to the operator $P$,'' but we'll neglect this, hoping that no confusion will arise.
   \item Observe that, due to their ellipticity, $P$ and $P^*$ are injective on $\mathcal{E}'_{D^+}$ and $\mathcal{E}'_{D^-}$, correspondingly  \footnote{For example, if $u \in \mathcal{E}'_{D^+}$ and $Pu=0$ then $u$ is smooth due to elliptic regularity, but then $u=0$ everywhere since the complement of $D^+$ is dense.}. Thus,
       \beq\label{E:ellineq}
       \ell^{\pm} \geq \tilde{\ell}^{\pm}.
       \eeq
   \item The sum of the degrees of a divisor $\mu$ and of its inverse is zero.
   \end{itemize}
\end{remark}

Although we have claimed that point divisors are also rigged divisors, this is not immediately clear when comparing the definitions \ref{D:divisor_point} and \ref{rigged}. Namely, we have to assign the spaces $L_\pm$ to a point divisor and to check that the definitions are equivalent in this case. This was done\footnote{Which is not trivial.} in \cite{GSinv}, if one defines the spaces associated with a point divisor as follows:
$$
L^+=\left\{\sum\limits_{1 \leq j \leq r}\sum\limits_{|\alpha| \leq p_j-1}c^{\alpha}_j \delta^{\alpha}(\cdot-x_j)\mid c^{\alpha}_j \in \mathbb{C}\right\}$$
and
$$
L^-=\left\{\sum\limits_{1 \leq j \leq s}\sum\limits_{|\alpha| \leq |q_s|-1}c^{\alpha}_j \delta^{\alpha}(\cdot-y_j) \mid c^{\alpha}_j \in \mathbb{C}\right\},
$$
where $\delta$ and $\delta^{\alpha}$ denote the Dirac delta function and its derivative corresponding to the multi-index $\alpha$.

It was also shown in \cite{GSinv}) that the degree $\deg_{P}(\mu)$ in this case is
\begin{equation}\label{Analog}
\sum\limits_{1 \leq j \leq r}\left[\binom{p_j+n-1}{n}-\binom{p_j+n-1-m}{n}\right]-\sum\limits_{1 \leq j \leq s}\left[\binom{q_j+n-1}{n}-\binom{q_j+n-1-m}{n}\right].
\end{equation}
Here, as before, $n$ is the dimension of the manifold $X$ and $m$ is the order of the operator $P$.

\begin{remark}
One observes a clear similarity between the combinatorial expressions in (\ref{Analog}) and (\ref{dimliouville}). It was one of the reasons to try to combine Liouville and Riemann-Roch type results.
\end{remark}

\subsection{Nadirashvili-Gromov-Shubin theorem on noncompact manifolds}

To state (a version of) the Gromov-Shubin theorem, we now introduce the spaces of solutions of $P$ with allowed singularities on $D^+$ and vanishing conditions on $D^-$.
\begin{notation}
For a compact subset $K$ of $\mathcal{X}$ and $u\in C^{\infty}(\mathcal{X} \setminus K)$, we shall write that
$$
u\in\Dom_K{P},
$$
if there is a compact neighborhood $\hat{K}$ of $K$ and $\hat{u} \in \Dom{P}$ such that $u=\hat{u}$ outside $\hat{K}$.
\end{notation}
\begin{defi}\label{D:LmuP}
For an elliptic operator $P$ and a rigged divisor $$\mu=(D^+, L^+; D^-, L^-),$$ \textbf{the space $L(\mu, P)$} is defined as follows: $u\in L(\mu,P)$ iff $u \in \Dom_{D^+}{P}$ and there exists $\tilde{u} \in \mathcal{D}'(\mathcal{X})$, such that $\tilde{u}=u$ on $\mathcal{X} \setminus D^+$, $P\tilde{u} \in L^+$, and $(u, L^-)=0$.

Here $(u, L^-)=0$ means that $u$ is orthogonal to every element in $L^-$ with respect to the canonical bilinear duality.
\end{defi}
\begin{remark}One notices that distributions $\tilde{u}$ are regularization of $u \in C^{\infty}(\mathcal{X} \setminus D^+)$.\end{remark}

In other words, the space $L(\mu,P)$ consists of solutions of the equation $Pu=0$ with poles allowed and zeros enforced by the divisor $\mu$. It is worthwhile to notice that since the manifold is non-compact, the domain $\Dom P$ of the operator $P$ will have to involve some conditions at infinity. This observation will be used later to treat a ``pole'' at infinity, i.e. Liouville property.

Now we can state a variant of Nadirashvili-Gromov-Shubin's version of the Riemann-Roch theorem.
\begin{thm}
\label{RR}
Let $P$ be an elliptic operator such that \mref{domainP} and properties \textit{($\mathcal{P}$1)-($\mathcal{P}$4)} are satisfied. Let also $\mu$ be a rigged divisor associated with $P$. If $P$ is a Fredholm operator on $\mathcal{X}$, then the following Riemann-Roch inequality holds:
\beq
\label{RRineq}
\dim{L(\mu,P)}-\dim{L(\mu^{-1},P^*}) \geq \Index{P}+\deg_{P}(\mu).
\eeq
If both $P$ and $P^*$ are Fredholm on $\mathcal{X}$, \mref{RRineq} becomes the Riemann-Roch equality:
\beq
\label{RReq}
\dim{L(\mu,P)}-\dim{L(\mu^{-1},P^*})=\Index{P}+\deg_{P}(\mu).
\eeq
\end{thm}
\begin{remark}\indent\begin{enumerate}
\item Although the authors of \cite{GSinv} do not state their theorem in the exact form above, the Riemann-Roch inequality \mref{RRineq} follows from their proof.
\item If one considers the difference $\dim{L(\mu,P)}-\dim{L(\mu^{-1},P^*})$ as some ``index of $P$ in presence of the divisor $\mu$'' (say, denote it by $\Index_\mu(P)$), the Riemann-Roch equality (\ref{RReq}) becomes
   \beq
\label{RReq2}
\Index_\mu(P)=\Index{P}+\deg_{P}(\mu)
\eeq
and thus it says that introduction of the divisor changes the index of the operator by $\deg_{P}(\mu)$.

Analogously, the inequality (\ref{RRineq}) becomes

   \beq
\label{RRineq2}
\Index_\mu(P)\geq \Index{P}+\deg_{P}(\mu).
\eeq
\end{enumerate}
\end{remark}


It is useful for our future considerations to mention briefly some of the ingredients\footnote{The reader interested in the main results only, can skip to Corollary \ref{RReqmu}.} of the proof from \cite{GSinv}.
To start, we define some auxiliary spaces. Let, as before, $K$ be a nowhere dense compact set and we denote for a function $u \in C^{\infty}(X \setminus K)$ by $\tilde{u} \in \mathcal{D}'(\mathcal{X})$ its (non-uniquely defined) regularization. I.e., $u=\tilde{u}$ on $\mathcal{X}\setminus K$. Let
\beqn
\begin{split}
\Gamma(\mathcal{X},\mu,P):=&\{u \in C^{\infty}(X \setminus D^+)  \mid  u \in \Dom_{D^+}{P}, \hspace{2pt} \exists \tilde{u} \in \mathcal{D}'(\mathcal{X}) \mbox{ such that }\\ &\tilde{u}=u \mbox{ on } \mathcal{X} \setminus D^+,
P\tilde{u} \in L^+ + C^{\infty}(\mathcal{X}) \mbox{ and } \langle u, L^-\rangle=0\}.
\end{split}
\eeqn
The regularization $\tilde{u}$ above is not unique, so we define the space of all such regularizations:
\beqn
\begin{split}
\tilde{\Gamma}(\mathcal{X},\mu,P):=\{\tilde{u} \in \mathcal{D}'(\mathcal{X}) \mid \tilde{u}_{|\mathcal{X} \setminus D^+} \in \Gamma(\mathcal{X}, \mu,P),
P\tilde{u} \in L^+ + C^{\infty}(\mathcal{X})
\}.
\end{split}
\eeqn
It follows from the definition of the space $\Gamma(\mathcal{X},\mu,P)$ that for any function $u$ in this space, the function $Pu$ (where $P$
is applied as a differential expression) extends uniquely to a smooth function, which we call $\tilde{P}u$, on the whole $\mathcal{X}$. In the same manner, we can also define the corresponding extension $\tilde{P^*}$ as a linear map from $\Gamma(\mathcal{X}, \mu^{-1}, P^*)$ to $C^\infty (\mathcal{X})$.

Let us also introduce the spaces of functions ``with enforced zeros'':
\beqn
\Gamma_{\mu}(\mathcal{X},P)=\{u \in \Dom{P} \mid \langle u,L^-\rangle=0\}
\eeqn
and
\beqn
\tilde{\Gamma}_{\mu}(\mathcal{X},P)=\{f \in \Domp{P^*} \mid  \langle f,\tilde{L}^-\rangle=0\}.
\eeqn
An inspection of these definitions leads to the following conclusions:
\begin{proposition}\label{P:RRtilde}
\indent
\begin{enumerate}
\item $\tilde{P}$ is a linear map from $\Gamma(\mathcal{X}, \mu, P)$ to $\tilde{\Gamma}_{\mu}(\mathcal{X}, P)$.
\item $\tilde{P^*}$ is a linear map from $\Gamma(\mathcal{X}, \mu^{-1}, P^*)$ to $\tilde{\Gamma}_{\mu^{-1}}(\mathcal{X}, P^*)$.
\item The spaces of solutions of interest are the kernels of the operators above:
\beqn\label{E:kernels}
L(\mu, P) = \Ker (\tilde{P}),\quad L(\mu^{-1}, P^*)=\Ker(\tilde{P^*}).
\eeqn
\end{enumerate}
\end{proposition}

Let us also introduce the duality
\beq
\label{duality1}
(\cdot, \cdot): \Gamma(\mathcal{X}, \mu, P) \times \tilde{\Gamma}_{\mu^{-1}}(\mathcal{X}, P^*) \rightarrow \mathbb{C}
\eeq
as follows:
\beqn
(u,f):=\langle \tilde{u},f\rangle, \quad u \in \Gamma(\mathcal{X}, \mu, P), f \in \tilde{\Gamma}_{\mu^{-1}}(\mathcal{X},P^*),
\eeqn
where $\tilde{u}$ is any element in the preimage of $\{u\}$ under the restriction map from $\tilde{\Gamma}(\mathcal{X},\mu,P)$ to $\Gamma(\mathcal{X}, \mu, P)$.
Similarly, we get the duality
\beq
\label{duality2}
(\cdot, \cdot): \tilde{\Gamma}_{\mu}(\mathcal{X}, P) \times \Gamma(\mathcal{X}, \mu^{-1}, P^*) \rightarrow \mathbb{C}
\eeq
These dualities are well-defined and non-degenerate \cite{GSinv}. Moreover, the following relation holds:
\beqn
(\tilde{P}u, v)=(u, \tilde{P^*}v), \quad \forall u\in \Gamma(\mathcal{X},\mu,P), v \in \Gamma(\mathcal{X}, \mu^{-1}, P^*).
\eeqn

by applying the additivity of Fredholm indices to some short exact sequences of the spaces introduced above (see \cite[Lemma 3.1 and Remark 3.2]{GSinv}), Gromov and Shubin then establish the following basic facts:
\begin{proposition}\cite[Lemma 3.1 -- Lemma 3.4]{GSinv})\indent\begin{enumerate}
\item \beq
\label{RRequality}
\dim{\Ker{\tilde{P}}}=\Index{P}+\deg_{P}(\mu)+\codim{\Image{\tilde{P^*}}},
\eeq
Note that the assumption that $P$ is Fredholm on $\mathcal{X}$ is important for (\ref{RRequality}) to hold true.
\item $(\Image{\tilde{P}})^{\circ}=\Ker{\tilde{P^*}}$,
\item $\Image{\tilde{P}} \subset (\Ker{\tilde{P^*}})^{\circ}$.
\item
\beq
\label{codim}
\dim{\Ker{\tilde{P^*}}}=\codim{(\Ker{\tilde{P^*}})^{\circ}} \leq \codim{\Image{\tilde{P}}}.
\eeq
\item The Riemann-Roch inequality \mref{RRineq} follows from \mref{RRequality} and \mref{codim}.
\item If $P^*$ is also Fredholm, one can apply \mref{RRineq} for $P^*$ and $\mu^{-1}$ instead of $P$ and $\mu$ to get
\beq
\label{RRinequality2}
\dim{\Ker{\tilde{P}}} \leq \Index{P}+\deg_{P}(\mu)+\dim{\Ker{\tilde{P^*}}}.
\eeq
Now, the Riemann-Roch equality \mref{RReq} follows from \mref{RRineq} and \mref{RRinequality2}.

In this case, as a byproduct of the proof of \mref{RReq} \cite[Theorem 2.12]{GSinv}, one also gets
$\Image{\tilde{P}}=(\Ker{\tilde{P^*}})^{\circ}$ and $\Image{\tilde{P}}=(\Ker{\tilde{P^*}})^{\circ}$.
\item
\end{enumerate}
Here $(\Image{\tilde{P}})^{\circ}$ and $(\Ker{\tilde{P^*}})^{\circ}$ are the annihilators of $\Image{\tilde{P}}$, $\Ker{\tilde{P^*}}$ with respect to the dualities \mref{duality1} and \mref{duality2}, respectively.
\end{proposition}

\begin{remark}
\label{RRfredholm}
If \mref{codim} becomes an equality, i.e.,
$$\dim{\Ker{\tilde{P^*}}}=\codim{\Image{\tilde{P}}},$$
one obtains the Riemann-Roch equality \mref{RReq} for the rigged divisor $\mu$ without assuming that $P^*$ is Fredholm on $\mathcal{X}$. Conversely, if \mref{RReq} holds, then $\Image{\tilde{P}}=(\Ker{\tilde{P^*}})^{\circ}$.
\end{remark}
As a result, we have the following useful corollary:
\begin{cor}
\label{RReqmu}
Let $P$ be Fredholm on $\mathcal{X}$, $\Image{P}=\Domp{P^*}$, and
$$\mu=(D^+, L^+; D^-, L^-)$$
be a rigged divisor on $\mathcal{X}$. Then the Riemann-Roch equality \mref{RReq} holds for $P$ and divisor $\mu$.

Moreover, the space $L(\mu^{-1}, P^*)$ is trivial, if the following additional condition is satisfied:
If $u$ is a smooth function in $\Dom{P}$ such that $\langle Pu, \tilde{L}^- \rangle=0$, then there exists a solution $v$ in $\Dom{P}$ of the equation $Pv=0$ satisfying $\langle u-v,L^- \rangle=0$.

In particular, this assumption holds automatically if
$D^-=\emptyset$.
\end{cor}

We end this section by recalling an application of Theorem \ref{RR}, which will be used later.
\begin{example}
\label{mainex}
(\cite[Example 4.6]{GSinv} and \cite{nadirashvili}) \emph{
Consider $P=P^*=-\Delta$ on $\mathcal{X}=\mathbb{R}^d$, where $d\geq 3$ and
$$\Dom{P}=\Dom{P^*}=\{u \mid u \in C^{\infty}(\mathbb{R}^d), \Delta u \in C^{\infty}_c (\mathbb{R}^d) \hspace{3pt}\mbox{and} \hspace{3pt} \lim_{|x| \rightarrow \infty}u(x)=0\},$$
$$\Domp{P}=\Domp{P^*}=C^{\infty}_c (\mathbb{R}^d).$$
Then the operators $P$ and $P^*$ are Fredholm on $\mathbb{R}^d$, $\Ker{P}=\Ker{P^*}=\{0\}$, $\Image{P}=\Image{P^*}=C^{\infty}_c(\mathbb{R}^d)$, and thus $\Index{P}=0$ (see \cite[Example 4.2]{GSinv}).}

\emph{Let
$$D^+=\{y_1, \ldots, y_k\}, \quad D^-=\{z_1, \ldots, z_l\}.$$
with all the points $y_1, \ldots, y_k, z_1, \ldots, z_l$ pairwise distinct. Consider the following distributional spaces: $L^+$ is the vector space spanned by Dirac delta distributions $\delta(\cdot-y_j)$ supported at the points $y_j$ $(1 \leq j \leq k)$; $L^-$ is spanned by the first order derivatives $\displaystyle \frac{\partial}{\partial x_{\alpha}}\delta(\cdot-z_j)$ of Dirac delta distributions supported at $z_j$ $(1 \leq j \leq l, 1 \leq \alpha \leq d)$. \footnote{Note that the secondary spaces $\tilde{L}^{\pm}$ are trivial.}}

\emph{Consider now the rigged divisor $\mu:=(D^+, L^+; D^-, L^-)$. Then $\deg_{-\Delta}(\mu)=k-dl$.
Furthermore,
$$L(\mu, -\Delta)=\left\{u \mid u(x)=\sum_{j=1}^k \frac{a_j}{|x-y_j|^{d-2}}, a_j \in \mathbb{C}, \hspace{3pt} \mbox{and} \hspace{3pt} \nabla u(z_j)=0, j=1,\ldots, l.\right\},$$
and
$$L(\mu^{-1}, -\Delta)=\left\{v \mid v(x)=\sum_{j=1}^l \sum_{\alpha=1}^d b_{j, \alpha}\frac{\partial}{\partial x_{\alpha}}\left(|x-z_j|^{2-d}\right), b_{j, \alpha} \in \mathbb{C}, \hspace{3pt} \mbox{and} \hspace{3pt} u(y_j)=0, j=1,\ldots, k.\right\}.$$
In this case, the Nadirashvili-Gromov-Shubin Riemann-Roch-type formula (Theorem \ref{RR}) is}
\beq
\label{RRexample}
\dim L(\mu,-\Delta)-\dim L(\mu^{-1},-\Delta)=k-dl.
\eeq
\end{example}

\chapter{The main results}
\label{C:main}

In this chapter, we consider a periodic elliptic operator $A$ of order $m$ on an $n$-dimensional co-compact abelian covering
$$
X \mathop{\mapsto}_{\Z^d} M.
$$
Notice, again, that the rank $d$ of the deck group does not have to be related in any way to the dimension $n$ of the manifold.

To make a combination of Liouville and Riemann-Roch theorems meaningful, we assume that the Liouville property holds for the operator $A$ at the level $\lambda=0$, i.e. (see Theorem \ref{Liouvillethm}), the \textbf{real} Fermi surface of $A$ (see Definition \ref{blochfermi}) is finite (modulo $G^*$-shifts).

The approach we will follow consists in finding appropriate functional spaces that would incorporate the polynomial growth at infinity and that could be handled by the general techniques and results of Gromov and Shubin.

There are two significantly different possibilities: 1. The Fermi surface finite, but non-empty. 2. The Fermi surface is empty. We start with the more interesting first one.

\section{Non-empty Fermi surface}
Suppose that $F_{A,\mathbb{R}}=\{k_1, \ldots, k_{\ell}\}$ (modulo $G^*$-shifts), where $\ell \geq 1$.

\subsection{Assumptions}
We need to make the following assumption on the local behavior of the Bloch variety of the operator $A$ around each quasimomentum $k_j$ in the real Fermi surface:

\textit{Assumption $\mathcal{A}$}
\begin{enumerate}[({$\mathcal{A}$}1)]
\item
\emph{For any quasimomentum $k$, the spectrum of the operator $A(k)$ is discrete.}

Under this assumption, the following lemma can be deduced immediately from Proposition \ref{semicont} and perturbation theory (see e.g., \cite{RS4, Ka}):

\begin{lemma}
For each quasimomentum $k_r\in F_{A,\mathbb{R}}$, there is an open neighborhood $V_r$ of $k_r$ in $\mathbb{R}^d$ and a closed contour $\Upsilon_r\subset V_r$, such that
 \begin{enumerate}
 \item The neighborhoods $V_r$ are mutually disjoint;
 \item The contour $\Upsilon_r$ surrounds the eigenvalue $0$ and does not contain any other points of the spectrum $\sigma(A(k_r))$;
 \item The intersection $\sigma{(A(k))} \cap \Upsilon_r$ is empty for any $k \in V_r$.
 \end{enumerate}
 \end{lemma}

Then, for any $k \in V_r$, we can define the Riesz projector
\beqn\label{E:riesz}
\Pi_r(k):=\frac{1}{2\pi i}\oint_{\Upsilon_r}(A(k)-zI)^{-1}dz
\eeqn
associated with $A(k)$ and the contour $\Upsilon_r$. Thus, $\Pi_r(k)A(k)$ is well-defined for any $k \in V_{r}$.
Let $m_r$ be the algebraic multiplicity of the eigenvalue $0$ of the operator $A(k_r)$. The immediate consequence is:

\begin{lemma}
The projector $\Pi_r(k)$ depends analytically on $k \in V_r$. In particular,
its range $R(\Pi_r(k))$ has the same dimension $m_r$ for all $k \in V_r$ and the union $\bigcup_{k \in V_r} R(\Pi_r(k))$ forms a trivial holomorphic vector bundle over $V_r$.
\end{lemma}
We denote by $A_r(k)$ the matrix representation of the operator $\Pi_r(k)A(k)|_{R(\Pi_r(k))}$ with respect to a fixed holomorphic basis $(f_j(k))_{1 \leq j \leq m_r}$ of the range $R(\Pi_r(k))$ when $k \in V_r$. Then $A_r(k)$ is an invertible matrix except only for $k=k_r$. We equip $\mathbb{C}^{m_r}$ with the maximum norm and impose the following integrability condition:
\item
\beqn
\int\limits_{V_r \setminus \{k_r\}} \|A_r(k)^{-1}\|_{\mathcal{L}(\mathbb{C}^{m_r})}\di{k}<\infty, \mbox{ for all } r=1,\dots, \ell,
\eeqn
where $\mathcal{L}(\mathbb{C}^{m_r})$ is the algebra of linear operators on $\mathbb{C}^{m_r}$.
\end{enumerate}
\begin{remark}
\label{assumptionA}
\indent
\begin{enumerate}[(i)]
\item
Thanks to Proposition \ref{realsymbol}, Assumption ($\mathcal{A}1$) is satisfied if $A$ is either self-adjoint or a real operator of even order.\footnote{Here $A$ is real means that $Au$ is real whenever $u$ is real.}
\item
When the rank $d$ of $G$ is greater than $2$, Assumption ($\mathcal{A}2$) holds at a \textbf{generic} spectral edge (see Chapter \ref{C:applications-LRR}).
\end{enumerate}
\end{remark}

\subsection{Spaces}
To formulate the results and to be able to use the conclusions of Gromov and Shubin, it is crucial to define spaces of solutions of the equation $Au=0$ that combine polynomial growth at infinity with satisfying the conditions imposed by a rigged divisor $\mu$.
\begin{defi}
\label{polyrigged}
Given any $p \in [1,\infty]$ and $N \in \mathbb{R}$, we define
\beqn
L_p(\mu, A, N):=L(\mu, A^p_{N}),
\eeqn
where the operator $A^p_{N}$ stands for $A$ with the domain
\beqn
\Dom{A^p_{N}}=\{ u \in V^p_N(X) \mid Au \in C^{\infty}_c(X)\}.
\eeqn
In other words, $L_p(\mu, A, N)$ is the space
\beqn
\{u \in \Dom_{D^+}{A^p_{N}} \mid \exists \tilde{u} \in \mathcal{D}'(X): \hspace{2pt} \tilde{u}=u \hspace{4pt}\mbox{on} \hspace{4pt} X \setminus D^+, A\tilde{u} \in L^+, (u, L^-)=0\}.
\eeqn
\end{defi}
We thus restrict the growth (in $L_p$-sense) of function $u$ at infinity to a polynomial of order $N$, impose the divisor $\mu$ conditions, and require that $u$ satisfies the homogeneous equation $Au=0$ outside the compact $D^+$.
\begin{remark}
\label{domApn}
Consider $u \in L_p(\mu, A, N)$. Let $K$ be a compact domain in $X$ such that $\bigcup_{g \in G}gK=X$. Define $G_{K,D^+}:=\{g \in G \mid \dist{(gK, D^+)}\geq 1\}$, where we use the notation $\dist{(\cdot, \cdot)}$ for the distance between subsets arising from the Riemannian distance on $X$. Since $Au=0$ on $X \setminus D^+$, the condition ``$u \in \Dom_{D^+}{A^p_{N}}$'' can be written equivalently as follows:
$$\{\|u\|_{L^2(gK)}\cdot\langle g \rangle^{-N}\}_{g \in G_{K,D^+}}\in \ell^p(G_{K,D^+}).$$
By Schauder estimates (see Proposition \ref{schauderest}), this condition can be rephrased as follows:
\beq
\label{growth-LRR}
\begin{split}
&\sup\limits_{x: \hspace{2pt}\dist(x, D^+) \geq 1}\frac{|u(x)|}{\dist(x, D^+)^{N}}<\infty, \hspace{57pt} \mbox{when} \hspace{15pt} p=\infty,\\
&\int\limits_{x: \hspace{2pt}\dist(x, D^+) \geq 1}\frac{|u(x)|^p}{\dist(x, D^+)^{pN}}\di \mu_X(x)<\infty, \quad \mbox{when} \hspace{15pt} 1 \leq p<\infty.
\end{split}
\eeq
So, depending on the sign of $N$, this condition controls how $u$ grows or decays at infinity.
\end{remark}

\subsection{Results}
Our first main result is contained in the next theorem, establishing a Liouville-Riemann-Roch type inequality.
\begin{thm}
\label{RRLineq}
Assume that either $p=\infty$ and $N \geq 0$ or $p \in [1,\infty)$ and $N>d/p$.
Let $p'$ be the H\"older conjugate of $p$. Then, under Assumption $\mathcal{A}$ imposed on the operator $A$, the following Liouville-Riemann-Roch inequality holds:
\begin{equation}
\label{RRLinequality}
\dim L_p(\mu,A,N)-\dim L_{p'}(\mu^{-1},A^*,-N) \geq \dim{V_{N}^p(A)}+\deg_A(\mu),
\end{equation}
where $\dim{V_{N}^p(A)}$ can be computed via Theorems \ref{LiouvilleDim} and \ref{dimvp}.
\end{thm}
\begin{remark}
\indent
\begin{itemize}
\item
This is an extension of the Riemann-Roch inequality \mref{RRineq} to include also Liouville conditions on growth at infinity.
\item
One may wonder why in comparison with the ``$\mu$-index'' $\displaystyle \dim L_p(\mu,A,N)-\dim L_p(\mu^{-1},A^*,-N)$, the above inequality only involves the \textbf{dimension of the kernel} $V^p_N(A)$, rather than a full index. The reason is that in this case, the two coincide, the co-dimension of the range being equal to zero.
\item
Assumption ($\mathcal{A}1$) forces the Fredholm index of $A(0)$ on $M$ to vanish
(see \cite[Theorem 4.1.4]{Kbook}). Therefore (by Atiyah's theorem \cite{Atiyah}), $\Index_M{A(k)}$ ($k \in \mathbb{C}^d$) and the $L^2$-index of $A$ are equal to zero as well.
\end{itemize}
\end{remark}
As an useful direct corollary of Theorem \ref{RRLineq}, we obtain
\begin{thm}
\label{existence}
If $\dim{V_{N}^p(A)}+\deg_A(\mu)>0$, then there exists a nonzero element in the space $L_p(\mu,A,N)$.

In other words, there exists a nontrivial solution of $Au=0$ of the growth described in \mref{growth-LRR} that also satisfies the conditions on ``zeros'' and ``singularities'' imposed by the rigged divisor $\mu$.
\end{thm}
To state the next results, we will need the following definition that addresses divisors containing only ``zeros'' or ``singularities''.
\begin{defi}
\indent
\begin{itemize}
\item
The divisor $(\emptyset, \{0\}; \emptyset, \{0\})$ is called \textbf{trivial}.
\item
Let $\mu=(D^+, L^+; D^-, L^-)$ be a rigged divisor on $X$. Then its \textbf{positive and negative parts} are defined as follows:
 $$\mu^+:=(D^+, L^+; \emptyset, \{0\}), \quad\mu^-:=(\emptyset, \{0\}; D^-, L^-).$$
 E.g., for a rigged divisor $\mu$, $\mu^+$ (resp. $\mu^-$) is trivial whenever $D^+=\emptyset$ (respectively, $D^-=\emptyset$).
\end{itemize}
\end{defi}
Our next result shows that if $\mu$ is positive (i.e., $\mu^-$ is trivial), the Liouville-Riemann-Roch inequality becomes an equality:
\begin{thm}
\label{RRLeq}
Let $\mu$ be a rigged divisor and $\mu^+$ - its positive part. Under the assumptions of Theorem \ref{RRLineq},
\begin{enumerate}
\item The space $L_{p'}((\mu^{+})^{-1},A^*,-N)$ is trivial (i.e., contains only zero).
\item
\begin{equation}
\label{RRLeqmu}
\dim L_p(\mu^+,A,N)=\dim V_{N}^p(A)+\deg_A(\mu^+).
\end{equation}
\item
In particular,
\beq
\label{RRLineqmu}
\dim L_p(\mu,A,N) \leq \dim V_{N}^p(A)+\deg_A(\mu^+).
\eeq
\end{enumerate}
\end{thm}
In other words, the inequality \mref{RRLineqmu} gives an upper bound for the dimension of the space $L_p(\mu,A,N)$ (see \mref{RRLinequality} for its lower bound) in terms of the degree of the positive part of the divisor $\mu$ and the dimension of the space $V^p_N(A)=L_p(\mu_0,A,N)$, where $\mu_0$ is the trivial divisor. Then \mref{RRLeqmu} shows that this estimate saturates for divisors containing only ``singularities''.

When $\mu^-$ is nontrivial, determining the triviality of the space $L_{p'}(\mu^{-1},A^*,-N)$ is more complicated. In the next proposition we show that if the degree of $\mu^+$ is sufficiently large, the space $L_{p'}(\mu^{-1},A^*,-N)$ degenerates to zero, while the spaces $L_{p'}(\mu^{-1},A^*,-N)$ can have arbitrarily large dimensions if $\mu^+$ is trivial.

We need to recall the following definition (e.g., \cite{mizohata}):
\begin{defi}
 A differential operator $A$ is said to have the \textbf{strong unique continuation property} if any local solution of the equation $Au=0$ that vanishes to the infinite order at a point, vanishes identically.
\end{defi}

We can now state the promised proposition:
\begin{prop}
\label{triviality}
\indent
\begin{enumerate}[(a)]
\item
For any $N \geq 0$, $p \in [1, \infty]$, and $d \geq 3$,
$$
\sup_{\mu=\mu^-}\dim L_{p'}(\mu^{-1}, -\Delta_{\mathbb{R}^d},-N)=\infty,
$$
where the supremum is taken over all divisors $\mu$ with trivial positive parts.
\item
Let the assumptions of Theorem \ref{RRLineq} be satisfied and $A^*$ have the strong unique continuation property. Let also the covering $X$ be connected and a point $x_0$ of $X$ selected. Suppose that $p \in [1,\infty]$, $N \in \mathbb{R}$, a compact nowhere dense set $D^-\subset X\setminus \{x_0\}$, and a finite dimensional subspace $L^-$ of $\mathcal{E}'_{D^-}(X)$ are fixed. Then there exists $M>0$ depending on $(A, p,N,x_0, D^-, L^-)$, such that
$$
\dim L_{p'}(\mu^{-1},A^*,-N)=0
$$
for any rigged divisor $\mu=(D^+, L^+; D^-, L^-)$ such that
$x_0 \in D^+ \subseteq X \setminus D^-$ and $L^+$ contains the linear span
$$\mathcal{L}^+_{M}:=\Span_{\mathbb{C}} \{\partial^{\alpha}_x \delta(\cdot-x_0)\}_{0 \leq |\alpha| \leq M}.$$
\end{enumerate}
\end{prop}
The second part of the above proposition is a reformulation of \cite[Proposition 4.3]{GSadv}.

Theorem \ref{RRLineq} can be improved under
the following \textbf{strengthened version of} \textit{Assumption $\mathcal{A}$}:
\beq \label{E:strength}
\mbox{ For each } 1 \leq r \leq \ell,\mbox{ the function }
k \in V_r \mapsto \|A_r(k)^{-1}\|^2_{\mathcal{L}(\mathbb{C}^{m_r})}
\mbox{ is integrable.}
\eeq
This happens, for instance, at generic spectral edges if $d \geq 5$.

\begin{thm}
\label{improve}
Let the assumption (\ref{E:strength}) above be satisfied. Then, if
\begin{itemize}
\item either $p \geq 2$ and $N \geq 0$,
\item or $p \in [1,2)$ and $2pN>(2-p)d$,
\end{itemize}
then the inequality \mref{RRLinequality} holds.

In particular, for any rigged divisor $\mu$,
\beq
\label{LRRL2poles}
\dim L_2(\mu^+, A ,0)=\deg_A (\mu^+)
\eeq
and
\beq
\label{LRRL2}
\dim L_2(\mu, A ,0)=\deg_A (\mu)+\dim L_2(\mu^{-1}, A^* ,0).
\eeq
\end{thm}

\begin{remark}
When $\mu$ is trivial, the equality \mref{LRRL2poles} means the absence of non-zero $L^2$- solutions (bound states). Thus, generically, spectral edges are not eigenvalues. In this case, the condition on integrability of $\|A_r(k)^{-1}\|^2_{\mathcal{L}(\mathbb{C}^{m_r})}$ is not required.
\end{remark}
The following $L^2$-solvability result (or Fredholm alternative) follows from \mref{LRRL2}.
\begin{prop}
\label{solvability}
Let $D$ be a non-empty compact nowhere dense subset of $X$ and a finite dimensional subspace $L \subset \mathcal{E}\cprime_{D}(X)$ be given.
Define the finite dimensional subspace $\tilde{L}=\{v \in \mathcal{E}\cprime_D(X) \mid A^*v \in L\}$. Consider any $f \in C^{\infty}_c(X)$ satisfying $(f, \tilde{L})=0$.
Under the assumptions of Theorem \ref{improve} for the periodic operator $A$ of order $m$,
the following two statements are equivalent:
\begin{enumerate}[(i)]
\item
$f$ is orthogonal to each element in the space $L_2(\mu, A^*,0)$, where $\mu$ is the rigged divisor $(D, L;0,\emptyset)$.
\item
The inhomogeneous equation $Au=f$ has a (unique) solution $u$ in $H^m(X)$ such that $(u, L)=0$.
\end{enumerate}
\end{prop}

We describe now examples that show that when $\mu$ is negative (i.e., $\mu^+$ is trivial), the Liouville-Riemann-Roch \underline{equality} might fail miserably (while still holding for some negative divisors).
\begin{prop}
\label{rrlexample1}
Consider the Laplacian $-\Delta$ on $\mathbb{R}^d$, $d \geq 3$.
For any $N \geq 0$ and positive integer $\ell$, there exists a rigged divisor $\mu=\mu^-$, such that
\begin{equation}
\label{RRLstrict}
\dim L_{\infty}(\mu,-\Delta,N)-\dim L_{1}(\mu^{-1},-\Delta,-N)\geq \ell+\dim{V_{N}^{\infty}(-\Delta)}+\deg_{-\Delta}(\mu).
\end{equation}
In other words, the difference between left and right hand sides of the Liouville-Riemann-Roch inequality can be made arbitrarily large.
\end{prop}
On the other hand, one can also achieve an equality for a negative divisor $\mu=\mu^-$. Indeed, let $A=A(D)$ be an elliptic constant-coefficient homogeneous differential operator of order $m$ on $\mathbb{R}^d$ that satisfies Assumption $\mathcal{A}$. Consider two non-negative integers $M_0 \geq M_1$, fix a point $x_0$ in $X$, and define $D^-:=\{x_0\}$ and the following finite dimensional vector subspace of $\mathcal{E}'_{D^-}(\mathbb{R}^d)$:
$$L^-:=\Span_{\mathbb{C}} \{\partial^{\alpha}\delta(\cdot-x_0)\}_{M_1 \leq |\alpha| \leq M_0}.$$
Let $\mu$ be a rigged divisor on $\mathbb{R}^d$ of the form $(D^+, L^+; D^-, L^-)$, where $D^+$ is a nowhere dense compact subset of $\mathbb{R}^d\setminus \{x_0\}$ and $L^+$ is a finite dimensional subspace of $\mathcal{E}'_{D^+}(\mathbb{R}^d)$.\footnote{A particular case is when $M_1=0$ and $\mu^+$ is trivial, i.e., $\mu=\mu^-$ becomes the point divisor $x_0^{-(M_0+1)}$ (see Definition \ref{D:divisor_point}).}
\begin{prop}
\label{rrlexample2}
Assume that one of the following two conditions holds:
\begin{itemize}
\item either
$1 \leq p<\infty$, $N>d/p+M_0$,
\item or
$p=\infty$, $N \geq M_0$.
\end{itemize}
Then $\dim L_{p'}(\mu^{-1},A^*,-N)=0$ and the Liouville-Riemann-Roch equality holds:
\begin{equation*}
\dim L_{p}(\mu,A,N)=\dim{V_{N}^{p}(A)}+\deg_{A}(\mu).
\end{equation*}
\end{prop}
\begin{cor}
\label{uppersemicont}
Suppose that $\{A_z\}_{z}$ and $\{\mu_z\}_z$ are families of such periodic elliptic operators $A_z$ satisfying the same assumption of Theorem \ref{improve} and of rigged divisors $\mu_z$ that depend continuously on a parameter $z$. Then the functions $z \mapsto \dim L_2(\mu_z, A_z ,0)$ and $z \mapsto \dim L_2(\mu_z^{-1}, A^*_z ,0)$ are upper-semicontinuous.
\end{cor}
\section{Empty Fermi surface}

For a periodic elliptic operator $A$, the claim of emptiness of the (real) Fermi surface (as always in this text, at the $\lambda=0$ level) is equivalent to its spectrum not containing zero (see Theorem \ref{specA}):
$$
0\notin \sigma(A).
$$
Interpreting the Fermi surface emptiness condition this way allows us to apply it even for non-periodic operators, the opportunity that we will use in this section.

Let us look at the periodic case first. Then the Liouville theorem becomes trivial, because the emptiness of the Fermi surface and thus zero being not in the spectrum implies that there is no non-zero polynomially growing solution (see \cite[Theorem 4.3]{KP2}, which is an analog of the Schnol' theorem, see e.g., \cite{CFKS, Glazman, Schnol})\footnote{The Schnol' type theorems claim that under appropriate conditions presence of a non-trivial solution of sub-exponential growth implies presence of spectrum. The example of hyperbolic plane shows that this is not always true, but a correct re-formulation of the sub-exponential growth condition \cite[Theorem 3.2.2 for the quantum graph case and Section 3.8]{BerKment} fixes this issue.}. As we show below, one can obtain now a Liouville-Riemann-Roch type result by combining the Riemann-Roch and the Schnol theorems. In fact, this will be done in the much more general case of $C^{\infty}$-bounded uniformly elliptic operators (not necessarily periodic) on a co-compact (not necessarily abelian) Riemannian covering of sub-exponential growth. We will use here the results of \cite{weakbloch} showing that a more general and stronger statement than the Schnol theorem holds for any $C^{\infty}$-bounded uniformly elliptic operator on a manifold of bounded geometry and sub-exponential growth.

We begin with some definitions from \cite{weakbloch} first \footnote{One could find details about analysis on manifolds of bounded geometry in e.g., \cite{Eich, Shubin_spectral, weakbloch}.}.
Let $\mathcal{X}$ be a co-compact connected Riemannian covering and $\mathcal{M}$ be its base. The deck group $G$ is a countable, finitely generated, and discrete (not necessarily abelian). Let $d_{\mathcal{X}}(\cdot, \cdot)$ be the $G$-invariant Riemannian distance on $\mathcal{X}$.
Due to the compactness of $\mathcal{M}$, there exists $r_{inj}>0$ (injectivity radius) such that
for every $r \in (0, r_{inj})$ and every $x \in \mathcal{X}$, the exponential geodesic mapping $\exp_{x}: T\mathcal{X}_x \rightarrow \mathcal{X}$ is a diffeomorphism of the Euclidean ball $B(0,r)$ centered at $0$ with radius $r$ in the tangent space $T\mathcal{X}_x$ onto the geodesic ball $B_{\mathcal{X}}(x,r)$ centered at $x$ with the same radius $r$ in $\mathcal{X}$. Taking $r_0 \in (0, r_{inj})$, the geodesic balls $B_{\mathcal{X}}(x,r)$, where $0<r \leq r_0$, are called \textbf{canonical charts} with $x$-coordinate in the charts.
\begin{defi}\cite{Shubin_spectral, weakbloch}
\label{subexp}
\begin{enumerate}[(i)]
\item
A differential operator $P$ of order $m$ on $\mathcal{X}$ is \textbf{$C^{\infty}$-bounded} if in every canonical chart, $P$ can be represent as $\sum_{|\alpha| \leq m}a_{\alpha}(x)\partial^{\alpha}_x$, where the coefficients $a_{\alpha}(x)$ are smooth and for any multi-index $\beta$, $|\partial^{\beta}_x a_{\alpha}(x)|\leq C_{\alpha\beta}$, where the constants $C_{\alpha\beta}$ are independent of the chosen canonical coordinates.
\item
A differential operator $P$ of order $m$ on $\mathcal{X}$ is \textbf{uniformly elliptic}\footnote{Clearly, any $G$-periodic elliptic differential operator with smooth coefficients on $\mathcal{X}$ is $C^{\infty}$-bounded uniformly elliptic.} if
\begin{equation*}
\label{uniform_ell}
|P_0^{-1}(x,\xi)|\leq C|\xi|^{-m}, \quad (x,\xi) \in T^*\mathcal{X}, \xi \neq 0.
\end{equation*}
Here $T^*\mathcal{X}$ is the cotangent bundle of $\mathcal{X}$, $P_0(x,\xi)$ is the principal symbol of the operator $P$, and $|\xi|$ is the length of the covector $(x,\xi)$ with respect to the metric on $T^*\mathcal{X}$ induced by the Riemannian metric on $\mathcal{X}$.
\item
$\mathcal{X}$ is of \textbf{subexponential growth} if the volumes of balls of radius $r$ grow subexponetially as $r \rightarrow \infty$, i.e., for any $\epsilon>0$ and $r>0$,
$$\sup_{x \in \mathcal{X}}\vol{B(x,r)}=O(\exp{(\epsilon r)}).$$
Here $\vol(\cdot)$ is the Riemannian volume on $\mathcal{X}$.
\item
Let $x_0$ be a fixed point in $\mathcal{X}$. A differential operator $P$ on $\mathcal{X}$ satisfies \textbf{Strong Schnol Property} (SSP) if the following statement is true:
If there exists a non-zero solution $u$ of the equation $Pu=\lambda u$ such that for any $\epsilon>0$
$$u(x)=O(\exp{(\epsilon d_{\mathcal{X}}(x,x_0))})$$
then $\lambda$ is in the spectrum of $P$.
\end{enumerate}
\end{defi}

We now turn to a brief discussion of growth of groups (see e.g., \cite{NY,Gri1,Gro_poly}).
Let us pick a finite, symmetric generating set $S$ of $G$. The \textbf{word metric} associated to $S$ is denoted by $d_S: G \times G \rightarrow \mathbb{R}$, i.e., for every pair $(g_1, g_2)$ of two group elements in $G$, $d_S(g_1, g_2)$ is the length of the shortest representation in $S$ of $g_1^{-1}g_2$ as a product of generators from $S$. Let $e$ be the identity element of $G$. The volume function of $G$ associated to $S$ is the function $\vol_{G,S}: \mathbb{N} \rightarrow \mathbb{N}$ defined by assigning to every $n \in \mathbb{N}$ the cardinality of the open ball $B_{G,S}(e,n)$ centered at $e$ with radius $n$ in the metric space $(G,d_S)$. Although the values of this volume function depend on the choice of the generating set $S$, its asymptotic growth type is independent of it. The group $G$ is said to be of \textbf{subexponential growth} if
$$\lim_{n \rightarrow \infty}\frac{\ln{\vol_{G,S}(n)}}{n}=0.$$
It is known that the deck group $G$ is of subexponential growth if and only if the covering $\mathcal{X}$ is so (see e.g., \cite[Proposition 2.1]{SaloffCoste}).
Virtually nilpotent groups clearly have polynomial growth, and the celebrated Gromov's theorem \cite{Gro_poly} shows that they are the only ones. Thus, any virtually nilpotent co-compact Riemannian covering $\mathcal{X}$ is of subexponential growth. Groups with intermediate growth, which were constructed by Grigorchuk \cite{Gri1}, provide other non-trivial examples of Riemannian coverings with subexponential growth.
\begin{thm}
\cite[Theorem 4.2]{weakbloch}
\label{SSP}
If $\mathcal{X}$ is of subexponential growth, then any $C^{\infty}$-bounded uniformly elliptic operator on $\mathcal{X}$ satisfies (SSP).
\end{thm}
\begin{remark}
A Schnol type theorem can be established without the subexponential growth condition, if the growth of a generalized eigenfunction is controlled in an integral (over an expanding ball), rather than point-wise sense. See \cite[Theorem 3.2.2 for the quantum graph case and Section 3.8]{BerKment}.
\end{remark}

Similarly, we also say that a positive function $\varphi: G \rightarrow \mathbb{R}^+$ is of \textbf{subexponential growth} if
$$\lim_{|g| \rightarrow \infty}\frac{\ln{\varphi(g)}}{|g|}=0,$$
where the \textbf{word length} of $g\in G$ is defined as $|g|:=d_{S}(e,g)$.
Again, this concept does not depend on the choice of the finite generating set $S$ (see \cite[Theorem 1.3.12]{NY}).
\begin{defi}
\label{subexp}
Let $\varphi$ be a positive function defined on the deck group $G$ such that both $\varphi$ and its inverse $\varphi^{-1}$ are of subexponential growth. Let us denote by $\mathcal{S}(G)$ the set of all such $\varphi$ on $G$. Then for any $p \in [1,\infty]$ and $\varphi \in \mathcal{S}(G)$, we define:
$$\cV^{p}_{\varphi}(\mathcal{X})=\{u \in C^{\infty}(\mathcal{X}) \mid \{\|u\|_{L^2(g\mathcal{F})}\varphi^{-1}(g)\}_{g \in G} \in \ell^p(G)\},$$
where $\mathcal{F}$ is a fundamental domain for $\mathcal{M}$ in $\mathcal{X}$.

Also, let $P^p_{\varphi}$ be the operator $P$ with the domain $\{u \in \cV^p_{\varphi}(\mathcal{X}) \mid Pu \in C^{\infty}_c(\mathcal{X})\}$. We denote by $L_p(\mu,P,\varphi)$ the space $L(\mu, P^p_{\varphi})$, where $\mu$ is a rigged divisor on $\mathcal{X}$. In a similar manner, we also define the space $L_p(\mu, P^*, \varphi)$, where $P^*$ is the transpose of $P$.
In particular, if $G=\bZ^d$ and $\varphi(g)=\langle g\rangle^N$, $\cV^p_{\varphi}(\mathcal{X})$ is the space $V^p_{N}(\mathcal{X})$ introduced in Definition \ref{polyspace}, while $L_p(\mu,P,\varphi)$ coincides with the space $L_p(\mu,P,N)$ appearing in Definition \ref{polyrigged}.
\end{defi}
We can now state our result.
\begin{thm}
\label{RRLbloch}
Consider any Riemannian co-compact covering $\mathcal{X}$ of subexponential growth with a discrete deck group $G$. Let $P$ be a $C^{\infty}$-bounded uniformly elliptic differential operator $P$ of order $m$ on $\mathcal{X}$ such that $0 \notin \sigma(P)$. Let us denote by $\varphi_0$ the constant function $1$ defined on $G$. Then the following statements are true:
\begin{enumerate}[(a)]
\item
For each rigged divisor $\mu$ on $\mathcal{X}$,  $L_{p}(\mu,P,\varphi)=L_{\infty}(\mu, P, \varphi_0)$, where $p \in [1,\infty]$ and $\varphi \in \mathcal{S}(G)$.  Thus, all the spaces $L_{p}(\mu,P,\varphi)$ are the same.

\item
$\dim{L_{\infty}(\mu,P,\varphi_0)}=\deg_{P}(\mu)+\dim{L_{\infty}(\mu^{-1},P^*,\varphi_0)}$.

\item
If $\mu$ is a positive divisor, i.e. $\mu=(D^+,L^+;\emptyset,0)$, then $\dim{L_{\infty}(\mu,P,\varphi_0)}=\deg_{P}(\mu)$.
\end{enumerate}
\end{thm}

Now let $D \subset\mathcal{X}$ be a compact nowhere dense subset, $L$ be a finite dimensional subspace of $\mathcal{E}\cprime_{D}(\mathcal{X})$, and $\mu:=(D, L;0,\emptyset)$ be a positive divisor.

We define the space $$\tilde{L}:=\{u \in \mathcal{E}\cprime_D(\mathcal{X}) \mid P^*u \in L\}.$$

The following analog of Corollary \ref{solvability} holds:
\begin{cor}
\label{subexp-solvability}
Let the assumptions of Theorem \ref{RRLbloch} hold and a function $f \in C^{\infty}_c(X)$ be given such that $(f, \tilde{L})=0$. The following statements are equivalent:
\begin{enumerate}[(i)]
\item
$f$ is orthogonal to the vector space $L_{\infty}(\mu, P^*,\varphi_0)$.
\item
There exists a unique solution $u$ of the inhomogeneous equation $Pu=f$ such that $(u, L)=0$ and $u \in \mathcal{V}^{p}_{\varphi}(\mathcal{X})$ for some $p \in [1,\infty]$ and $\varphi \in \mathcal{S}(G)$.
\item
The equation $Pu=f$ admits a unique solution $u$ which has subexponential decay and satisfies $(u, L)=0$.
\item
The equation $Pu=f$ admits a unique solution $u$ which has exponential decay and satisfies $(u, L)=0$.
\end{enumerate}
\end{cor}

\begin{remark}
\indent
\begin{enumerate}[(i)]
\item
Comparing to the Riemann-Roch formula \mref{RReq}, the Fredholm index of $P$ does not appear in the formula in Theorem \ref{RRLbloch} (b) since $P$ is invertible in this case.
\item
When $\mu$ is trivial, Theorem \ref{RRLbloch} (c) becomes Theorem \ref{SSP} and our Corollary \ref{subexp-solvability} is an analog of \cite[Theorem 4.2.1]{Kbook} for the co-compact Riemannian coverings of subexponential growth.
\end{enumerate}
\end{remark}

\chapter{Proofs of the main results}
\label{C:proofs-LRR}
First, we introduce some notions.

\begin{defi} We will often use the notation $A \lesssim B$ to indicate that the quantity $A$ is less or equal than the quantity $B$ up to some multiplicative constant factor, which does not affect the arguments.
\end{defi}

\begin{defi}
\label{apmn}
For each $s, N \in \mathbb{R}$ and $p \in [1,\infty]$, we denote by $V^{p}_{s,N}(X)$ the vector space consisting of all function $u \in C^{\infty}(X)$ such that for some (and thus any) compact subset $K$ of $X$ satisfying \mref{fundamental}, the sequence
$\{\|u\|_{H^s(gK)}\langle g \rangle^{-N}\}_{g \in G}$ belongs to $\ell^p(G)$. For a periodic elliptic operator $A$ we put
$$V^{p}_{s,N}(A):=V^{p}_{s,N}(X) \cap \Ker{A}.$$
Let also $A^p_{s,N}$ be the elliptic operator $A$ with the domain
$$\Dom{A^p_{s,N}}=\{u \in V^{p}_{s,N}(X) \mid Au \in C^{\infty}_c(X)\}.$$
\end{defi}
When $s=0$, this reduces to the notions of $V^{p}_{N}(X)$ and $V^{p}_{N}(A)$ introduced in Definition \ref{polyspace}, and of the operator $A^p_N$ and its domain $\Dom{A^p_N}$ in Definition \ref{polyrigged}.

\section{Proof of Theorem 3.6}
So, let $A$ be a periodic elliptic differential operator of order $m$ on $X$ and a pair $(p, N)$ satisfy the assumption of the theorem.

Let $\cF$ be the closure of a fundamental domain for $G$-action on $X$. We also pick a compact neighborhood $\hat{\cF}$ in $X$ of $\cF$, so the conclusion of Proposition \ref{schauderest} applies.

Our proof will be done in several steps.

\textbf{Step 1.} We claim that given $p \in [1,\infty], N \in \mathbb{R}$, and any rigged divisor $\mu=(D^+,L^+;D^-,L^-)$, one has
\beq
\label{redefine}
L(\mu,A^p_{m,N})=L(\mu,A^p_N)=L_p(\mu,A,N).
\eeq
Indeed, it suffices to show that $L(\mu,A^p_N) \subseteq L(\mu,A^p_{m,N})$.

Consider $u \in L(\mu,A^p_{N})$. Due to Remark \ref{domApn}, this implies that
\beq
\label{domainApn}
\{\|u\|_{L^2(g\hat{\cF})}\cdot \langle g \rangle^{-N}\}_{G_{\hat{\cF},D^+}}\in \ell^p(G_{\hat{\cF},D^+}),
\eeq
where $G_{\hat{\cF}, D^+}=\{g \in G \mid \dist{(g\hat{\cF},D^+)} \geq 1\}$.
Let $\mathcal{O}:=X \setminus D^+$ then $Au=0$ on $\mathcal{O}$ and moreover, the set $G^{\mathcal{O}}=\{g \in G \mid g\hat{\cF} \cap D^+=\emptyset\}$ contains $G_{\hat{\cF}, D^+}$. Due to the Schauder estimate of Proposition \ref{schauderest}, for any $g \in G_{\hat{\cF}, D^+}$, one has
\beq
\label{schauderApn}
\|u\|_{H^m(g\cF)} \lesssim \|u\|_{L^2(g\hat{\cF})}.
\eeq

By \mref{domainApn} and \mref{schauderApn},
\beq
\{\|u\|_{H^m(g\cF)}\cdot \langle g \rangle^{-N}\}_{G_{\cF,D^+}}\in \ell^p(G_{\cF,D^+}).
\eeq
Using Remark \ref{domApn} again, this shows that $u \in L(\mu, A^p_{m,N})$, which in turn proves \mref{redefine}.

\begin{defi}
We denote by $(A^p_{m,N})^*$ the elliptic operator $A^*$ with the domain
\beqn
\Dom{(A^p_{m,N})^*}=\{v \in V^{p'}_{m, -N}(X) \mid A^*v \in C^{\infty}_c(X)\},
\eeqn
where $1/p+1/p'=1$. In other words, $(A^p_{m,N})^*=(A^*)^{p'}_{m,-N}$.

We also define
$$\Domp{A^p_{m,N}}=\Domp{(A^p_{m,N})^*}:=C^{\infty}_c(X).$$
\end{defi}
Clearly, $C^{\infty}_c(X) \subseteq \Dom{A^p_{m,N}}$ and $\Dom{(A^p_{m,N})^*} \subseteq C^{\infty}(X)$ (see \mref{domainP}).

This step shows that instead of dealing with $A^p_N$, it suffices to work with $A^p_{m,N}$ and its ``adjoint'' $(A^p_{m,N})^*$ (in the sense of Subsection \ref{RRsetting}), which are easier to deal with.

In the next steps, we will apply Theorem \ref{RR} to the operators $A^p_{m,N}$ and $(A^p_{m,N})^*$.

\textbf{Step 2.}
In order to apply Theorem \ref{RR}, we need to check properties $(\mathcal{P}1)-(\mathcal{P}4)$.
The first three, $(\mathcal{P}1)-(\mathcal{P}3)$ hold by definition. To show that $(\mathcal{P}4)$ also holds, let us consider $u \in \Dom{A^p_{m,N}}$ and $v \in \Dom{(A^p_{m,N})^*}$.

Note that since the operator $A$ is $G$-periodic,
\beq
\|Au\|_{L^2(g\cF)} \lesssim \|u\|_{H^m(g\cF)}
\eeq
and
\beq
\|A^*v\|_{L^2(g\cF)} \lesssim \|v\|_{H^m(g\cF)}\eeq
for any $g \in G$.
Now, by H\"older's inequality, we have

\beq
\label{est1}
\begin{split}
\left|\sum_{g \in G}\langle Au, v \rangle_{L^2(g\cF)}\right|
\leq& \sum_{g \in G}\left|\langle Au, v \rangle_{L^2(g\cF)}\right|
\leq\sum_{g \in G}\|Au\|_{L^2(g\cF)}\cdot \|v\|_{L^2(g\cF)}
\\
\leq&\|\{\|Au\|_{L^2(g\cF)}\langle g\rangle^{-N}\}_{g \in G}\|_{\ell^p(G)}\cdot \|\{\|v\|_{L^2(g\cF)}\langle g\rangle^{N}\}_{g \in G}\|_{\ell^{p'}(G)}\\
\lesssim&\|\{\|u\|_{H^m(g\cF)}\langle g\rangle^{-N}\}_{g \in G}\|_{\ell^p(G)}\cdot \|\{\|v\|_{H^m(g\cF)}\langle g\rangle^{N}\}_{g \in G}\|_{\ell^{p'}(G)}
\\
<&\infty.
\end{split}
\eeq
Similarly,
$$\left|\sum_{g \in G}\langle u, A^*v \rangle_{L^2(g\cF)}\right|\leq
\|\{\|u\|_{H^m(g\cF)}\langle g\rangle^{-N}\}_{g \in G}\|_{\ell^p(G)}\cdot \|\{\|v\|_{H^m(g\cF)}\langle g\rangle^{N}\}_{g \in G}\|_{\ell^{p'}(G)}
<\infty.$$
Hence, both $\langle A^p_{m,N}u, v \rangle$ and $\langle u, (A^p_{m,N})^*v \rangle$ are well-defined.

Our goal is to show that these two quantities are equal. To do this, for each $r \in \mathbb{N}$,
we define
$$G_r:=\{g \in G \mid |g| \geq r\},$$
and by $\cF_r$ the union of all shifts of $\cF$ by deck group elements whose word lengths do not exceed $r$, i.e.,
$$\cF_r:=\bigcup_{g \notin G_{r+1}}g\cF=\bigcup_{|g| \leq r}g\cF.$$
Obviously, $\cF_r \Subset \cF_{r+1}$ for any $r \in \mathbb{N}$ and the union of these subsets $\cF_r$ is the whole covering $X$.
Let $\phi_r \in C^{\infty}_c(X)$ be a cut-off function such that $\phi_n=1$ on $\cF_r$ and $\supp{\phi_r} \Subset \cF_{r+1}$.
Furthermore, all derivatives of $\phi_r$ are uniformly bounded with respect to $r$. In particular, the following estimates hold for any smooth function $w$ on $X$ and any $g\in G$:
\beq
\label{uniform}
\|\phi_r w\|_{H^m(g\cF)} \lesssim \|w\|_{H^m(g\cF)}, \quad \|(1-\phi_r) w\|_{H^m(g\cF)} \lesssim \|w\|_{H^m(g\cF)}.
\eeq
Let $u_r:=\phi_r u$ and $v_r:=\phi_r v$. Since $u_r$ and $v_r$ are compactly supported smooth functions on $X$, $\langle Au_r, v_r \rangle=\langle u_r, A^*v_r \rangle$. Therefore, it is enough to show that
\beq
\label{conv1}
\langle Au_r, v_r \rangle \rightarrow \langle Au,v \rangle \quad \mbox{and} \quad \langle u_r, A^*v_r \rangle \rightarrow \langle u, A^*v \rangle,
\eeq
as $r \rightarrow \infty$. By symmetry, we only need to show the first part of \mref{conv1}.
We use the triangle inequality to reduce \mref{conv1} to checking that
\beq
\label{conv2}
\lim_{r \rightarrow \infty}\langle A(u-u_r), v \rangle=\lim_{r \rightarrow \infty}\langle Au_r, (v-v_r) \rangle=0.
\eeq
We repeat the argument of \mref{est1} for the pairs of functions $((1-\phi_r )u, v)$ and $(\phi_r u, (1-\phi_r)v)$, and then use \mref{uniform} to derive
\beq
\label{conv3}
\begin{split}
&|\langle Au_r, (v-v_r)\rangle|+|\langle A(u-u_r), v\rangle|\\
\leq &\sum_{g \in G}\left|\langle A(\phi_r u),(1-\phi_r)v\rangle_{L^2(g\cF)}\right|+\left|\langle A((1-\phi_r)u),v\rangle_{L^2(g\cF)}\right|\\
=&\sum_{|g| \geq r+1}\left|\langle A(\phi_r u),(1-\phi_r)v\rangle_{L^2(g\cF)}\right|
+\left|\langle A((1-\phi_r)u),v\rangle_{L^2(g\cF)}\right|
\\
\lesssim &\|\{\|\phi_r u\|_{H^m(g\cF)}\langle g \rangle^{-N}\}_{g \in G_{r+1}}\|_{\ell^p(G_{r+1})}
\cdot  \|\{\|(1-\phi_r)v\|_{H^m(g\cF)}\langle g \rangle^{N}\}_{g \in G_{r+1}}\|_{\ell^{p'}(G_{r+1})}\\
&+\|\{\|(1-\phi_r)u\|_{H^m(g\cF)}\langle g \rangle^{-N}\}_{g \in G_{r+1}}\|_{\ell^p(G_{r+1})}
\cdot  \|\{\|v\|_{H^m(g\cF)}\langle g \rangle^{N}\}_{g \in G_{r+1}}\|_{\ell^{p'}(G_{r+1})}\\
\lesssim &\|\{\|u\|_{H^m(g\cF)}\langle g \rangle^{-N}\}_{g \in G_{r+1}}\|_{\ell^p(G_{r+1})}
\cdot  \|\{\|v\|_{H^m(g\cF)}\langle g \rangle^{N}\}_{g \in G_{r+1}}\|_{\ell^{p'}(G_{r+1})}.
\end{split}
\eeq
Since $u \in V^p_{m,N}(X)$ and $v \in V^{p'}_{m,-N}(X)$, it follows that  as $r \rightarrow \infty$, either
$$\|\{\|u\|_{H^m(g\cF)}\langle g \rangle^{-N}\}_{g \in G_{r+1}}\|_{\ell^p(G_{r+1})}$$
or
$$\|\{\|v\|_{H^m(g\cF)}\langle g \rangle^{N}\}_{g \in G_{r+1}}\|_{\ell^{p'}(G_{r+1})}$$
converges to zero (depending on either $p$ or $p'$ is finite), while the other one stays bounded.
Thus, we have
$$\lim_{r \rightarrow \infty}\|\{\|u\|_{H^m(g\cF)}\langle g \rangle^{-N}\}_{g \in G_{r+1}}\|_{\ell^p(G_{r+1})}
\cdot  \|\{\|v\|_{H^m(g\cF)}\langle g \rangle^{N}\}_{g \in G_{r+1}}\|_{\ell^{p'}(G_{r+1})}=0.
$$
This fact and \mref{conv3} imply \mref{conv2}. Hence, the property $(\mathcal{P}4)$ holds for $A^p_{m,N}$ and $(A^p_{m,N})^*$.

\textbf{Step 3.}
Clearly,
\beq
\Ker{A^p_{m,N}}=\{u \in \Dom A^p_{m,N} \mid Au=0\}=V^p_{m,N}(A)=V^p_{N}(A).
\eeq
The latter equality is due to Schauder estimates (see \mref{schauderApn} in Step 1).
Also,
\beq
\Ker (A^p_{m,N})^*=V^{p'}_{m,-N}(A^*)=V^{p'}_{-N}(A^*)=0
\eeq
according to Theorem \ref{UCinfty}. Hence, the kernels of $A^p_{m,N}$ and $(A^p_{m,N})^*$ are finite dimensional.

To prove that $A^p_{m,N}$ is Fredholm on $X$, we only need to show that
\beq
\Image A^p_{m,N}=C^{\infty}_c(X)=\left(\Ker{(A^p_{m,N})^*}\right)^{\circ}.
\eeq
 Given any $f \in C^{\infty}_c(X)$, we want to find a solution $u$ of the equation $Au=f$ such that $u \in V^p_{N}(X)$. If such a solution $u$ is found, then automatically $u$ is in $V^p_{m,N}(X)$ by the same argument as in Step 1 and the fact that $Au=0$ on the complement of the compact support of $f$.
Thus, $f$ must belong to the range of $A^p_{m,N}$ and the proof is then finished. So our remaining task is to find such a solution $u$. This can be done as follows: First,
we pick a cut-off function $\eta_r$ such that $\eta=1$ around $k_{r}$ and $\supp \eta_{r} \Subset V_{r}$, where $V_r$ is the neighborhood of $k_r$ appearing in \textit{Assumption} $\mathcal{A}$. Define
\beq
\eta:=\sum_{r=1}^{\ell} \eta_r
\eeq
and notice that the Floquet transform $\textbf{F}f$ is smooth in $(k,x)$ since $f \in C^{\infty}_c(X)$.
We decompose $\textbf{F}f=\eta\textbf{F}f+(1-\eta)\textbf{F}f$.
Since the operator $A(k)$ is invertible when $k \notin F_{A, \mathbb{R}}$, the operator function
\beq
\widehat{u_0}(k):=A(k)^{-1}((1-\eta(k))\textbf{F}f(k))
\eeq
is well-defined and smooth in $(k,x)$.
By Theorem \ref{pwfloquet}, the function $u_0:=\textbf{F}^{-1}\widehat{u_0}$ has rapid decay.
We recall that when $k \in V_r$, the Riesz projection $\Pi_r(k)$ is defined in \textit{Assumption} $\mathcal{A}$. Clearly,
\beq
0 \notin \sigma(A(k)_{|R(1-\Pi_r(k))}),
\eeq
where we use the notation $R(T)$ for the range of an operator $T$. Now the operator function
\beq
\widehat{v_r}(k):=\eta_r(k)(A(k)_{|R(1-\Pi_r(k))})^{-1}(1-\Pi_r(k))\textbf{F}f(k)
\eeq
is also smooth and thus the function $v_r:=\textbf{F}^{-1}\widehat{v_r}$ has rapid decay by Theorem \ref{pwfloquet}.
In particular, $u_0, v_r (1 \leq r \leq \ell)$ are in the space $V^{\infty}_0(X)$.

Let us fix $1 \leq r \leq \ell$. For any $k \in V_r\setminus \{k_r\}$, due to $(\mathcal{A}4)$, we can define the operator function \beq
\widehat{w_r}(k):=\eta_r(k)(A(k)_{|R(\Pi_r(k))})^{-1}\Pi_r(k)\textbf{F}f(k),
\eeq
which is in the range of $\Pi_r(k)$.
By expanding $\Pi_r(k)\textbf{F}f(k)$ in terms of the basis $(f_j(k))_{1 \leq j \leq m_r}$, one sees that
$$\|\widehat{w_r}(k)\|_{L^2_k(X)} \lesssim \max\limits_{1 \leq j \leq m_r}\|A_r(k)^{-1}f_j(k)\|_{L^2_k(X)}\cdot \|\textbf{F}f(k)\|_{L^2_k(X)}.$$
From this and the integrability condition in $(\mathcal{A}4)$, we obtain
\beqn
\begin{split}
\int\limits_{\mathbb{T}^d}\|\widehat{w_r}(k)\|_{L^2_k(X)}\di{k} &\lesssim \int\limits_{V_r\setminus \{k_r\}}\|\textbf{F}f(k)\|_{L^2_k(X)}\cdot\|(A_r(k)^{-1}f_j(k))_{1 \leq j \leq m_r}\|_{\ell^{\infty}(\mathbb{C}^{m_r})}\di{k}\\
&\lesssim \sup\limits_{k \in V_r}\|\textbf{F}f(k)\|_{L^2_k(X)}\cdot\int\limits_{V_r\setminus \{k_r\}}\|A_r(k)^{-1}\|_{\mathcal{L}(\mathbb{C}^{m_r})}\di{k}<\infty.
\end{split}
\eeqn
Hence, $\widehat{w_r} \in L^1(\mathbb{T}^d, \mathcal{E}^0)$.

Summing up, the function $\widehat{u}:=\widehat{u}_0+\sum\limits_{1 \leq r \leq \ell}(\widehat{v}_r+\widehat{w}_r)$ belongs to $L^1(\mathbb{T}^d, \mathcal{E}^0)$, and moreover, it satisfies the equation
\beq
\begin{split}
A(k)\widehat{u}(k)&=A(k)\widehat{u_0}(k)+\sum\limits_{1 \leq r \leq \ell}A(k)(\widehat{v_r}(k)+\widehat{w_r}(k))\\
&=(1-\eta(k))\textbf{F}f(k)+\sum\limits_{1 \leq r \leq \ell}\eta_r(k)\textbf{F}f(k)=\textbf{F}f(k).
\end{split}
\eeq
From the above equality, $\widehat{u}(k,x)$ is smooth in $x$ for each quasimomentum $k$.
We define $u:=\textbf{F}^{-1}\widehat{u}$ by using the formula \mref{flinv}. According to Lemma \ref{flrl}, $u \in L^2_{loc}(X)$.
For any $\phi \in C^{\infty}_c(X)$, we can use Fubini's theorem to get
\beqn
\begin{split}
&\langle u, A^*\phi\rangle_{L^2(X)}=\int\limits_{X}\textbf{F}^{-1}\widehat{u}(k,x)\cdot A^*\phi(x)\di\mu_X(x)
=\frac{1}{(2\pi)^d}\int\limits_{\mathbb{T}^d}\int\limits_{X}\widehat{u}(k,x)\cdot A^*\phi(x)\di\mu_X(x)\di{k}\\
&=\frac{1}{(2\pi)^d}\int\limits_{\mathbb{T}^d}\int\limits_{X}A\widehat{u}(k,x)\cdot\phi(x)\di\mu_X(x)\di{k}=\frac{1}{(2\pi)^d}\int_{\mathbb{T}^d}\int_{X}A(k)\widehat{u}(k,x)\cdot\phi(x)\di\mu_X(x)\di{k}\\
&=\frac{1}{(2\pi)^d}\int\limits_{\mathbb{T}^d}\int\limits_{X}\textbf{F}f(k,x)\cdot\phi(x)\di\mu_X(x)\di{k}=\langle f, \phi \rangle_{L^2(X)}.
\end{split}
\eeqn
Hence, $u$ is a weak solution of the inhomogeneous equation $Au=f$ on $X$.
Elliptic regularity then implies that $u$ is a classical solution and therefore, $u \in V^{\infty}_0(X)$ due to Lemma \ref{flrl} again. If either $N \geq 0$ when $p=\infty$ or $N>d/p$ when $p \in [1,\infty)$, we always have $V^{\infty}_0(X) \subseteq V^{p}_{N}(X)$. Thus, this shows that $A^p_{m,N}$ is a Fredholm operator on $X$.

\textbf{Step 4}.
Due to considerations in Step 2 and Step 3, the operator $A^p_{m,N}$ satisfies the assumption of Theorem \ref{RR}. Then the Liouville-Riemann-Roch inequality \mref{RRLinequality} follows immediately from \mref{redefine} and Theorem \ref{RR}.
\qed

\begin{remark}
\label{rrl2d-example}
Assumption $(\mathcal{A}2)$ is needed to guarantee the validity of the Liouville-Riemann-Roch inequality \mref{RRLineq} (at least when $p=\infty$ and $N=0$). Indeed, consider $-\Delta$ in $\mathbb{R}^2$. Let $\mu$ be the point divisor $(\{0\}, L, \emptyset, 0)$, where $L=\mathbb{C}\delta_0$.
It is not difficult to see that the space $L_{\infty}(\mu, -\Delta, 0)$ contains only constant functions, since the standard fundamental solution $u_0(x)=-\frac{1}{2\pi}\ln{|x|}$ is not bounded at infinity. Clearly, $L_1(\mu^{-1}, -\Delta,0)$ is trivial. Hence, we have:
\beq
\dim L_{\infty}(\mu, -\Delta, 0)=1<2=\deg_{-\Delta}(\mu)+\dim V^{\infty}_0(-\Delta)+\dim L_1(\mu^{-1}, -\Delta,0).
\eeq
\end{remark}

%

\section{Proof of Theorem \ref{RRLeq}}
According to the Step 3 of the proof of Theorem \ref{RRLineq}, the operator $A^p_{m,N}$ is Fredholm on $X$ and
\beq
\Image{A^p_{m,N}}=C^{\infty}_c(X)=\Domp{(A^{p}_{m,N})^*}.
\eeq
Now we can apply Corollary \ref{RReqmu} to finish the proof of the equality \mref{RRLeqmu}. The upper bound estimate \mref{RRLineqmu} follows from \mref{RRLeqmu} and the trivial inclusion $L_p(\mu, A, N) \subseteq L_p(\mu^+, A, N)$.
\qed
\subsection{Proof of Proposition \ref{triviality}}

\begin{enumerate}[a.]
\item
It suffices to prove the statement for the case $p=\infty$. If $r \geq 0$, we define a point divisor $\mu_r:=(\emptyset,0;\{0\}, L_r^-)$, where
\beq
L_r^-:=\left\{\sum_{|\alpha| \leq r}c_{\alpha}\partial^{\alpha}\delta(\cdot-0) \mid c_{\alpha} \in \mathbb{C}\right\}.
\eeq
Let us consider the function $v_{\alpha}(x):=\partial^{\alpha}(|x|^{2-d})$ for each multi-index $\alpha$ such that $\displaystyle |\alpha|>N+2$.
It is clear that $|v_{\alpha}(x)| \lesssim |x|^{-|\alpha|-d+2}$ for $x \neq 0$. Therefore,
\beq
\sum\limits_{g \in \bZ^d}\|v_{\alpha}\|_{L^2([0,1)^d+g)}\cdot \langle g \rangle^{N} \lesssim \sum\limits_{g \in \bZ^d}\langle g \rangle^{-|\alpha|-d+2+N}<\infty.
\eeq
Since $|x|^{2-d}$ is a fundamental solution of $-\Delta$ on $\mathbb{R}^d$ (up to some multiplicative constant), $v_{\alpha}$ belongs to the space $L_{1}(\mu_r^{-1}, -\Delta_{|\mathbb{R}^d}, -N)$ provided that $N+2<|\alpha| \leq r$. Now let us pick multi-indices $\alpha_1, \ldots, \alpha_{r-N-2}$ such that $|\alpha_j|=N+2+j$ for any $1 \leq j \leq r-N-2$.
By homogeneity, the functions $v_{\alpha_j}$ are linearly independent as smooth functions on $\mathbb{R}^d \setminus \{0\}$.
By letting $r \rightarrow \infty$, this proves the statement a.

\item
We define $\mu_0:=(\emptyset, 0; D^-, L^-)$. Now suppose the contrary, that for any $M>0$, the space $L_{p'}(\mu_M^{-1}, A^*,-N)$ is non-trivial for some rigged divisor $\mu_M=(D^+, L^+_{M}; D^-, L^-)$ such that $\mathcal{L}^+_{M}\subseteq L^+_{M}$. Note that $L_{p'}(\mu_M^{-1}, A^*,-N)$ is a subspace of $L_{p'}(\mu_0^{-1}, A^*, -N)$.
It follows from Proposition \ref{RRLeq} that $L_{p'}(\mu_0^{-1}, A^*, -N)$ is a finite dimensional vector space and thus, we equip it with any norm $\|\cdot\|$.
Thus, there is a sequence $\{u_M\}_{M \in \mathbb{N}}$ in $L_{p'}(\mu_0^{-1}, A^*, -N)$ such that $\|u_M\|=1$ and $(u_M, L^+_M)=0$. In particular, $(u_M, \mathcal{L}^+_M)=0$ and therefore, $\partial^{\alpha}u_M(x_0)=0$ for any $0 \leq |\alpha| \leq M$. By passing to a subsequence if necessary, there exists $v \in L_{p'}(\mu_0^{-1}, A^*, -N)$ for which $\lim\limits_{M \rightarrow \infty}\|u_M-v\|=0$. It is clear that $\|w\|_{C^{M}(K)}\lesssim \|w\|$ for any $w$ in $ L_{p'}(\mu_0^{-1}, A^*, -N)$, $M \geq 0$, and compact subset $K \Subset X \setminus D^-$. Hence, for any multi-index $\alpha$, $\partial^{\alpha}v(x_0)=\lim\limits_{M \rightarrow \infty}\partial^{\alpha}u_M(x_0)=0$. As a local smooth solution of $A^*$, $v$ must vanish on $X \setminus D^-$ due to the strong unique continuation property of $A^*$. Consequently, $v=0$ as an element in $L_{p'}(\mu_0^{-1}, A^*, -N)$ and this gives us a contradiction with $\|v\|=1$. This completes our proof.
\end{enumerate}
\qed

\begin{remark}
\indent
\begin{enumerate}[(i)]
\item
There are large classes of elliptic operators with the strong unique continuation property, e.g. elliptic operators of second-order with smooth coefficients and elliptic operators with real analytic coefficients.
\item
Note that the finiteness of the real Fermi surface $F_{A,\mathbb{R}}$ would imply the weak unique continuation property of $A^*$, i.e., $A^*$ does not have any non-zero compactly supported solution (see e.g., \cite{Kbook}).
We do not know whether the first statement of Proposition \ref{triviality} still holds without the strong unique continuation property requirement for $A^*$.
\end{enumerate}
\end{remark}

\section{Proof of Theorem 3.13}
The proof is similar to the one of Theorem \ref{RRLineq}, except for Step 3, where it needs a minor modification. We keep the same considerations and notions as in Step 3.

The goal here is to prove the solvability in $V^p_N(X)$ of the equation $Au=f$, where $f \in C^{\infty}_c(X)$. Under the assumption (\ref{E:strength}), the functions $\widehat{w}_r$ $(1 \leq r \leq \ell)$ defined in Step 3 belong to $L^2(\mathbb{T}^d, \mathcal{E}^0)$. Thus, $\widehat{u} \in L^2(\mathbb{T}^d, \mathcal{E}^0)$ and then by Theorem \ref{pwfloquet}, $u$ is in $L^2(X)$ and $Au=f$. This means that $u \in V^{2}_0(X)$.

If $p \geq 2, N \geq 0$, the inclusion $V^{2}_0(X) \subseteq V^2_N(X)$ is obvious, while if $p \in [1,2), N>(1/p-1/2)d$, one can use H\"older's inequality to obtain the inclusion $V^{2}_0(X) \subset V^p_N(X)$.

This completes the proof of the first statement. In particular, when $p=2, N=0$, both operators $A^2_0$ and $(A^2_0)^*=(A^*)^2_0$ are Fredholm. Therefore, we obtain the equality \mref{LRRL2}, since $\dim V^2_0(A)=\dim V^2_0(A^*)=0$ according to Theorem \ref{UCinfty} (a).
\qed

\begin{remark}
\indent
\begin{enumerate}[(a)]
\item
The integrability of $\|A_r(k)^{-1}\|^2_{\mathcal{L}(\mathbb{C}^{m_r})}$ is important for validity of Theorem \ref{improve}. Indeed, let us consider $A=-\Delta$ on $\mathbb{R}^d$ $(d<5)$ and the point divisor $\mu$ representing a simple pole at $0$. Then the fundamental solution $c_d |x|^{2-d}$ does not belong to $L_2(\mu, -\Delta, 0)$ and thus, $\dim L_2(\mu,-\Delta,0)=0<1=\deg_{-\Delta}(\mu)$. Therefore, the
equalities \mref{LRRL2poles} and \mref{LRRL2} do not hold in this case.
\item
Under the assumption of Theorem \ref{improve}, the Liouville-Riemann-Roch inequality \mref{RRLinequality} holds for \textbf{any $N \geq 0$} if and only if $p \geq 2$. Indeed, suppose that $d \geq 5$ and $p<2$, then $(2-d)p \geq -d$ and therefore,
\beq
\int_{|x| \geq 1}|x|^{(2-d)p}\di x=\infty.
\eeq
This implies that $L_p(\mu, -\Delta,0)=\{0\}$, where $\mu$ is the same point divisor mentioned in the previous remark. So \mref{RRLinequality} fails, since,
\beq
\dim L_p(\mu,-\Delta,0)=0<1=\deg_{-\Delta}(\mu).
\eeq
\end{enumerate}
\end{remark}
\subsection{Proof of Proposition \ref{solvability}}
We evoke the operators $A^2_{m,0}$ and $(A^2_{m,0})^*$ and their corresponding domains from the proof of Theorem \ref{RRLineq}.
Now we recall from our discussion in Subsection \ref{GSRR} the notations of the operators
\beq
\widetilde{A^{2}_{m,0}}: \Gamma(X,\mu^{-1},A^{2}_{m,0}) \rightarrow \tilde{\Gamma}_{\mu^{-1}}(X,A^{2}_{m,0})
\eeq
and
\beq
\widetilde{(A^2_{m,0})^*}: \Gamma(X,\mu,(A^{2}_{m,0})^*) \rightarrow \tilde{\Gamma}_{\mu}(X,(A^{2}_{m,0})^*),
\eeq
which are extensions of $A^{2}_{m,0}$ and $(A^{2}_{m,0})^*$ with respect to the divisors $\mu^{-1}$ and $\mu$, correspondingly.
From \mref{LRRL2} and Remark \ref{RRfredholm}, we obtain the duality
\beq
L_2(\mu, A^*,0)^{\circ}=(\Ker \widetilde{(A^2_{m,0})^*)}^{\circ}=\Image \widetilde{A^{2}_{m,0}}.
\eeq
To put it differently, $f$ is orthogonal to $\tilde{L}$ and $L_2(\mu, A^*,0)$ if and only if $f=Au$ for some $u$ in the space
\beq
\Gamma(X,\mu^{-1},A^{2}_{m,0})=\{v \in H^m(X) \mid Av \in C^{\infty}_c(X), \langle v, L \rangle=0\}.
\eeq
This proves the equivalence of (i) and (ii).
\qed

\subsection{Proof of Proposition \ref{rrlexample1}}
For $\ell \in \mathbb{N}$, let us choose $\ell$ distinct points $z_1, \ldots, z_{\ell}$ in $\mathbb{R}^d$ and define $\mu_{\ell}:=(\emptyset,0; D^-, L^-)$, where $D^-=\{z_1, \ldots, z_{\ell}\}$ and
\beq
L^-=\left\{\sum\limits_{1 \leq j \leq \ell}\sum\limits_{1 \leq \alpha \leq d}c_{j\alpha}\frac{\partial}{\partial x_{\alpha}}\delta(x-z_j) \mid c_{j\alpha} \in \mathbb{C}\right\}.
\eeq
In terms of the notations in Example \ref{mainex}, $k=0, l=\ell$.

Let us recall now the spaces $L(\mu, -\Delta)$ and $L(\mu^{-1},-\Delta)$ from Example \ref{mainex}.
By definition, $L_{\infty}(\mu, -\Delta,0)=\mathbb{C}.$
Hence,
\beq
\label{exineq1}
\dim L_{\infty}(\mu, -\Delta,0)=1=\dim L(\mu, -\Delta)+\dim V^{\infty}_0(-\Delta).
\eeq
On the other hand, if $v \in L_{1}(\mu^{-1}, -\Delta,0)$, then
\beq
\lim\limits_{R \rightarrow \infty}\sum\limits_{|g| \geq R}\|v\|_{L^2([0,1)^d+g)}=0.
\eeq
Hence, $\|v\|_{L^2([0,1)^d+g)} \rightarrow 0$ as $|g| \rightarrow \infty$. Using elliptic regularity, this is equivalent to $\lim\limits_{|x| \rightarrow \infty}v(x)=0$.
Thus, $L_{1}(\mu^{-1}, -\Delta,0)$ is a subspace of $L(\mu^{-1}, -\Delta)$.
Define
\beq
v_{j\alpha}(x):=\frac{\partial}{\partial x_{\alpha}}|x-z_j|^{2-d}.
\eeq
Then $v_{j\alpha} \in L(\mu^{-1}, -\Delta)$ (see Example \ref{mainex})\footnote{In physics, the functions $v_{j\alpha}$ in $L(\mu^{-1},-\Delta)$ are dipoles potentials of dipoles located at the equilibrium positions $z_j$.}. We now claim that $v_{j\alpha} \notin L_1(\mu^{-1}, -\Delta,0)$ for any $1 \leq j \leq \ell$ and $1 \leq \alpha \leq d$. Suppose this is not true:
\beq
v_{j\alpha}(x)=(2-d)\frac{(x_{\alpha}-(z_j)_{\alpha})}{|x-z_j|^d} \in L_1(\mu^{-1}, -\Delta,0).
\eeq
This implies that for some $R>0$, we have
$$V_{\alpha,R}:=\sum\limits_{g \in \mathbb{Z}^d,|g| \geq R}\left(\int_{[0,1)^d+g}\frac{|x_{\alpha}-(z_j)_{\alpha}|^2}{|x-z_j|^{2d}}\di{x}\right)^{1/2}<\infty.$$
But this leads to a contradiction, since
\beq
V_{\alpha,R} \gtrsim \sum\limits_{g \in \mathbb{Z}^d, |g| \geq R}\frac{|g_{\alpha}|}{|g|^{d}} \geq \sum\limits_{g \in \mathbb{Z}^d, g_{\alpha} \neq 0, |g| \geq R}\frac{1}{|g|^{d}}=\infty.
\eeq

From this and linear independence of functions $v_{j\alpha}$ as smooth functions on $\mathbb{R}^d \setminus D^-$, it follows that
\beq
\label{exineq2}
\dim L_1(\mu^{-1}, -\Delta, 0)\leq \dim L(\mu^{-1}, -\Delta)-d\ell.
\eeq
From \mref{RRexample}, \mref{exineq1} and \mref{exineq2}, we get
\beq
\label{exineq3}
\begin{split}
&d\ell+\dim L_1(\mu^{-1}, -\Delta, 0)+\deg_{-\Delta}(\mu)+\dim V^{\infty}_0(-\Delta)
\\&\leq \dim L(\mu^{-1}, -\Delta)+\deg_{-\Delta}(\mu)+\dim V^{\infty}_0(-\Delta)
=\dim L_{\infty}(\mu, -\Delta, 0),
\end{split}
\eeq
which yields the inequality \mref{RRLstrict}.
\qed

\begin{remark}
\indent
\begin{enumerate}[(i)]
\item
Note that the examples from Proposition \ref{rrlexample1} also show that the Liouville-Riemann-Roch inequality can be strict in some other cases as well.

\textbf{Case 1:}
$p=\infty$ and $N \geq 0$.

If $(d-1)\ell +1\geq \dim V^{\infty}_N(-\Delta)$, one obtains from \mref{exineq3} that
\beq
\begin{split}
&\dim L_1(\mu^{-1}, -\Delta, -N)+\deg_{-\Delta}(\mu)+\dim V^{\infty}_N(-\Delta)+\ell\\
\leq &\dim L_1(\mu^{-1}, -\Delta, 0)+\deg_{-\Delta}(\mu)+\dim V^{\infty}_0(-\Delta)+d\ell\leq \dim L_{\infty}(\mu, -\Delta, N).
\end{split}
\eeq

\textbf{Case 2:}
$1 \leq p<\infty$ and $N>d/p$.

Note that $\dim L_p(\mu,-\Delta,N) \geq 1$, since this space contains constant solutions.
Each function $v_{j\alpha}$ does not belong to the space $L_{p'}(\mu^{-1},-\Delta,-N)$. In fact, for $R>2|z_j|$ large enough and $p>1$, we get
\beq
\sum\limits_{g \in \mathbb{Z}^d, |g|>R}\|v_{j\alpha}\|_{L^2([0,1)^d+g)}^{p'}\cdot \langle g \rangle^{p'N} \gtrsim \sum\limits_{\min\limits_{1 \leq l \leq d}g_l>R}\langle g \rangle^{p'(N-d)}
=\infty.
\eeq
The case when $p=1$ and $N>d-1$ can be treated similarly.

Now, as in the proof of \mref{exineq3}, we obtain
the inequality
\beq
\begin{split}
&\dim L_{p'}(\mu^{-1}, -\Delta, -N)+\deg_{-\Delta}(\mu)+\dim V^{p}_N(-\Delta)+\ell\\
\leq &\dim L(\mu^{-1}, -\Delta)-d\ell+\deg_{-\Delta}(\mu)+\dim V^p_N(-\Delta)+\ell\leq \dim L_{p}(\mu, -\Delta, N),
\end{split}
\eeq
provided that $(d-1)\ell+1 \geq \dim V^{p}_N(-\Delta)$.

\item
One can also modify our example in Proposition \ref{rrlexample1} to obtain examples of \mref{RRLstrict} in the case of point divisors. For instance, we could take the point divisor $\mu=(\emptyset,0;D^-,L^-)$, where $D^-=\{z_1,\ldots,z_{\ell}\}$ and
\beq
L^-=\Span_{\mathbb{C}}\{\partial^{\alpha}\delta(x-z_j)\}_{1 \leq j \leq \ell, 0 \leq |\alpha| \leq 1}.
\eeq
Similarly,
\beq
L_{\infty}(\mu,-\Delta,0)=L_1(\mu^{-1},-\Delta,0)=\{0\}.
\eeq
Moreover, $\deg_{-\Delta}(\mu)=-\ell(d+1)$. Hence,
\beq
\dim L_{\infty}(\mu,-\Delta,0)=(\ell(d+1)-1)+\deg_{-\Delta}(\mu)+\dim V^{\infty}_0(-\Delta)+\dim L_1(\mu^{-1},-\Delta,0).
\eeq

The method can also be easily adapted to providing examples of the inequality \mref{RRLstrict} when both the positive parts $\mu^+$ and negative parts $\mu^-$ of the rigged divisors $\mu$ are non-trivial.
\end{enumerate}
\end{remark}

\subsection{Proof of Proposition \ref{rrlexample2}}
We can assume w.l.o.g that $x_0=0$. Let us now fix a pair $(p,N)$ as in the assumption of the statement. We recall the notations of the operator $A^p_{m,N}$ and its corresponding domain $\Dom{A^p_{m,N}}$ from Definition \ref{apmn}.

In order to prove the statement of the proposition, we will apply Corollary \ref{RReqmu} to the operator $P:=A^p_{m,N}$.

From our assumption on the operator $A$ and the pair $(p,N)$ and from the conclusion of Step 3 of the proof of Theorem \ref{RRLineq}, we only need to show the following claim: If $u$ is a smooth function on $\mathbb{R}^d$ such that $|u(x)| \lesssim \langle x \rangle^{N}$ and $\langle Au, \tilde{L}^- \rangle=0$, then there is a polynomial $v$ of degree $M_0$ satisfying $Av=0$ and $\langle u-v, L^- \rangle=0$. Indeed, if this claim holds true, $v$ will belong to the space $\Dom{A^p_{m,N}}$ due to our condition on $p$ and $N$. This will fulfill all the necessary assumptions of Corollary \ref{RReqmu} in order to apply it.

To prove the claim, we first introduce the following polynomial:
\beq
v(x):=\sum\limits_{M_1 \leq |\alpha| \leq M_0}\frac{\partial^{\alpha}u(0)}{\alpha!}x^{\alpha}.
\eeq
Hence, $\langle v-u,g \rangle=0$ if $g=\partial^{\alpha}\delta(\cdot-0)$ and $M_1 \leq |\alpha| \leq M_0$.
Let $a(\xi)$ be the symbol of the constant-coefficient differential operator $A(x,D)$, i.e.,
\beq
A=A(x,D)=\sum\limits_{|\alpha|=m}\frac{1}{\alpha!}\partial^{\alpha}_{\xi}a(0)D^{\alpha}.
\eeq
Define $\tilde{M}_j:=\max\{0, M_j-m\}$ for $j \in \{0,1\}$.
A straightforward calculation gives:
\beq
\begin{split}
A(x,D)v(x)&=i^{m}\sum\limits_{|\alpha|=m}\sum\limits_{|\beta|=\tilde{M}_1}^{\tilde{M}_0}\frac{1}{\alpha!}\frac{1}{\beta!}\partial^{\alpha}_{\xi}a(0)\partial^{\alpha+\beta}_x u(0)\cdot x^{\beta}\\
&=\sum\limits_{|\beta|=\tilde{M}_1}^{\tilde{M}_0}\frac{1}{\beta!}\left(\sum\limits_{|\alpha|=m}\frac{1}{\alpha!}\partial^{\alpha}_{\xi}a(0)\cdot D^{\alpha}(\partial^{\beta}_{x}u)(0)\right)\cdot x^{\beta}\\
&=\sum\limits_{|\beta|=\tilde{M}_1}^{\tilde{M}_0}\frac{1}{\beta!}A\partial^{\beta}u(0) \cdot x^{\beta}.
\end{split}
\eeq
Because $\partial^{\beta}\delta(\cdot-0) \in \tilde{L}^-$ when $\tilde{M}_1 \leq |\beta|\leq \tilde{M}_0$, we obtain $A\partial^{\beta}u(0)=\partial^{\beta}Au(0)=0$ for such multi-indices $\beta$. Now we conclude that $Av=0$, which proves our claim.
\qed

\begin{remark}
\indent
\begin{enumerate}[(i)]
\item
If $d>m$ in Proposition \ref{rrlexample2}, then any elliptic real-constant-coefficient homogeneous differential operator $A$ of order $m$ on $\mathbb{R}^d$ satisfies Assumption $\mathcal{A}$. Notice that $m$ must be even. Since $F_{A, \mathbb{R}}=\{0\}$, it is not hard to see from Theorem \ref{LiouvilleDim} that
\beq
\dim V^p_N(A)=\dim V^p_N((-\Delta)^{m/2}).
\eeq
 In particular, if $\mu$ is the point divisor $x_0^{-(M_0+1)}$, the Liouville-Riemann-Roch formula becomes
\beq
\dim L_p(\mu, A,N)=\left\{
  \begin{array}{@{}ll@{}}
    h_{d,[N]}^{(m)}-h_{d,M_0}^{(m)} & \text{if}\ p=\infty. \\
    h_{d,\floor{N-d/p}}^{(m)}-h_{d,M_0}^{(m)} & \text{otherwise,}
  \end{array}\right.
\eeq
though this also has an elementary proof.

Here for $A, B, C \in \mathbb{N}$, we denote by $h_{A,B}^{(C)}$ the quantity $\binom{A+B}{A}-\binom{A+B-C}{A}$, where we adopt the agreement in Definition \ref{D:neg_binom}.
\item
As a special case of Proposition \ref{rrlexample2}, the Liouville-Riemann-Roch equality for the Laplacian operator could occur when $\mu^-$ is non-trivial (compare with Theorem \ref{RRLeq}). As we have seen, the corresponding spaces $L_{p'}(\mu^{-1}, -\Delta, -N)$ in this situation are trivial.

It is worth mentioning that it is possible to obtain the Liouville-Riemann-Roch equality in certain cases when the dimensions of the spaces $L_{p'}(\mu^{-1}, -\Delta, -N)$ can be arbitrarily large. For instance, let $p=\infty$ and $r\geq N+3$, we define
\beq
\mu:=(\emptyset,0; D^-, L^-)\mbox{ with }D^-=\{0\}
\eeq
and
\beq
L^-=\Span_{\mathbb{C}}\{\partial^{\alpha}\delta(\cdot-0)\}_{|\alpha|=r }.
\eeq
Then clearly $L_{\infty}(\mu, -\Delta,N)=V^{\infty}_N(-\Delta)$. From the proof of the second part of Proposition \ref{triviality},
\beq
\begin{split}
L_1(\mu^{-1},-\Delta,-N)&=\Span_{\mathbb{C}}\{\partial^{\alpha}(|x|^{2-d})\}_{|\alpha|=r}\\
&=\{u \in C^{\infty}(\mathbb{R}^d \setminus \{0\}) \mid -\Delta u \in L^-, \lim_{|x| \rightarrow \infty}|u(x)|=0\}.
\end{split}
\eeq
By Theorem \ref{RR}, it is easy to see that the dimension of this space is equal to the degree of the divisor $\mu^{-1}$ (see also Example \ref{mainex}). Thus,
\beq
\dim L_{\infty}(\mu,-\Delta,N)=\dim L_1(\mu^{-1}, -\Delta,-N)+\deg_{-\Delta}(\mu)+\dim V^{\infty}_N(-\Delta)
\eeq
and as $ r \rightarrow \infty$,
\beq
\dim L_1(\mu^{-1}, -\Delta,-N)=\binom{d+r-1}{d-1}-\binom{d+r-3}{d-1}\rightarrow \infty.
\eeq
\end{enumerate}
\end{remark}

\subsection{Proof of Corollary \ref{uppersemicont}}
In a similar manner to the proof of Corollary \ref{solvability}, for the rigged divisor $\mu_z$ the corresponding extension operator
\beq
\widetilde{(A_z)^{2}_{m,0}}: \Gamma(X,\mu_z,(A_z)^{2}_{m,0}) \rightarrow \tilde{\Gamma}_{\mu_z}(X,(A_z)^{2}_{m,0})
\eeq
is Fredholm. As in the proof of \cite[Theorem 2]{WangGS}, we can deduce the upper-semicontinuity of $\dim \Ker\widetilde{(A_z)^{2}_{m,0}}$ by using \cite[Theorem 1]{WangGS} and \cite[Theorem 3]{WangGS}. Since $\Ker\widetilde{(A_z)^{2}_{m,0}}=L_2(\mu_z, A_z,0)$, this finishes our proof. The upper-semicontinuity of $z \mapsto \dim L_2(\mu_z^{-1}, A_z^*,0)$ is proved similarly.
\qed
\section{Proof of Theorem 3.23}
The key of the proof is the following statement:
\begin{lemma}
\label{RRblochschnol}
Let us consider $p_1, p_2 \in [1, \infty]$ and two positive functions $\varphi_1$ and $\varphi_2$ in $\mathcal{S}(G)$ such that we assume either one of the following two conditions:
\begin{itemize}
\item
$p_1^{-1}+p_2^{-1} \geq 1$ and $\varphi_1 \varphi_2$ is bounded on $G$.
\item
$p_1^{-1}+p_2^{-1} \leq 1$ and $\varphi_1^{-1}\varphi_2^{-1}$ is bounded on $G$.
\end{itemize}
Then the Riemann-Roch formula holds:
\beq
\dim L_{p_1}(\mu, P, \varphi_1)= \deg_{P}(\mu)+\dim L_{p_2}(\mu^{-1}, P^*, \varphi_2),
\eeq
where $\mu$ is any rigged divisor on $\mathcal{X}$.
\end{lemma}

\textbf{Proof of Lemma \ref{RRblochschnol}.}
We follow the strategy of the proof of Theorem \ref{RRLineq}.
As in Definition \ref{apmn}, for each $s \in \mathbb{R}$, $\varphi \in \mathcal{S}(G)$ and $p \in [1,\infty]$, let us introduce the following space
\beq
\mathcal{V}^{p}_{s, \varphi}(\mathcal{X}):=\{u \in C^{\infty}(\mathcal{X}) \mid \{\|u\|_{H^s(g\cF)}\cdot \varphi(g)\}_{g \in G} \in \ell^p(G)\}.
\eeq
and we denote by $P^{p}_{m, \varphi}$ the operator $P$ with the domain
\beq
\Dom P^{p}_{m, \varphi}:=\{u \in \mathcal{V}^{p}_{m, \varphi}(\mathcal{X}) \mid Pu \in C^{\infty}_c(\mathcal{X})\}.
\eeq
For the elliptic differential operator $P^*$, we also use the corresponding notations $(P^*)^{p}_{m, \varphi}$ and $\Dom (P^*)^{p}_{m, \varphi}$.

Now let us fix a pair of two real numbers $(p_1, p_2)$ and a pair of two functions $(\varphi_1, \varphi_2)$ satisfying the conditions of Lemma \ref{RRblochschnol}.
From now on, we will consider the operator $P^{p_1}_{m,\varphi_1}$ and its ``adjoint''  $(P^{p_1}_{m, \varphi_1})^*:=(P^*)^{p_2}_{m, \varphi_2}$.
As before, we define $\Domp P^{p_1}_{m, \varphi_1}=\Domp (P^*)^{p_2}_{m, \varphi_2}=C^{\infty}_c(\mathcal{X})$.

Our goal is to verify the assumptions of Theorem \ref{RR} for the operator $P^{p_1}_{m,\varphi_1}$ and its adjoint $(P^*)^{p_2}_{m,\varphi_2}$. The proof also goes through four steps as in Theorem \ref{RRLineq}. We consider two cases.

\textbf{Case 1.}
$p_1^{-1}+p_2^{-1} \geq 1, \varphi_1 \varphi_2 \lesssim 1$.

The proof of Step 1 stays exactly the same as before (see Remark \ref{schauderuniform}).

For Step 2, the first three properties $(\mathcal{P}1)-(\mathcal{P}3)$ are immediate. For the property $(\mathcal{P}4)$, we want to show that whenever $u \in \Dom P^{p_1}_{m, \varphi_1}$ and $v \in \Dom (P^*)^{p_2}_{m, \varphi_2}$, then
\beq
\label{P4schnol}
\langle Pu, v\rangle=\langle u, P^*v\rangle.
\eeq
Because $P$ and $P^*$ are $C^{\infty}$-bounded, we can repeat the approximation procedure and obtain similar estimates from the proof of Theorem \ref{RRLineq} for showing the identity \mref{P4schnol} whenever $u \in \Dom P^{p_1}_{m, \varphi_1}$ and $v \in \Dom (P^*)^{p_1'}_{m, \varphi_1^{-1}}$.
On the other hand, $\mathcal{V}^{p_2}_{m, \varphi_2}(\mathcal{X}) \subseteq \mathcal{V}^{p_1'}_{m, \varphi_1^{-1}}(\mathcal{X})$ and hence, $\Dom (P^*)^{p_2}_{m, \varphi_2} \subseteq \Dom (P^*)^{p_1'}_{m, \varphi_1^{-1}}$. With this inclusion, it is enough to conclude the property $(\mathcal{P}4)$ in this case, which finishes Step 2.

For Step 3, first, it is clear that the kernels of $P^{p_1}_{m,\varphi_1}$ and $(P^*)^{p_2}_{m,\varphi_2}$ are both trivial since $P$ and $P^*$ satisfy (SSP) (see Theorem \ref{SSP}).
So the rest is to verify the Fredholm property of both operators $P^{p_1}_{m,\varphi_1}$ and $(P^*)^{p_2}_{m,\varphi_2}$, i.e., to prove that
\beq
\label{imageschnol}
\Image P^{p_1}_{m,\varphi_1}=\Image (P^*)^{p_2}_{m,\varphi_2}=C^{\infty}_c(\mathcal{X}).
\eeq
Let us prove \mref{imageschnol} for the operator $P^{p_1}_{m,\varphi_1}$ since the other identity is proved similarly.
We denote by $G_P(x,y)$ the Green's function of $P$ at the level $\lambda=0$, i.e., $G_p(x,y)$ is the Schwartz kernel of the resolvent operator $P^{-1}$.
It is known that $G_P(x,y)\in C^{\infty}(\mathcal{X} \times \mathcal{X} \setminus \Delta)$, where $\Delta=\{(x,x) \mid x \in \mathcal{X}\}$. Moreover, all of its derivatives have exponential decay off the diagonal (see \cite[Theorem 2.2]{Shubin_spectral}).
However, it is more convenient for us to use its $L^2$-norm version, i.e., \cite[Theorem 2.3]{Shubin_spectral}: there exists $\varepsilon>0$ such that for every $\delta>0$ and every multi-indices $\alpha, \beta$, one can find a constant $C_{\alpha\beta\delta}>0$ such that
\beq
\label{expdecay}
\int\limits_{x: d_{\mathcal{X}}(x,y) \geq \delta}|\partial^{\alpha}_x\partial^{\beta}_y G_P(x,y)|^2 \exp{(\varepsilon d_{\mathcal{X}}(x,y))}\di\mu_{\mathcal{X}}(x) \leq C_{\alpha \beta \delta}.
\eeq
Here the derivatives $\partial^{\alpha}_x, \partial^{\beta}_y$ are taken with respect to canonical coordinates and the constants $C_{\alpha \beta \delta}$ do not depend on the choice of such canonical coordinates. Note that these estimates \mref{expdecay} still work in the case of exponential growth of the volume of the balls on $\mathcal{X}$.
Let $f \in C^{\infty}_c(\mathcal{X})$ and $K$ be its compact support in $\mathcal{X}$.
We introduce
\beq
u(x):=P^{-1}f(x)=\int_{\mathcal{X}}G_P(x,y)f(y)\di\mu_{\mathcal{X}}(y),
\eeq
where $\mu_{\mathcal{X}}$ is the Riemannian measure on $\mathcal{X}$.
Thus $u \in L^2(\mathcal{X})$, since $P^{-1}$ is a bounded operator on $L^2(\mathcal{X})$.
It is clear that $u$ is a weak solution of the equation $Pu=f$, and hence, by regularity, $u$ is a smooth solution.
We only need to prove that $u \in \mathcal{V}^{1}_{m, \varphi_1}(\mathcal{X})\subseteq \mathcal{V}^{p_1}_{m, \varphi_1}(\mathcal{X})$.
Let us consider any $g$ in $G_{\bar{\cF}, K}:=\{g \in G \mid \dist{(g\bar{\cF},K)}>1\}$. Since $\mathcal{X}$ is quasi-isometric to the metric space $(G,d_S)$ via the orbit action by the \v{S}varc-Milnor lemma, it is not hard to see that there are constants $C_1, C_2>0$ such that for every $g \in G_{\bar{\cF}, K}$ and every $(x,y) \in g\cF \times K$, one has
\beq
\label{quasi-iso}
2C_1|g|-C_2 \leq d_{\mathcal{X}}(x,y) \leq (2C_1)^{-1}|g|+C_2.
\eeq
Taking $\delta=1$, we can find $\varepsilon>0$ so that the decay estimates \mref{expdecay} are satisfied. Now using H\"older's inequality, \mref{expdecay}, and \mref{quasi-iso}, we derive
\beq
\label{expdecay2}
\begin{split}
&\|u\|_{H^m(g\cF)} \lesssim \sup_{g \in G_{\bar{\cF}, K}}\max_{|\alpha| \leq m}\left(\int_{g\cF}\left|\int_{K} |\partial^{\alpha}_x G_P(x,y)|\cdot |f(y)|\di\mu_{\mathcal{X}}(y)\right|^2\di{\mu_{\mathcal{X}}(x)}\right)^{1/2}\\
& \lesssim \|f\|_{L^2(\mathcal{X})}\cdot \max_{|\alpha| \leq m}\left(\int_{g\cF}\int_{K} |\partial^{\alpha}_x G_P(x,y)|^2\di\mu_{\mathcal{X}}(y)\di{\mu_{\mathcal{X}}(x)}\right)^{1/2}\\
& \lesssim \|f\|_{L^2(\mathcal{X})}\cdot  \exp{(-2C_1\varepsilon |g|})\max_{|\alpha| \leq m} \sup_{y \in K}\left(\int_{g\cF} |\partial^{\alpha}_x G_P(x,y)|^2 \exp{(\varepsilon d_{\mathcal{X}}(x,y))}\di{\mu_{\mathcal{X}}(x)})\right)^{1/2}\\
&\lesssim \|f\|_{L^2(\mathcal{X})}\cdot \exp{(-2C_1\varepsilon |g|})\lesssim \|f\|_{L^2(\mathcal{X})}\cdot \varphi_1(g) \cdot \exp{(-C_1\varepsilon |g|}).
\end{split}
\eeq
Note that the above estimates hold up to multiplicative constants, which are uniform with respect to $g \in G_{\bar{\cF}, K}$.
Therefore, $u$ belongs to $\mathcal{V}^{p_1}_{m,\varphi_1}(\mathcal{X})$, which proves \mref{imageschnol}. In particular, the Fredholm indices of the operators $P^{p_1}_{m, \varphi_1}$ and $(P^*)^{p_2}_{m, \varphi_2}$ vanish. Now we are able to apply Theorem \ref{RR} to finish the proof of Lemma \ref{RRblochschnol} in this case.
\hspace{3pt}

\textbf{Case 2.}
$p_1^{-1}+p_2^{-1} \leq 1, \varphi_1^{-1} \varphi_2^{-1} \lesssim 1$. Consider a rigged divisor $\mu$ on $\mathcal{X}$.
By assumptions, $L_{p_2'}(\mu, P, \varphi_2^{-1}) \subseteq L_{p_1}(\mu, P, \varphi_1)$ and
\beq
L_{p_1'}(\mu^{-1}, P^*, \varphi_1^{-1}) \subseteq L_{p_2}(\mu^{-1}, P^*, \varphi_2).
\eeq
From these inclusions and Case 1, we get
\beq
\label{RRcase2}
\begin{split}
\dim L_{p_1}(\mu, P, \varphi_1)&=\deg_P(\mu)+\dim L_{p_1'}(\mu^{-1}, P^*, \varphi_1^{-1}) \leq \deg_P(\mu)+L_{p_2}(\mu^{-1}, P^*, \varphi_2)\\
&=\dim L_{p_2'}(\mu, P, \varphi_2^{-1}) \leq \dim L_{p_1}(\mu, P, \varphi_1).
\end{split}
\eeq
Since all of the above inequalities must become equalities, this yields the corresponding Riemann-Roch formula in this case. \textbf{This finishes the proof of Lemma  \ref{RRblochschnol}
}

We now use this Lemma to prove all of the required statements.

First, one can get the identity in the second statement of Theorem \ref{RRLbloch} by taking $p_1=p_2=\infty$ and $\varphi_1=\varphi_2=\varphi_0$ in Lemma \ref{RRblochschnol}. Also, due to Theorem \ref{SSP}, there is no non-zero solution of $P^*$ with subexponential growth. This implies that if $\mu^{-1}=(\emptyset, 0; D^+, L^+)$, the space $L_{\infty}(\mu^{-1}, P^*, \varphi_0)$ is trivial. Thus, the third statement follows immediately from the second statement. For the first statement, let us consider $p \in [1,\infty]$ and a function $\varphi \in \mathcal{S}(G)$.
Now from Lemma \ref{RRblochschnol} and the second statement, one gets:
\begin{itemize}
\item
If $\varphi$ is bounded,
\beq
\label{RRblocheq1}
\dim L_{1}(\mu, P, \varphi)= \deg_{P}(\mu)+\dim L_{\infty}(\mu^{-1}, P^*, \varphi_0)=\dim L_{\infty}(\mu, P, \varphi_0).
\eeq
\item
If $\varphi^{-1}$ is bounded and $1 \leq p \leq \infty$,
\beq
\label{RRblocheq2}
\dim L_{p}(\mu, P, \varphi)= \deg_{P}(\mu)+\dim L_{\infty}(\mu^{-1}, P^*, \varphi_0)=\dim L_{\infty}(\mu, P, \varphi_0).
\eeq
\end{itemize}
We consider three cases.
\begin{enumerate}[\mbox{Case} 1.]
\item
If $\varphi$ is bounded, the two spaces $L_{1}(\mu, P, \varphi)$ and $L_{\infty}(\mu, P, \varphi_0)$ are the same since their dimensions are equal to each other by \mref{RRblocheq1}. Moreover,
\beq
L_{1}(\mu, P, \varphi) \subseteq L_{p}(\mu, P, \varphi) \subseteq L_{p}(\mu, P, \varphi_0) \subseteq L_{\infty}(\mu, P, \varphi_0).
\eeq
This means that all these spaces are the same.
\item
 If $\varphi^{-1}$ is bounded, $L_{\infty}(\mu, P, \varphi_0) \subseteq L_{\infty}(\mu, P, \varphi)$. Using \mref{RRblocheq2} with $p=\infty$, we have
 $L_{\infty}(\mu, P, \varphi_0)=L_{\infty}(\mu, P, \varphi)$.
 Moreover, \mref{RRblocheq2} also yields that all the spaces $L_{p}(\mu, P, \varphi)$, where $1 \leq p \leq \infty$, must have the same dimension and therefore, they are the same space, which coincides with $L_{\infty}(\mu, P, \varphi_0)$.
\item
If neither $\varphi$ nor $\varphi^{-1}$ is bounded, we can consider the function $\phi:=\varphi+\varphi^{-1}$. Clearly, $\phi$ is in $\mathcal{S}(G)$ and
$\phi \geq 2$.
Then according to Case 1 and Case 2,
\beq
L_p(\mu, P, \phi)=L_{\infty}(\mu, P, \varphi_0)=L_p(\mu, P, \phi^{-1}).
\eeq
Also,
\beq
L_p(\mu, P, \phi^{-1}) \subseteq L_p(\mu, P, \varphi) \subseteq L_p(\mu, P, \phi),
\eeq
since $\phi^{-1} \leq \varphi \leq \phi$.
This means that $L_p(\mu, P, \varphi)=L_{\infty}(\mu, P, \varphi_0)$.
\end{enumerate}
\qed

We can now prove the Corollary \ref{subexp-solvability}.

As in the proof of Corollary \ref{solvability}, the equivalence of the first three statements is an easy consequence of Theorem \ref{RRLbloch} and Remark \ref{RRfredholm}. It is obvious that (iv) implies (iii). To see the converse, one can repeat the argument in the proof of Lemma \ref{RRblochschnol} to show that the solution $u=P^{-1}f$ has exponential decay due to \mref{expdecay}. By the unique solvability of the equation $Pu=f$ in $L^2(\mathcal{X})$, (iii) implies (iv).
\qed

\chapter[Examples]{Specific examples of Liouville-Riemann-Roch theorems}
\label{C:applications-LRR}

In this Chapter, we look at some examples of applications of the results of Chapter \ref{C:main}.

\section{Self-adjoint operators}
Let $A$ is a bounded from below self-adjoint periodic elliptic operator of order $m$ on an abelian co-compact covering $X$. We start with a brief reminder of some notions from Chapter \ref{C:preliminaries-LRR}.

For any real quasimomentum $k$, the operator $A(k)$ is self-adjoint, and its spectrum is discrete, consisting of real eigenvalues of finite multiplicities, which can be listed in non-decreasing order as
\begin{equation}
\label{eigenv}
\lambda_{1}(k) \leq \lambda_{2}(k) \leq \ldots \nearrow \infty.
\end{equation}
For each $j \in \mathbb{N}$, the function $k \mapsto \lambda_{j}(k)$ is called the \textbf{$j^{th}$ band function}. It is known (see e.g., \cite{Wilcox}) that band functions $\lambda_j$ are continuous, $G^*$-periodic and piecewise analytic in $k$. It is more convenient to consider the band functions as functions on the torus $\mathbb{T}^d$.

The range $I_j$ of the $j^{th}$ band function is called the \textbf{$j^{th}$-band}.
The bands $I_j$ can touch or overlap, but sometimes they may leave an open gap.
According to Theorem \ref{specA},
\beq
\sigma(A)=\bigcup\limits_{j}I_j.
\eeq
Therefore, the spectrum of a self-adjoint periodic elliptic operator $A$ has a \textbf{band-gap structure}. An endpoint of a spectral gap is called a \textbf{gap edge} (or a \textbf{spectral edge}).
\begin{figure}[H]
\begin{center}
\begin{tikzpicture}[xscale=1,yscale=1]
\draw[thick, ->] (0,0) -- (8,0) node[below] {$k$};
\draw[thick, ->] (0,0) -- (0,6.5) node[left] {$\lambda$};
\draw[ultra thick, red] (0, 0.6) -- (0, 2.5);
\draw[ultra thick, red] (0, 4) -- (0, 6);
\draw[thick, ->] (-1.65, 3) node[left] {$\sigma(A)$} -- (-0.1, 5);
\draw[thick, ->] (-1.65, 3) -- (-0.1, 1.5);
\draw[dashed, thick] (0, 6) -- (8, 6);
\draw[dashed, thick] (0, 4) -- (8, 4);
\draw[dashed, thick] (0, 4.85) -- (2, 4.85);
\draw[very thick] (0, 4) .. controls (3,5.25) .. (8, 6) node[right] {$\lambda_3$};
\draw[very thick] (0,6) .. controls (3,4) and (5,4) .. (8,4.05) node[right] {$\lambda_2$};
\draw (0.25, 4.1) node[above] {$I_2$};
\draw (0.25, 5) node[above] {$I_3$};
\draw[dashed, thick] (0,2.5) -- (4, 2.5);
\draw[thick, decorate,decoration={brace,amplitude=6pt},xshift=0pt,yshift=0pt] (-0.07,2.55) -- (-0.07,3.95);
\draw (-0.7,3.5) node [below]{$\text{Gap}$};
\draw[very thick] (0,0.6) to[out=30,in=180] (4,2.48);
\draw[very thick] (4,2.48) to[out=0,in=160] (8,1) node[right] {$\lambda_1$};
\draw (0.25, 1.4) node[above] {$I_1$};
\draw[fill, blue] (0,2.5) circle (2pt);
\draw[fill, blue] (0,4) circle (2pt);
\draw[thick, ->] (1.5, 3.2) node[right] {$\text{Gap Edges}$} -- (0.05, 3.9);
\draw[thick, ->] (1.5, 3.2) -- (0.06, 2.6);
\end{tikzpicture}
\caption{An example of $\sigma(A)$.}
\end{center}
\end{figure}
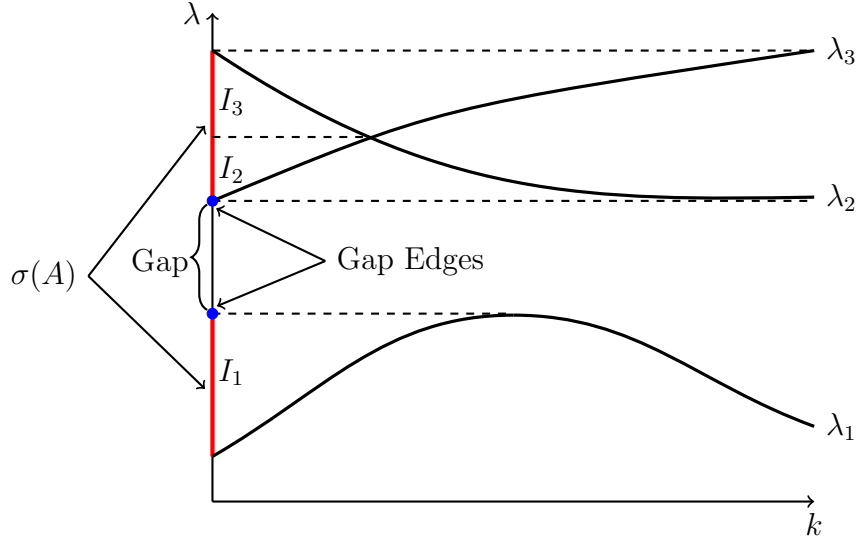
To apply the results from Section \ref{C:main}, we will reformulate \textit{Assumption $\mathcal{A}$} from Section \ref{C:main}. For the relevant notations the reader is referred to Chapter \ref{C:preliminaries-LRR}.

\textit{Assumption $\mathcal{A}\cprime$}

\textit{Suppose that the Fermi surface $F_{A,\mathbb{R}}$ is finite and consists of the points $\{k_1, \cdots, k_{\ell}\}$ (modulo $G^*$-shifts). Let
\beq
\{\lambda_{r,j}\}_{j=\overline{1,m_{r}}}
\eeq
be the set of dispersion branches that are equal to $0$ at the quasimomentum $k_{r}$ $(1 \leq r \leq \ell)$.
There exists a family of \textit{pairwise disjoint} neighborhoods $V_{r}$ of $k_{r}$ such that the function
\beq
k \in V_{r} \mapsto \max_{1 \leq j \leq m_{r}}|\lambda_{r,j}(k)|^{-1}
\eeq
is $L^1$-integrable.}\footnote{Note that for each $k \in V_{r} \setminus \{k_{r}\}$, $\lambda_{r,j}(k) \neq 0$ since $V_{r} \cap F_{A, \mathbb{R}}=\{k_{r}\}$.}

Clearly, in the self-adjoint case the \textit{Assumption $\mathcal{A}$}  and \textit{Assumption $\mathcal{A}\cprime$} are equivalent.

\begin{notation}
Given a natural number $N$,
\begin{itemize}
\label{binom}
\item
we denote by $h_{d,N}$ the dimension of the space of all harmonic polynomials of order at most $N$ in $d$-variables, i.e.,
\beq
h_{d,N}:=\dim V^{\infty}_N(-\Delta_{\mathbb{R}^d})=\binom{d+N}{d}-\binom{d+N-2}{d};
\eeq
\item
we also denote by $c_{d,N}$ the dimension of the space of all homogeneous polynomials of degree $N$ in $d$ variables, i.e.,
\beq
c_{d,N}:=\binom{d+N}{d}-\binom{d+N-1}{d}=\binom{d+N-1}{N}.
\eeq
\end{itemize}
\end{notation}

\subsection[Non-degenerate edges]{Periodic operators with non-degenerate spectral edges}

Let $\lambda$ be an energy level that coincides with one of the gap edges in the spectrum of $A$.
By shifting the spectrum, we can assume that $\lambda=0$.

As it is discussed in Chapter \ref{C:preliminaries-LRR}, one can expect that the Fermi surfaces of $A$ at the spectral edges normally are finite, and hence Liouville type results are applicable in these situations.
We make the following assumption.

\textit{Assumption $\mathcal{B}$.}
There exists a band function $\lambda_j(k)$ such that for each quasimomentum $k_r$ in the real Fermi surface $F_{A,\mathbb{R}}$, one has
\begin{enumerate}[($\mathcal{B}$1)]
\item
$0$ is a simple eigenvalue of the operator $A(k_r)$.
\item
The Hessian matrix $\Hess{\lambda_j}(k_0)$ is non-degenerate.
\end{enumerate}

As it has been mentioned before, it is commonly believed in mathematics and physics literature (see e.g, \cite{Ksurvey}) that generically (with respect to the coefficients and other free parameters of the operator), extrema of band functions for second order operators of mathematical physics are isolated, attained by a single band and have non-degenerate Hessians, and thus satisfy the above assumption.

Suppose now that the free abelian rank $d$ of the deck group $G$ is greater than $2$. The non-degeneracy assumption ($\mathcal{B}2$) implies the integrability of function $|\lambda_j(k)|^{-1}$ over a small neighborhood of $F_{A, \mathbb{R}}$.
Hence, \textit{Assumption $\mathcal{A}$} follows from \textit{Assumption $\mathcal{B}$}.

Due to Theorem \ref{LiouvilleDim}, the dimension of the space $V^{\infty}_N(A)$ is equal to $\ell h_{d,[N]}$ (see Notation \ref{binom}).
Applying the results in Section \ref{C:main}, we obtain the following results for a `generic' self-adjoint second-order periodic elliptic operator $A$:
\begin{thm}
\label{generic2nd}
Suppose $d \geq 3$ and $N \in \mathbb{R}$. Let $\mu=(D^+, L^+; D^-, L^-)$ be a rigged divisor on $X$ and, as before, $\mu^+:=(D^+, L^+; \emptyset, 0)$.
\label{generic}
\begin{enumerate}[a.]
\item  \label{generic1}
If $N \geq 0$, then
$$\ell h_{d,[N]}+\deg_{A}(\mu)+\dim L_1(\mu^{-1}, A^*, -N) \leq \dim L_{\infty}(\mu, A, N) \leq \ell h_{d,[N]}+\deg_{A}(\mu^+),$$
and
\beq
\dim L_{\infty}(\mu^+, A, N)=\ell h_{d,[N]}+\deg_{A}(\mu^+).
\eeq
\item \label{generic2}
If $p \in [1, \infty)$ and $N>d/p$, then
$$\ell h_{d, \floor{N-d/p}}+\deg_{A}(\mu)+\dim L_{p'}(\mu^{-1}, A^*, -N) \leq \dim L_{p}(\mu, A, N) \leq \ell h_{d,\floor{N-d/p}}+\deg_{A}(\mu^+),$$
and
\beq
\dim L_{p}(\mu^+, A, N)=\ell h_{d,\floor{N-d/p}} +\deg_{A}(\mu^+).
\eeq
\item  \label{generic4}
For $d \geq 5$, the inequality
\beq
\deg_{A}(\mu)+\dim L_{p'}(\mu^{-1}, A^*, -N) \leq \dim L_{p}(\mu, A, N) \leq \deg_{A}(\mu^+)
\eeq
holds, assuming one of the following two conditions:
\begin{itemize}
\item
$p\in [1,2)$, $N \in (d(2-p)/(2p), d/p]$.
\item
$p \in [2,\infty)$, $N \in [0,d/p]$.
\end{itemize}
\end{enumerate}
\end{thm}

\begin{example}
\end{example}
\begin{enumerate}
\item
Let $A=A^*=-\Delta+V$ be a periodic Schr\"odinger operator with real-valued electric potential $V$ on a co-compact abelian cover $X$. The domain of $A$ is the Sobolev space $H^2(X)$, and thus, $A$ is self-adjoint. Without loss of generality, we can suppose that $0$ is the bottom of its spectrum. It is well-known \cite{KS} that \textit{Assumption $\mathcal{B}$} holds in this situation. Hence, all the conclusions of Theorem \ref{generic2nd} hold with $\ell=1$, since $F_{A, \mathbb{R}}=\{0\}$ (modulo $2\pi\bZ^d$-shifts).
\item
In \cite{BS_2001, BS_2003}, a deep analysis of the dispersion curves at the bottom of the spectrum was developed for a wide class of periodic elliptic operators of second-order on $\mathbb{R}^d$. Namely, let $\Gamma$ be a lattice in $\mathbb{R}^d$ and $\Gamma^*$ be its dual lattice, then these operators admit the following \textit{regular factorization}:
$$A=\overline{f(x)}b(D)^*g(x)b(D)f(x),$$
where $b(D)=\sum\limits_{j=1}^d -i\partial_{x_j}b_j: L^2(\mathbb{R}^d, \mathbb{C}^n) \rightarrow L^2(\mathbb{R}^d, \mathbb{C}^m)$ is a linear homogeneous differential operator whose coefficients $b_j$ are constant $m \times n$-matrices of rank $n$ (here $m \geq n$),  $f$ is a $\Gamma$-periodic and invertible $n \times n$ matrix function such that $f$ and $f^{-1}$ are in $L^{\infty}(\mathbb{R}^d)$, and $g$ is a $\Gamma$-periodic and positive definite $m \times m$-matrix function such that for some constants $0<c_0 \leq c_1$, $c_0 \textbf{1}_{m \times m} \leq g(x) \leq c_1 \textbf{1}_{m \times m},  x \in \mathbb{R}^d$. The existence of this factorization implies that the first band function attains a simple and nondegenerate minimum with value $0$ at \textit{only} the quasimomentum $k=0$ (modulo $\Gamma^*$-shifts).

This covers the previous example since, it was noted from \cite{BS_2001} that the periodic Schr\"odinger operator $D(g(x)D)+V(x)$ with a periodic metric $g(x)$ and a periodic potential $V(x)$ can be reformulated properly to admit a regular factorization.
\item
Consider a self-adjoint periodic magnetic Schr\"odinger operator in $\mathbb{R}^n$ ($n>2$)
$$H=(-i\nabla+A(x))^2+V(x),$$
where $A(x)$ and $V(x)$ are real-valued periodic magnetic and electric potentials, respectively.
Using a gauge transformation, we can always assume w.l.o.g. the following normalized condition on $A$:
$$\int\limits_{\mathbb{T}^n}A(x)\di{x}=0.$$
Note that the transpose of $H$ is the magnetic Schr\"odinger operator
$$H^*=(-i\nabla-A(x))^2+V(x).$$
From the discussion of \cite{KP2} (see also \cite[Theorem 3.1.7]{Kbook}), the lowest band function of $H$ has a unique nondegenerate extremum at a single quasimomentum $k_0$ (modulo $G^*$-shifts) if the magnetic potential $A$ is small enough, e.g., $\|A\|_{L^r(\mathbb{T}^n)}\ll 1$ or some $r>n$. Thus, we obtain the same conclusion as the case without magnetic potential. It is crucial that one has to assume the smallness of the magnetic potential since there are examples \cite{Shterenberg} showing that the bottom of the spectrum can be degenerate if the magnetic potential is large enough.
\end{enumerate}
We end this part by providing an illustration of Corollary \ref{uppersemicont} for families of periodic elliptic operators with non-degenerate spectral edges.
\begin{cor}
Let $A_0$ be a periodic elliptic operator on a co-compact abelian covering $X$ with the deck group $G=\mathbb{Z}^d$, where $d \geq 5$, such that its real Fermi surface $F_{A_0,\mathbb{R}}$ consists of finitely many simple and non-degenerate minima of the $j^{th}$-band function of $A_0$ ($j \geq 1$).
Let $B$ be a symmetric and periodic differential operator on $X$ such that $B$ is $A_0$-bounded. We consider the perturbation $A_z=A_0+zB$, $z \in \mathbb{R}$.
By standard perturbation theory, there exists a continuous function $\lambda(z,k)$ defined for small $z$ and all quasimomenta $k$ such that $k \mapsto \lambda(z,k)$ is the $j^{th}$ band function of $A_z$ and $\lambda_j(z,k)$ is analytic in $z$. Let $\lambda_z$ be the minimum value of the band function $\lambda(z,k)$ of the perturbation $A_z$.
Then for any rigged divisor $\mu$, there exists $\varepsilon>0$ such that
\begin{equation*}
\dim L_2(\mu, A_z-\lambda_z, 0) \leq \dim L_2(\mu, A_0, 0),
\end{equation*}
for any $z$ satisfying $|z|<\varepsilon$.
\end{cor}
The corresponding statement for non-degenerate maxima also holds.
\subsection{Periodic operators with Dirac points}
An important situation in solid state physics and material sciences is when two branches of the dispersion relation touch, forming a conical junction point, which is called a Dirac (or sometimes  ``diabolic'') point. Two-dimensional massless Dirac operators or $2D$-Schr\"odinger operators with honeycomb-symmetric potentials are typical models of periodic operators with conical structures \cite{ComBerk, FefWei, Gru, KPcarbon}. Presence of such a point is the reason of the miraculous properties of graphene and some other carbon allotropes.

Let us consider a self-adjoint periodic elliptic operator $A$ such that there are two branches $\lambda_+$ and $\lambda_-$ of the dispersion relation of $A$ that meet at $\lambda=0, k=k_r$,
forming a Dirac cone. Equivalently, we can assume that locally around each quasimomentum $k_r$ in the Fermi surface $F_{A, \mathbb{R}}$, for some $c_{r} \neq 0$, one has
\begin{equation*}
\begin{split}
\lambda_+(k)&=c_{r} |k-k_r|\cdot (1+O(|k-k_r|),\\
\lambda_-(k)&=-c_{r} |k-k_r|\cdot (1+O(|k-k_r|).
\end{split}
\end{equation*}
It is immediate to see that the functions $|\lambda_+|^{-1}$ and $|\lambda_-|^{-1}$ are integrable over a small neighborhood of $F_{A, \mathbb{R}}$ provided that $d>1$. Hence, we conclude:
\begin{thm}
Suppose $d \geq 2$.
Assume that the Fermi surface $F_{A, \mathbb{R}}$ consists of $\ell$ Dirac conical points. Then, in Notations \ref{binom}, as in Theorem \ref{generic}, we have
\label{generic}
\begin{enumerate}[a.]
\item
If $N \geq 0$, then
\beq
2\ell c_{d,[N]}+\deg_{A}(\mu)+\dim L_1(\mu^{-1}, A^*, -N) \leq \dim L_{\infty}(\mu, A, N) \leq 2\ell c_{d,[N]}+\deg_{A}(\mu^+)
\eeq
and
\beq
\dim L_{\infty}(\mu^+, A, N)=2\ell c_{d,[N]}+\deg_{A}(\mu^+).
\eeq
\item
If $p \in [1, \infty)$ and $N>d/p$, then
\beq
\begin{split}
2\ell c_{d, \floor{N-d/p}}+\deg_{A}(\mu)+\dim L_{p'}(\mu^{-1}, A^*, -N) \leq \dim L_{p}(\mu, A, N)\\ \leq 2\ell c_{d,\floor{N-d/p}}+\deg_{A}(\mu^+)
\end{split}
\eeq
and
\beq
\dim L_{p}(\mu^+, A, N)=2\ell c_{d,\floor{N-d/p}}+\deg_{A}(\mu^+).
\eeq
\item
If $d \geq 3$ and a pair $(p,N)$ satisfy the condition c.of Theorem \ref{generic}, the conclusion of Theorem \ref{generic} \mref{generic4} also holds.
\end{enumerate}
\end{thm}
\begin{example}
\end{example}
We consider here Schr\"odinger operators with honeycomb lattice potentials in $\mathbb{R}^2$.

Let us start with recalling briefly some notions from \cite{FefWei,ComBerk}.
The triangular lattice $\Lambda_h=\mathbb{Z}\textbf{v}_1 \oplus \mathbb{Z}\textbf{v}_2$ is spanned by the basis vectors:
\beq
\textbf{v}_1=a\left(\frac{\sqrt{3}}{2}, \frac{1}{2}\right)^t, \quad \textbf{v}_2=a\left(\frac{\sqrt{3}}{2}, -\frac{1}{2}\right)^t (a>0).
\eeq
The dual lattice is
\beq
\Lambda^*_h=\mathbb{Z}\textbf{k}_1 \oplus \mathbb{Z}\textbf{k}_2,
\eeq
where
\beq
\textbf{k}_1=\frac{4\pi}{a\sqrt{3}}\left(\frac{1}{2}, \frac{\sqrt{3}}{2}\right)^t, \quad \textbf{k}_2=\frac{4\pi}{a\sqrt{3}}\left(\frac{1}{2}, -\frac{\sqrt{3}}{2}\right)^t.
\eeq
We define
\beq
\textbf{K}=\frac{1}{3}(\textbf{k}_1-\textbf{k}_2), \quad \textbf{K}\cprime=-\textbf{K}.
\eeq
The Brillouin zone $\mathcal{B}_h$, a fundamental domain of the quotient $\mathbb{R}^2/\Lambda^*_h$, can be chosen as a hexagon in $\mathbb{R}^2$ such that the six vertices of this hexagon fall into two groups:
\begin{enumerate}
\item
$\textbf{K}$ type-vertices: $\textbf{K}$, $\textbf{K}+\textbf{k}_2$, $\textbf{K}-\textbf{k}_1$.
\item
$\textbf{K}\cprime$ type-vertices: $\textbf{K}\cprime$, $\textbf{K}\cprime-\textbf{k}_2$, $\textbf{K}\cprime+\textbf{k}_1$.
\end{enumerate}
Note that these groups of vertices are invariant under the clockwise rotation $\mathcal{R}$ by $2\pi/3$.

A honeycomb lattice potential $V \in C^{\infty}(\mathbb{R}^2)$ is real, $\Lambda_h$-periodic, and there exists a point $x_0 \in \mathbb{R}^2$ such that $V$ is inversion symmetric (i.e., even) and $\mathcal{R}$-invariant with respect to $x_0$ (see \cite[Remark 2.4]{FefWei} and \cite{ComBerk} for constructions of honeycomb lattice potentials). Now assume that $V$ is a honeycomb lattice potential and consider the Schr\"odinger operator $H^{\varepsilon}=-\Delta+\varepsilon V$ ($\varepsilon \in \mathbb{R}$). One of the main results of \cite{FefWei} is that except possibly for $\varepsilon$ in a countable and closed set $\tilde{\mathcal{C}}$, the dispersion relation of $H^{\varepsilon}$ has conical singularities at each vertex of $\mathcal{B}_h$. Assume that $\lambda_j^{\varepsilon}$, $j\in \mathbb{N}$, are the band functions of the operator $H^{\varepsilon}$ for each $\varepsilon \in \mathbb{R}$. Then according to \cite[Theorem 5.1]{FefWei}, when $\varepsilon \notin \tilde{\mathcal{C}}$, there exists some $j \in \mathbb{N}$ such that the Fermi surface $F_{H^{\varepsilon}, \lambda_j^{\varepsilon}(\textbf{K})}$ of the operator $H^{\varepsilon}$ at the level $\lambda_j^{\varepsilon}(\textbf{K})$ contains (at least) two Dirac points located at the quasimomenta $\textbf{K}$ and $\textbf{K}\cprime$ (modulo shifts by vectors in the dual lattice $\Lambda^*_h$). Now our next corollary is a direct consequence of \cite[Theorem 5.1]{FefWei} and our previous discussion:
\begin{cor}
Let $\mu$ be a rigged divisor on $\mathbb{R}^2$ and $V$ be a honeycomb lattice potential such that
\beq
V_{1,1}:=\int\limits_{\mathbb{R}^2/\Lambda_h}e^{-i(\textbf{k}_1+\textbf{k}_2)\cdot x}V(x)\di x\neq 0.
\eeq
Then for $\varepsilon \notin \tilde{\mathcal{C}}$, there exists $j \in \mathbb{N}$ such that the following inequalities hold:
\begin{enumerate}
\item if
$p=\infty, N \geq 0$:
\beq
\dim L_{\infty}(\mu, H^{\varepsilon}-\lambda_j^{\varepsilon}(\textbf{K}), N)\geq 4([N]+1)+\deg_{H^{\varepsilon}-\lambda_j^{\varepsilon}(\textbf{K})}(\mu)+\dim L_1(\mu^{-1}, H^{\varepsilon}-\lambda_j^{\varepsilon}(\textbf{K}), -N) ,
\eeq
\beq
\dim L_{\infty}(\mu^+, H^{\varepsilon}-\lambda_j^{\varepsilon}(\textbf{K}), N) \geq 4([N]+1)+\deg_{H^{\varepsilon}-\lambda_j^{\varepsilon}(\textbf{K})}(\mu^+).
\eeq
\item If
$1 \leq p<\infty, N>2/p$:
\beq
\dim L_{p}(\mu, H^{\varepsilon}-\lambda_j^{\varepsilon}(\textbf{K}), N)\geq 4(\floor{N-2/p}+1)+\deg_{H^{\varepsilon}-\lambda_j^{\varepsilon}(\textbf{K})}(\mu)+\dim L_{p'}(\mu^{-1}, H^{\varepsilon}-\lambda_j^{\varepsilon}(\textbf{K}), -N) ,
\eeq
\beq
\dim L_{p}(\mu^+, H^{\varepsilon}-\lambda_j^{\varepsilon}(\textbf{K}), N) \geq 4(\floor{N-2/p}+1)+\deg_{H^{\varepsilon}-\lambda_j^{\varepsilon}(\textbf{K})}(\mu^+).
\eeq
\end{enumerate}
Moreover, there exists $\varepsilon_0>0$ such that for all $\varepsilon \in (-\varepsilon_0, \varepsilon_0) \setminus \{0\}$, we have
\begin{itemize}
\item If
$\varepsilon V_{1,1}>0$, the above $j$ can be chosen as $j=1$.
\item If
$\varepsilon V_{1,1}<0$, the above $j$ can be chosen as $j=2$.
\end{itemize}
\end{cor}
\begin{remark}
Other results on existence of Dirac points in the dispersion relations of Schr\"odinger operators with periodic potentials on honeycomb lattices are established for instance in \cite{KPcarbon} for quantum graph models of graphene and carbon nanotubes materials and in \cite{ComBerk} for many interesting models including both discrete, quantum graph, and continuous ones.
\end{remark}

\section[Non-self-adjoint operators]{Non-self-adjoint second order elliptic operators}\label{S:nonsa}
We now consider a class of possibly non-self-adjoint second-order elliptic operators arising in probability theory. Let $A$ be a $G$-periodic linear elliptic operator of second-order acting on functions $u$ in $C^{\infty}(X)$ such that in local coordinate system $(U; x_1, \ldots, x_n)$, the operator $A$ can be represented as
\beq
\label{realop}
A=-\sum\limits_{1 \leq i,j \leq n}a_{ij}(x)\partial_{x_i}\partial_{x_j}+\sum_{1 \leq j \leq n}b_j(x)\partial_{x_j}+c(x),
\eeq
where the coefficients $a_{ij}, b_j, c$ are \textit{real}, smooth, and $G$-periodic. The matrix $a(x):=(a_{ij}(x))_{1 \leq i,j \leq n}$ is positive definite.
We notice that the coefficient $c(x)$ of zeroth-order of $A$ is globally defined on $X$, since it is the image of the constant function $1$ via $A$.
\begin{defi}\cite{Agmonperiodic, LinPinchover, Pinsky}
\label{principaleival}
\begin{enumerate}[a.]
\item
A function $u$ on $X$ is called a \textbf{$G$-multiplicative} with exponent $\xi \in \mathbb{R}^d$, if it satisfies
\beq
u(g\cdot x)=e^{\xi \cdot g}u(x), \quad \forall x \in X, g \in G.
\eeq
In other words, $u$ is a Bloch function with quasimomentum $i\xi$.
\item
The \textit{generalized principal eigenvalue} of $A$ is defined by
\beqn
\Lambda_A:=\sup\{\lambda \in \mathbb{R} \mid (A-\lambda)u=0 \hspace{4pt} \mbox{ has a positive solution $u$}\}.
\eeqn
\end{enumerate}
The principal eigenvalue is a generalized version of the bottom of the spectrum in the self-adjoint case (see e.g., \cite{Agmondecay}).
\end{defi}
 Let $A^*$ be the formal adjoint operator to $A$. The generalized principal eigenvalues of $A^*$ and $A$ are equal, i.e., $\Lambda_A=\Lambda_{A^*}$.
For an operator $A$ of the type \mref{realop} and any $\xi \in \mathbb{R}^d$, it is known \cite{Agmonperiodic, Kbook, LinPinchover, Pinsky}  that there exists a unique real number $\Lambda_A(\xi)$ such that the equation $(A-\Lambda_A(\xi))u=0$ has a \textit{positive} $G$-multiplicative solution $u$. We list some known properties of this important function $\Lambda_A(\xi)$. The reader can find proofs in \cite{Agmonperiodic, LinPinchover, Pinsky} (see also \cite[Lemma 5.7]{KP2}).
\begin{prop}
\label{lambda}\indent
\begin{enumerate}[a.]
\item
$\Lambda_A=\max\limits_{\xi \in \mathbb{R}^d} \Lambda_A(\xi)=\Lambda_A(\xi_0)$ for a unique $\xi_0$ in $\mathbb{R}^d$.
\item
The function $\Lambda_A(\xi)$ is strictly concave, real analytic, bounded from above, and its gradient $\nabla \Lambda_A(\xi)$ vanishes at only its unique maximum point $\xi = \xi_0$. The Hessian of the function $\Lambda_A(\xi)$ is non-degenerate at all points.
\item
$\Lambda_A(\xi)$ is the principal eigenvalue with multiplicity one of the operator $A(i\xi)$.
\item
$\Lambda_A \geq 0$ if and only if $A$ admits a positive periodic (super-) solution, which is also equivalent to the existence of a positive $G$-multiplicative solution $u$ to the equation $Au=0$.
\item
$\Lambda_A=0$ if and only if there is exactly one normalized positive solution $u$ to the equation $Au=0$.
\end{enumerate}
\end{prop}
We are interested in studying Liouville-Riemann-Roch type results for such operators $A$ satisfying $\Lambda_A(0) \geq 0$, which implies that $A$ has a positive solution.
\begin{example}
\end{example}
\begin{enumerate}
\item
Any operator $A$ as in (\ref{realop}) without zeroth-order term satisfies $\Lambda_A=\Lambda_A(0)=0$.
\item
If the zeroth-order coefficient $c(x)$ of the operator $A$ is nonnegative on $X$, $\Lambda_A(0)$ is also nonnegative.\footnote{In general, the converse of this statement is not true: e.g., consider $A^*$ in this case then the zeroth-order coefficient of the transpose $A^*$ is not necessarily nonnegative while $\Lambda_{A^*}(0)=\Lambda_A(0) \geq 0$.}
Indeed, let $u$ be a positive and periodic solution to the equation $Au=\Lambda_A(0)u$. If $\Lambda_A(0)<0$, it follows from the equation that $u$ is a positive and periodic subsolution, namely $Au < 0$ on $X$. By the strong maximum principle, $u$ must be constant. This means that $0 \leq cu=Au<0$, which is a contradiction!
\end{enumerate}
Before stating the main result of this subsection, let us provide a key lemma.
\begin{lemma}
\label{dispersionnsa}\indent
\begin{enumerate}[a.]
\item
If $\Lambda_A(0)>0$, then $F_{A,\mathbb{R}}=\emptyset$.
\item
If $\Lambda_A(0)=0$, then $F_{A,\mathbb{R}}=\{0\}$ (modulo $G^*$-shifts). In this case, there exists an open strip $V$ in $\mathbb{C}^d$ containing the imaginary axis $i\mathbb{R}^d$ such that for any $k \in V$, there is exactly one (isolated and nondegenerate eigenvalue) point $\lambda(k)$ in $\sigma(A(k))$ that is close to $0$. The dispersion function $\lambda(k)$ is analytic in $V$ and $\lambda(ik)=\Lambda_A(k)$ if $k \in \mathbb{R}^d$. Moreover,
\begin{itemize}
\item
When $\Lambda_A>0$, $k=0$ is a non-critical point of the dispersion $\lambda(k)$ in $V \cap \mathbb{R}^d$ as well as of the function $\Lambda_A(\cdot)$ in $\mathbb{R}^d$.
\item
When $\Lambda_A=0$, $k=0$ is a non-degenerate extremum of the dispersion $\lambda(k)$ in $V \cap \mathbb{R}^d$ as well as of the function $\Lambda_A(\cdot)$ in $\mathbb{R}^d$.
\end{itemize}
\end{enumerate}
\end{lemma}
The statements of this lemma are direct consequences of \cite[Lemma 5.8]{KP2}, Kato-Rellich theorem (see e.g., \cite[Theorem XII.8]{RS4}), and Proposition \ref{lambda}.

\begin{thm}
\label{LRRnonsa}
Let $A$ be a periodic elliptic operator of second-order with real and smooth coefficients on $X$ such that $\Lambda_A(0) \geq 0$. Let $\mu$ be a rigged divisor on $X$ and $\mu^+$ be its positive part. Then
\begin{enumerate}[a.]
\item
If $\Lambda_A(0)>0$,
\beq
\dim L_{\infty}(\mu^+,A,\varphi)=\deg_{A}(\mu^+)
\eeq
and
\beq
\dim{L_{\infty}(\mu,A,\varphi)}=\deg_{A}(\mu)+\dim{L_{\infty}(\mu^{-1},A^*,\varphi^{-1})}
\eeq
for any function $\varphi \in \mathcal{S}(G)$ (see Definition \ref{subexp}).
\item
If $\Lambda_A>\Lambda_A(0)=0$ and $d \geq 2$, then
\begin{itemize}
\item
For any $N \geq 0$,
\beq
\dim L_{\infty}(\mu^+, A, N)=c_{d,[N]}+\deg_{A}(\mu^+)
\eeq
and
\beq
c_{d,[N]}+\deg_{A}(\mu)+\dim L_1(\mu^{-1}, A^*, -N) \leq \dim L_{\infty}(\mu, A, N) \leq  c_{d,[N]}+\deg_{A}(\mu^+).
\eeq
\item
For any $p \in [1, \infty)$, $N>d/p$,
\beq
\dim L_{p}(\mu^+, A, N)=c_{d,\floor{N-d/p}}+\deg_{A}(\mu^+)
\eeq
and
\beq
c_{d, \floor{N-d/p}}+\deg_{A}(\mu)+\dim L_{p'}(\mu^{-1}, A^*, -N) \leq \dim L_{p}(\mu, A, N) \leq c_{d,\floor{N-d/p}}+\deg_{A}(\mu^+).
\eeq
\item
For $d \geq 3$ and a pair $(p,N)$ satisfying the condition in Theorem \ref{generic} c.,
\beq
\deg_{A}(\mu)+\dim L_{p'}(\mu^{-1}, A^*, -N)\leq \dim L_p(\mu, A, N) \leq \deg_A(\mu^+).
\eeq
\end{itemize}
\item
If $\Lambda_A=\Lambda_A(0)=0$ and $d \geq 3$, then
\begin{itemize}
\item
For any $N \geq 0$,
\beq
\dim L_{\infty}(\mu^+, A, N)=h_{d,[N]}+\deg_{A}(\mu^+)
\eeq
and
\beq
h_{d,[N]}+\deg_{A}(\mu)+\dim L_1(\mu^{-1}, A^*, -N) \leq \dim L_{\infty}(\mu, A, N) \leq  h_{d,[N]}+\deg_{A}(\mu^+).
\eeq
\item
For any $p \in [1, \infty)$, $N>d/p$,
\beq
\dim L_{p}(\mu^+, A, N)=h_{d,\floor{N-d/p}}+\deg_{A}(\mu^+)
\eeq
and
\beq
h_{d,\floor{N-d/p}}+\deg_{A}(\mu)+\dim L_{p'}(\mu^{-1}, A^*, -N) \leq \dim L_{p}(\mu, A, N) \leq h_{d,\floor{N-d/p}}+\deg_{A}(\mu^+).
\eeq
\item
For $d \geq 5$ and a pair $(p,N)$ satisfying the condition c. of Theorem \ref{generic},
\beq
\deg_{A}(\mu)+\dim L_{p'}(\mu^{-1}, A^*, -N)\leq \dim L_p(\mu, A, N) \leq \deg_A(\mu^+).
\eeq
\end{itemize}
\end{enumerate}
\end{thm}
\begin{proof}
To compute the dimensions of the spaces $V^p_N(A)$, we use Lemma \ref{dispersionnsa} to apply Theorem \ref{LiouvilleDim} and Theorem \ref{dimvp}. Then the statements of Theorem \ref{LRRnonsa} follows immediately from Theorem \ref{RRLbloch}, Lemma \ref{dispersionnsa}, Theorem \ref{RRLineq}, Proposition \ref{RRLeq}, and Theorem \ref{improve}.
\end{proof}

\chapter[Auxiliary statements and proofs]{Auxiliary statements and proofs of technical lemmas}
\label{C:auxiliary}

\section[Floquet functions]{Properties of Floquet functions on abelian coverings}\label{S:auxfloquet}
We recall briefly another construction of Floquet functions on $X$.
As we discussed in Section 2, it suffices to define the Bloch function $e_k(x)$ with quasimomentum $k$ and the powers $[x]^j$ on $X$, where $j \in \bZ_+^d$.
\begin{defi}
\label{addfunc}
A smooth mapping $h$ from $X$ to $\mathbb{R}^d$ is called an \textit{additive function} if the following condition holds:
\beq
\label{additive}
h(g\cdot x)=h(x)+g,
\eeq
where $(x,g) \in X \times \bZ^d$.
\end{defi}
There are various ways of constructing such additive functions (see e.g., \cite{KP2, LinPinchover}). We will fix such an additive function $h$ and write is as a tuple of scalar functions: $h=(h_1, \ldots, h_d)$.
Let $j=(j_1, \ldots, j_d) \in \bZ_+^d$ be a multi-index. We define
\beq
[x]^j:=h(x)^j=\prod_{m=1}^d h_m(x)^{j_m},
\eeq
and
\beq
e_k(x):=\exp{(ik\cdot h(x))}.
\eeq
Clearly, $e_k(g\cdot x)=e^{ik\cdot g}e_k(x)$. Then a Floquet function $u$ of order $N$ with quasimomentum $k$ is of the form
$$u(x)=e_k(x)\sum_{|j| \leq N}p_j(x)[x]^j,$$
where $p_j$ is smooth and periodic. Observe that the notion of Floquet functions is independent of the choice of $h$. Namely, $u$ is also a Floquet function with the same order and quasimomentum with respect to another additive function $\tilde{h}$. Indeed, the difference $w:=h-\tilde{h}$ between two additive functions $h$ and $\tilde{h}$ is a periodic function. Hence, one can rewrite
\beq
\begin{split}
u(x)&=e^{ik\cdot \tilde{h}(x)}\sum_{|j| \leq N}e^{ik\cdot w(x)}p_j(x)\prod_{m=1}^d (\tilde{h}_m(x)+w_m(x))^{j_m}\\
&=e^{ik\cdot \tilde{h}(x)}\sum_{|j| \leq N}e^{ik\cdot w(x)}p_j(x)\sum_{j\cprime \leq j}\binom{j}{j\cprime}w(x)^{j-j\cprime}\tilde{h}(x)^{j\cprime}\\
&=e^{ik\cdot \tilde{h}(x)}\sum_{|j| \leq N}\tilde{p}_j(x)\tilde{h}(x)^{j},
\end{split}
\eeq
where $\displaystyle \tilde{p}_j(x):=\sum_{j \leq j \cprime}\binom{j\cprime}{j}e^{ik\cdot w(x)}p_{j\cprime}(x)w(x)^{j\cprime-j}$ is periodic.
The following simple lemma is needed later.
\begin{lemma}
\label{powers}
Let $K$ be a compact neighborhood in $X$. Then for any multi-index $j \in \bZ_+^d$, there exists some constant $C$ such that for any $x \in K$ and $g \in \bZ^d$, one has
$$\left|[g\cdot x]^j-g^j \right|\leq C\langle g \rangle^{|j|-1}.$$
\end{lemma}
\begin{proof}
\beqn
\left|[g\cdot x]^j-g^j \right|=\left|\prod_{m=1}^d (h_m(x)+g_m)^{j_m}-\prod_{m=1}^d g_m^{j_m}\right|
\leq C\langle g \rangle^{|j|-1},
\eeqn
for some $C>0$ depending on $\|h\|_{L^{\infty}(K)}$.
\end{proof}
\section[properties of $A(k)$]{Basic properties of the family $\{A(k)\}_{k \in \mathbb{C}^d}$}
\label{a(k)}
We discuss another (equivalent) model of the operator family $A(k)$, which is sometimes useful to refer\footnote{Compare with the discussion of Floquet multipliers in \cite{Ksurvey} and discussions of vector bundles $E_k$ in Chapter \ref{C:preliminaries-LRR}.}. In this part, we abuse notations and identify elements in $L^2(M)$ with their periodic extensions that belong to $L^2_{loc}(X)$.

Consider now an additive function $h$ on the abelian covering $X$ (see Definition \ref{addfunc}).
Let $\mathcal{U}_k$ be the mapping that multiplies a function $f(x)$ in $L^2_k(X)$ by $e^{-ik\cdot h(x)}$. Thus, $\mathcal{U}_k$ is an invertible bounded linear operator in $\mathcal{L}(L^2_k(X), L^2(M))$ and its inverse is given by the multiplication by $e^{ik\cdot h(x)}$. Note that the operator norms of $\mathcal{U}_k$ and $\mathcal{U}_k^{-1}$ are bounded by
\beq
e^{|\Im{k}|\cdot \|h\|_{L^{\infty}(\overline{\mathcal{F}})}},
\eeq
where $\mathcal{F}$ is a fundamental domain.

Consider the following elliptic operator:
\beq
\widehat{A}(k):=\mathcal{U}_k A(k) \mathcal{U}_k^{-1}.
\eeq
The operator $\widehat{A}(k)$, with the Sobolev space $H^m(M)$ as the domain, is a closed and unbounded operator in $L^2(M)$.
For each complex quasimomentum $k$, the two linear operators $\widehat{A}(k)$ and $A(k)$ are similar and thus, their spectra are identical. In terms of spectral information, it is no need to distinguish $A(k)$ and its equivalent model $\widehat{A}(k)$. One of the benefits of working with the later model is that $\widehat{A}(k)$ acts on the $k$-independent domain of periodic functions on $X$, while the differential expression becomes a polynomial in $k$. Moreover, the operator $A(k)$ now acts on sections of the appropriate linear bundle $E_k$ (see Chapter \ref{C:preliminaries-LRR}).

The following proposition gives a simple sufficient condition on the principal symbol of the operator $A$ so that the spectra of $A(k)$ are discrete. More general criteria on the discreteness of spectra can be found, for instance, in \cite{Agmon}.
\begin{prop}
\label{realsymbol}
If $A$ has real principal symbol, then for each $k \in \mathbb{C}^d$, then $A(k)$, as an unbounded operator on $L^2(E_k)$, has discrete spectrum, i.e., its spectrum consists of isolated  eigenvalues with finite (algebraic) multiplicities.
\end{prop}
\begin{proof}
Let $B$ be the real part of the operator $A$. Since $A$ has real principal symbol, the principal symbols of $A$ and $B$ are the same. By pushing down to operators on $M$, the differential operator $\widehat{A}(0)-\widehat{B}(0)$ is of lower order.
Also, the principal symbols of the operators $\widehat{A}(k)$ and $\widehat{A}(0)$ are identical.
Thus, we see that $\widehat{A}(k)$ is a perturbation of the self-adjoint elliptic operator $\widehat{B}(0)$ by a lower order differential operator on the compact manifold $M$. It follows from \cite{Agmon} that the spectrum of $\widehat{A}(k)$ is discrete. This finishes the proof.
\end{proof}

If the spectrum $\sigma(A(k))$ is discrete, then the family of operators $\{\widehat{A}(k)\}_{k \in \mathbb{C}^d}$, has compact resolvents and is analytic of type (A) in the sense of Kato \cite{Ka}\footnote{A different approach to the analyticity of this operator family is taken in \cite{Ksurvey,Kbook}.}. Therefore, this family satisfies the upper-semicontinuity of the spectrum (see e.g., \cite{Ka, RS4}). We provide this statement here without a proof\footnote{Stronger results about properties of spectra of analytic Fredholm operator functions are available in \cite{ZKKP}}.
\begin{prop}
\label{semicont}
Consider $k_0 \in \mathbb{C}^d$. If $\Gamma$ is a compact subset of the complex plane such that $\Gamma \cap \sigma(A(k_0))=\emptyset$, then there exists $\delta>0$ depending on $\Gamma$ and $k_0$ such that $\Gamma \cap \sigma(A(k))=\emptyset$, for any $k$ in the ball $B_{k_0}(\delta)$ centered at $k_0$ with radius $\delta$.
\end{prop}
\begin{remark}\indent
\begin{enumerate}[(i)]
\item
The Hilbert bundle $\mathcal{E}^m$ becomes the trivial bundle $\mathbb{C}^d \times H^m(M)$ via the holomorphic bundle isomorphism defined from the linear maps $\mathcal{U}_k$, where $k \in \mathbb{C}^d$.
\item
In general, one can use the analytic Fredholm theorem to see that the essential spectrum\footnote{Here we use the definition of the essential spectrum of an operator $T$ as the set of all $\lambda \in \mathbb{C}$ such that $T-\lambda$ is not Fredholm.} of $A(k)$ is empty for any $k \in \mathbb{C}^d$, but this is not enough to conclude that these spectra are discrete in the non-self-adjoint case. For example, if we consider the $\bZ$-periodic elliptic operator $A=e^{2i\pi x}D_x$ on $\mathbb{R}$, a simple argument shows that
\beqn
  \sigma(A(k))=\left\{
  \begin{array}{@{}ll@{}}
    \mathbb{C}, & \text{if}\ k\in 2\pi \bZ \\
    \emptyset, & \text{otherwise.}
  \end{array}\right.
\eeqn
A similar example for the higher-dimensional case $\mathbb{R}^d$, $d>1$ can be cooked up easily from the above example.
\end{enumerate}
\end{remark}
\section[Floquet transform]{Properties of Floquet transforms on abelian coverings}\label{S:auxfloquet}

We describe here some useful properties of the Floquet transform $\textbf{F}$ on abelian coverings (see more about this in \cite{Ksurvey}).
First, due to \mref{fl}, one can see that the Floquet transform $\textbf{F}f(k,x)$ of a nice function $f$, e.g., $f \in C^{\infty}_c(X)$, is periodic in the quasimomentum variable $k$ and moreover, it is a quasiperiodic function in the $x$-variable, i.e.,
\beq
\textbf{F}f(k,g\cdot x)=\gamma_k(g)\cdot\textbf{F}f(k,x)=e^{ik\cdot g}\cdot\mathcal{F}f(k,x), \quad \mbox{for any} \hspace{3pt} (g,x) \in G \times X.
\eeq
It follows that $\textbf{F}f(k,\cdot)$ belongs to $H^s_k(X)$, for any $k$ and $s$.
Therefore, it is enough to regard the Floquet transform of $f$ as a smooth section of the Hilbert bundle $\mathcal{E}^s$ over the torus $\mathbb{T}^d$ (which can be identified with the Brillouin zone $B$).

Let $K \Subset X$ be a domain such that $\bigcup\limits_{g \in G}gK=X$. Then given any real number $s$, we denote by $\mathcal{C}^s(X)$ the Frechet space consisting of all functions $\phi \in H^s_{loc}(X)$ such that for any $N \geq 0$, one has
\beq
\sup\limits_{g \in G}\|\phi\|_{H^s(gK)}\cdot \langle g \rangle^{N}<\infty.
\eeq
In terms of Definition \ref{polyspace},
\beq
\bigcap\limits_{N \geq 0} V^{\infty}_N(X)=\mathcal{C}^0(X)\cap C^{\infty}(X).
\eeq
The following theorem collects Plancherel type results for the Floquet transform\footnote{Details, as well as Paley-Wiener type results can be found in \cite{Kbook, Ksurvey, KP1, KP2}.}.
\begin{thm}
\label{pwfloquet}\indent
\begin{enumerate}[a.]
\item
The Floquet transform $\textbf{F}$ is an isometric isomorphism between the Sobolev space $H^s(X)$ and the space $L^2(\mathbb{T}^d, \mathcal{E}^s)$ of $L^2$-integrable sections of the vector bundle $\mathcal{E}^s$.
\item
The Floquet transform $\textbf{F}$ expands the periodic elliptic operator $A$ of order $m$ in $L^2(X)$ into a direct integral of the fiber operators $A(k)$ over $\mathbb{T}^d$.
\beq
\textbf{F}A\textbf{F}^{-1}=\int^{\oplus}\limits_{\mathbb{T}^d}A(k)\di{k}.
\eeq
Equivalently, $\textbf{F}(Af)(k)=A(k)\textbf{F}f(k)$ for any $f \in H^m(X)$.
\item
The Floquet transform
\beq
\textbf{F}: \mathcal{C}^s(X) \rightarrow C^{\infty}(\mathbb{T}^d, \mathcal{E}^s)
\eeq
is a topological isomorphism, where $C^{\infty}(\mathbb{T}^d, \mathcal{E}^s)$ is the space of smooth sections of the vector bundle $\mathcal{E}^s$.
Furthermore, under the Floquet transform $\textbf{F}$, the operator
\beq
A: \mathcal{C}^m(X) \rightarrow \mathcal{C}^0(X)
\eeq
becomes a morphism of sheaves of smooth sections arising from the holomorphic Fredholm morphism $A(k)$ between the two holomorphic Hilbert bundles $\mathcal{E}^m$ and $\mathcal{E}^0$ over the torus $\mathbb{T}^d$, i.e., it is an operator from $C^{\infty}(\mathbb{T}^d, \mathcal{E}^m)$ to $C^{\infty}(\mathbb{T}^d, \mathcal{E}^0)$ such that it acts on the fiber of $\mathcal{E}^m$ at $k$ as the fiber operator $A(k): H^m_k(X) \rightarrow L^2_k(X)$.
\item
The inversion $\textbf{F}^{-1}$ of the Floquet transform is given by the formula
\beq
\label{flinv}
f(x)=\frac{1}{(2\pi)^d}\int\limits_{\mathbb{T}^d}\textbf{F}f(k,x)\di{k},
\eeq
provided that one can make sense both sides of \mref{flinv} (as functions or distributions).
\end{enumerate}
\end{thm}
We prove a simple analog of the Riemann-Lebesgue lemma for the Floquet transform.
\begin{lemma}\indent
\label{flrl}
\begin{enumerate}[a.]
\item
Let $\widehat{f}(k,x)$ be a function in $L^1(\mathbb{T}^d,\mathcal{E}^0)$. Then the inverse Floquet transform $f:=\textbf{F}^{-1}\widehat{f}$ belongs to $L^{2}_{loc}(X)$ and
$$\sup_{g \in G}\|f\|_{L^2(g\mathcal{F})}<\infty.$$
Here $\mathcal{F}$ is a fixed fundamental domain.
Moreover, one also has
$$\lim_{|g| \rightarrow \infty}\|f\|_{L^2(g\mathcal{F})}=0.$$
\item
If $f \in V^{1}_0(X)$ then $\textbf{F}f(k,x) \in C(\mathbb{T}^d, \mathcal{E}^0)$.
\end{enumerate}
\end{lemma}
\begin{proof}
We recall that $L-2$ sections of $\mathcal{E}^0$ can be identified with the elements of $L^2(\mathcal{F})$.

To prove the first statement, we use the identity $\mref{flinv}$ and the Minkowski's inequality to obtain
\beq
\begin{split}
\|f\|_{L^2(g\mathcal{F})}&=\frac{1}{(2\pi)^d}\left\|\int_{\mathbb{T}^d}\textbf{F}f(k,\cdot)\di{k}\right\|_{L^2(g\mathcal{F})}=\frac{1}{(2\pi)^d}\left\|\int_{\mathbb{T}^d}e^{ik\cdot g}\textbf{F}f(k,\cdot)\di{k}\right\|_{L^2(\mathcal{F})}\\
&\leq \frac{1}{(2\pi)^d}\int_{\mathbb{T}^d}\left\|\textbf{F}f(k,\cdot)\right\|_{L^2(\mathcal{F})}\di{k}=\frac{1}{(2\pi)^d}\|\textbf{F}f(k,x)\|_{L^1(\mathbb{T}^d, \mathcal{E}^0)}<\infty.
\end{split}
\eeq
To show that
\beq
\lim_{|g| \rightarrow \infty}\|f\|_{L^2(g\mathcal{F})}=0,
\eeq
one can easily modify the standard proof of the Riemann-Lebesgue lemma, i.e., by using Theorem \ref{pwfloquet} a. and then approximating $\textbf{F}f$ by a sequence of functions in $L^2(\mathbb{T}^d, \mathcal{E}^0)$.

The second statement follows directly from \mref{fl} and the triangle inequality.
\end{proof}

\section{A Schauder type estimate}\label{S:schauder}

For convenience, we state a well-known Schauder type estimate for solutions of a periodic elliptic operator $A$, which we need to refer to several times in this text. We also sketch its proof for the sake of completeness.

For any open subset  $\mathcal{O}$ of $X$ such that $\hat{K} \subset \mathcal{O}$, we define
\beq
G_{\hat{K}}^{\mathcal{O}}:=\{g \in G \mid g\hat{K} \subset \mathcal{O}\}.
\eeq
\begin{prop}
\label{schauderest}
Let $K\subset X$ be a compact set with non-empty interior and $\mathcal{O}$ be its open neighborhood. There exists a compact subset $\hat{K}\subset X$ such that $K \Subset \hat{K} \subset \mathcal{O}$ and the following statement holds: For any $\alpha \in \mathbb{R}^+$, there exists $C>0$ depending on $\alpha, K, \hat{K}$ such that
\beq
\label{schauderineq}
\|u\|_{H^{\alpha}(gK)} \leq C\cdot \|u\|_{L^2(g\hat{K})},
\eeq
for any $g \in G_{\hat{K}}^{\mathcal{O}}$ and any solution $u \in C^{\infty}(\mathcal{O})$ satisfying the equation $Au=0$ on $\mathcal{O}$.
\end{prop}
\begin{proof}
Let $B$ be an almost local \footnote{I.e., for some $\varepsilon>0$, the support of the Schwartz kernel of $B$ is contained in an $\varepsilon$-neighborhood of the diagonal of $X \times X$.} pseudodifferential parametrix of $A$ such that $B$ commutes with actions of the deck group $G$ (see e.g., \cite[Lemma 2.1.1]{Kbook} or \cite[Proposition 3.4]{Shubin_spectral}). Hence, $BA=1+T$ for some almost-local and periodic smoothing operator $T$ on $X$.
This implies that for some compact neighborhood $\hat{K}$ (depending on the support of the Schwartz kernel of $T$ and the subset $K$) and for any $\alpha \geq 0$, one can find some $C>0$ so that for any smooth function $v$ on a neighborhood of $\hat{K}$, one gets
$$\|Tv\|_{H^{\alpha}(K)} \leq C\cdot \|v\|_{L^2(\hat{K})}.$$
In particular, for any $g \in G_{\hat{K}}^{\mathcal{O}}$ and $u \in C^{\infty}(\mathcal{O})$,
\beq
\label{schauder1}
\|Tu^g\|_{H^{\alpha}(K)} \leq C\cdot \|u^g\|_{L^2(\hat{K})},
\eeq
where $u^g$ is the $g$-shift of the function $u$ on $\mathcal{O}$. Since $T$ is $G$-periodic, from \mref{schauder1}, we obtain
\beq
\label{schauder2}
\|Tu\|_{H^{\alpha}(gK)} \leq C\cdot \|u\|_{L^2(g\hat{K})}.
\eeq

The important point here is the uniformity of the constant $C$ with respect to $g \in G_{\hat{K}}^{\mathcal{O}}$.

Suppose now that $Au=0$ on $\mathcal{O}$. Thus, $u=BAu-Tu=-Tu$ on $\mathcal{O}$. This identity and \mref{schauder2} imply the estimate
\beq
\label{schauder}
\|u\|_{H^{\alpha}(gK)}=\|Tu\|_{H^{\alpha}(gK)}\leq C\cdot \|u\|_{L^2(g\hat{K})}, \quad \forall g \in G_{\hat{K}}^{\mathcal{O}}.
\eeq
\end{proof}
\begin{remark}
\label{schauderuniform}
We need to emphasize that Proposition \ref{schauderest} holds in a more general context. Namely, it is true for any \textit{$C^{\infty}$-bounded uniformly elliptic operator} $P$ on a co-compact Riemannian covering $\mathcal{X}$ with a discrete deck group $G$.
In this setting, $P$ is invertible modulo a \textit{uniform smoothing operator} $T$ on $\mathcal{X}$ (see \cite[Definition 3.1]{Shubin_spectral} and \cite[Proposition 3.4]{Shubin_spectral}). Now the estimate \mref{schauder2} follows easily from the uniform boundedness of the derivatives of any order of the Schwartz kernel of $T$ on canonical coordinate charts and a routine argument of partition of unity.
Another possible approach is to invoke uniform local apriori estimates \cite[Lemma 1.4]{Shubin_spectral}.
\end{remark}
\section{A variant of Dedekind's lemma}\label{S:dedekind}
It is a well-known theorem by Dedekind (see e.g., \cite[Lemma 2.2]{morandi}) that distinct unitary characters of an abelian group $G$ are linearly independent as functions from $G$ to a field $\mathbb{F}$. The next lemma is a refinement of Dedekind's lemma when $\mathbb{F}=\mathbb{C}$. We notice that a proof by induction method can be found in \cite[Lemma 4.4]{randlescoste}. For the sake of completeness, we will provide our analytic proof using Stone-Weierstrass's theorem.
\begin{lemma}
\label{dedekind}
Consider a finite number of distinct unitary characters $\gamma_1, \ldots, \gamma_{\ell}$ of the abelian group $\mathbb{Z}^d$. Then there are vectors $g_1, \ldots, g_{\ell}$ in $\mathbb{Z}^d$ and $C>0$ such that for any $v=(v_1, \ldots, v_{\ell})\in \mathbb{C}^{\ell}$, we have
$$\max\limits_{1 \leq s \leq \ell}\left|\sum_{r=1}^{\ell} v_r \cdot \gamma_{r}(g_s)\right| \geq C\cdot \max\limits_{1 \leq r \leq \ell}|v_r|.$$
\end{lemma}
\begin{proof}
By abuse of notation, we can regard $\gamma_1, \ldots, \gamma_{\ell}$ as distinct points of the torus $\mathbb{T}^d$.

For each tuple $(g_1, \ldots, g_{\ell})$ in $(\bZ^{d})^{\ell}$, let $W(g_1, \ldots, g_{\ell})$ be a $\ell \times \ell$-matrix whose $(s,r)$-entry
$\displaystyle W(g_1, \ldots, g_{\ell})_{s,r}$ is $\gamma_{s}^{g_r}$, for any $1 \leq r,s \leq \ell$. We equip $\mathbb{C}^{\ell}$ with the maximum norm. Then the conclusion of the lemma is equivalent to the invertibility of some operator $W(g_1, \ldots, g_{\ell})$ acting from $\mathbb{C}^{\ell}$ to $\mathbb{C}^{\ell}$.

Suppose for contradiction that the determinant function $\det{W(g_1, \ldots, g_{\ell})}$ is zero on $(\bZ^d)^{\ell}$, i.e., for any $g_1, \ldots, g_{\ell} \in \mathbb{Z}^d$, one has
$$0=\det{W(g_1, \ldots, g_{\ell})}=\sum_{\sigma \in S_{\ell}} \sign{(\sigma)}\cdot \left(\gamma_{\sigma(1)}^{g_1}\ldots \gamma_{\sigma(\ell)}^{g_{\ell}}\right),$$
where $S_{\ell}$ is the permutation group on $\{1, \ldots, \ell\}$.
Thus, the above relation also holds for any trigonometric polynomial $P(\gamma_1, \ldots, \gamma_{\ell})$ on $(\mathbb{T}^{d})^{\ell}$, i.e.,
$$\sum_{\sigma \in S_{\ell}} \sign{(\sigma)}\cdot P(\gamma_{\sigma(1)}, \ldots, \gamma_{\sigma(\ell)})=0.$$
By using the fact that the trigonometric polynomials are dense in $C((\mathbb{T}^d)^{\ell})$ in the uniform topology (Stone-Weierstrass theorem), we conclude that
\beq
\label{stone}
\sum_{\sigma \in S_{\ell}} \sign{(\sigma)}\cdot f(\gamma_{\sigma(1)}, \ldots, \gamma_{\sigma(\ell)})=0,
\eeq
for any continuous function $f$ on $(\mathbb{T}^d)^{\ell}$.

Now for each $1 \leq r \leq \ell$, let us select some smooth cutoff functions $\omega_{r}$ supported on a neighborhood of the point $\gamma_r$ such that $\omega_{r}(\gamma_s)=0$ whenever $s \neq r$. We define $f \in C((\mathbb{T}^d)^{\ell})$ as follows
$$f(x_1, \ldots , x_{\ell}):=\prod_{r=1}^{\ell} \omega_{r}(x_r), \quad x_1, \ldots, x_{\ell} \in \mathbb{T}^d.$$
Hence, $f(\gamma_{\sigma(1)}, \ldots , \gamma_{\sigma(\ell)})$ is non-zero if and only if $\sigma$ is the trivial permutation.
By substituting $f$ into \mref{stone}, we get a contradiction, which proves our lemma.
\end{proof}

\section[Other technical notes]{Proofs of some other technical statements}
\label{technical}
In this section, we will use the notation $\cF$ for the closure of a fundamental domain for $G$-action on the covering $X$.
\subsection{Proof of Theorem \ref{dimvp}}
If $u \in V^{\infty}_{\floor{N-d/p}}(A)$, then $u \in V^{\infty}_{N_0}(A)$ for some nonnegative integer $N_0$ such that $N_0<N-d/p$. Thus,
\beq
\sum_{g \in G} \|u\|^p_{L^2(g\cF)} \langle g \rangle^{-pN} \lesssim \sum_{g \in G} \langle g \rangle^{p(N_0-N)}<\infty.
\eeq
Hence, $V^{\infty}_{\floor{N-d/p}}(A) \subseteq V^{p}_N(A)$.

Now suppose that $F_{A, \mathbb{R}}=\{k_1, \ldots, k_{\ell}\}$ (modulo $G^*$-shifts), where $\ell \in \mathbb{N}$. It suffices to show that
\beq
\label{vp}
V^{p}_N(A) \subseteq V^{\infty}_{\floor{N-d/p}}(A).
\eeq
A key ingredient of the proof of \mref{vp} is the following statement.
\begin{lemma}
\label{induction}
Suppose that $\mathcal{N}>d/p$.
\begin{enumerate}[(i)]
\item
If $u \in V^p_{\mathcal{N}}(A) \cap V^{\infty}_{\mathcal{M}}(A)$ for some $0 \leq \mathcal{M}<\mathcal{N}+1-d/p$, then $u \in V^{\infty}_{\mathcal{M}\cprime}(A)$ for some $\mathcal{M}\cprime < \mathcal{N}-d/p$. In particular, $u \in V^{\infty}_{\mathcal{N}-d/p}(A)$.
\item
If u is in $V^p_{\mathcal{N}}(A) \cap V^{\infty}_{\mathcal{N}+1-d/p}(A)$, then $u\in V^{\infty}_{\mathcal{N}-d/p}(A)$.
\end{enumerate}
\end{lemma}
Instead of proving Lemma \ref{induction} immediately, let us assume first its validity and  prove \mref{vp}. Consider $u \in V^p_N(A)$.

\textit{Case 1}. $p>1$.

We prove by induction that if $0 \leq s \leq d-1$, then
\beq
\label{inductionstep}
u \in V^p_{N+d/p-(s+1)/p}(A) \cap V^{\infty}_{N-s/p}(A).
\eeq
The statement holds for $s=0$ since clearly, $V^p_N(A) \subseteq V^{\infty}_{N}(A)$ and $N+d/p-1/p \geq N$.
Now suppose that \mref{inductionstep} holds for $s$ such that $s+1 \leq d-1$.
Since $1-1/p>0$, we can apply Lemma \ref{induction} (i) to $u$ and the pair $(\mathcal{N}, \mathcal{M})=(N+d/p-(s+1)/p, N-s/p)$ to deduce that $u \in V^{\infty}_{N-(s+1)/p}(A)$. Therefore, \mref{inductionstep} also holds for $s+1$.
In the end, we have
\beq
u \in V^p_{N}(A) \cap V^{\infty}_{N-(d-1)/p}(A).
\eeq Applying Lemma \ref{induction} (i) again, we conclude that $u$ belongs to $V^{\infty}_{\mathcal{M}\cprime}(A)$ for some $\mathcal{M}\cprime<N-d/p$. In other words, $u$ is in $V^{\infty}_{\floor{N-d/p}}(A)$.

\textit{Case 2}. $p=1$.

As in Case 1, we apply Lemma \ref{induction} (ii) and induction to prove that
\beq
u \in V^1_{N+d-1-s}(A) \cap V^{\infty}_{N-s}(A)
\eeq
for any $0 \leq s \leq d-1$. Hence,
\beq
u \in V^1_N(A) \cap V^{\infty}_{N+1-d}(A).
\eeq
Due to Lemma \ref{induction} (ii) again, one concludes that
\beq
u \in V^1_N(A) \cap V^{\infty}_{N-d}(A).
\eeq
Applying now Lemma \ref{induction} (i) to $u$ and the pair $(\mathcal{N}, \mathcal{M})=(N, N-d)$, we conclude that
\beq
u \in V^{\infty}_{\floor{N-d}}(A)=V^{\infty}_{N-(d+1)}(A).
\eeq
Thus, Theorem \ref{dimvp} follows from Lemma \ref{induction}.

Let us turn now to our proof of Lemma \ref{induction}, which consists of several steps.
\begin{enumerate}[(i)]
\item
\textbf{Step 1.}
The lemma is trivial if $u=0$. So from now on, we assume that $u$ is non-zero.
Since $V^p_{\mathcal{N}}(A) \subseteq V^{\infty}_{\mathcal{N}}
(A)$, we can apply Theorem \ref{Liouvillethm} (ii) to represent $u \in V^{p}_{\mathcal{N}}(A)$ as a finite sum of Floquet solutions of $A$, i.e.,
\beq
u=\sum_{r=1}^{\ell} u_r.
\eeq
Here $u_r$ is a Floquet function of order $M_r \leq {\mathcal{N}}$ with quasimomentum $k_r$.
Let $N_0$ be the highest order among all the orders of the Floquet functions $u_r$ appearing in the above representation.
Without loss of generality, we can assume that there exists $r_0 \in [1,\ell]$ such that for any $r \leq r_0$, the order $M_r$ of $u_r$ is maximal among all of these Floquet functions.
Thus, $M_r=N_0 \leq \mathcal{M}$ when $r \in [1,r_0]$.
To prove our lemma, it suffices to show that $N_0<\mathcal{N}-d/p$.

\textbf{Step 2.} According to Proposition \ref{schauderest}, we can pick a compact neighborhood $\hat{\cF}$ of $\cF$ such that for any $\alpha \geq 0$,
\beq
\|u\|_{H^{\alpha}(g\cF)}\leq C \cdot \|u\|_{L^2(g\hat{\cF})}
\eeq
for some $C>0$ independent of $g \in G$.

Let $\alpha>n/2$, then the Sobolev embedding theorem yields the estimate
\beq
\label{schauder}
\|u\|_{C^{0}(g\cF)} \lesssim \|u\|_{L^2(g\hat{\cF})}, \quad \forall g \in G.
\eeq
From \mref{schauder} and the fact that $u \in V_{\mathcal{N}}^p(A)$, we obtain
\beq
\label{a.e1}
\sup_{x \in \cF}\left(\sum_{g \in G}\left|u(g\cdot x)\right|^p \langle g \rangle^{-p\mathcal{N}}\right) \lesssim
\sum_{g \in G}\|u\|_{L^2(g\hat{\cF})}^p \langle g \rangle^{-p\mathcal{N}} <\infty.
\eeq

\textbf{Step 3.}
One can write
\beq
u(x)=\sum_{r=1}^{\ell}u_r(x)=\sum_{r=1}^{r_0}e_{k_r}(x)\sum_{|j|=N_0}a_{j,r}(x)[x]^{j}+O(|x|^{N_0-1}).
\eeq
Here each function $a_{j,r}$ is $G$-periodic and the remainder term $O(|x|^{N_0-1})$ is an exponential-polynomial with periodic coefficients of order at most $N_0-1$.
Hence, for any $(g,x) \in G \times \cF$, we get
\beq
u(g \cdot x)=\sum_{r=1}^{r_0}e^{ik_r\cdot g}\sum_{|j|=N_0}e_{k_r}(x)a_{j,r}(x)[g \cdot x]^{j}+O(\langle g\rangle^{N_0-1}).
\eeq
Since $N_0-1 \leq \mathcal{M}-1<\mathcal{N}-d/p$, the series
\beq
\sum_{g \in \bZ^d} \langle g \rangle^{p(N_0-1)}\cdot\langle g \rangle^{-p\mathcal{N}}
\eeq
converges.
From this and \mref{a.e1}, we deduce that
\beq
\sup\limits_{x \in \cF}\sum_{g \in G}\left|\sum_{r=1}^{r_0}e^{ik_r\cdot g}\sum_{|j|=N_0}e_{k_r}(x)a_{j,r}(x)[g \cdot x]^{j}\right|^p \langle g \rangle^{-p\mathcal{N}}<\infty.
\eeq
By Lemma \ref{powers},
\beq
|[g\cdot x]^j-g^j|=O(\langle g \rangle^{N_0-1})
\eeq
 for any multi-index $j$ such that $|j|=N_0$.
This implies that
\beq
\label{a.e2}
\sup\limits_{x \in \cF}\sum_{g \in G}\left|\sum_{|j|=N_0}\sum_{r=1}^{r_0}e_{k_r}(x)a_{j,r}(x)e^{ik_r\cdot g}g^{j}\right|^p \langle g \rangle^{-p\mathcal{N}}<\infty.
\eeq

\textbf{Step 4.} We will use Lemma \ref{dedekind} to reduce the condition \mref{a.e2} to the one without exponential terms $e^{ik_r \cdot g}$, so we could assume that $F_{A, \mathbb{R}}=\{0\}$ (modulo $G^*$-shifts).

Indeed, let $\gamma_1, \ldots, \gamma_{r_0}$ be distinct unitary characters of $G$ that are defined via $\gamma_r(g):=e^{ik_r \cdot g}$, where $r \in \{1, \ldots, r_0\}$ and $g \in G$. Now, due to Lemma \ref{dedekind}, there are $g_1, \ldots, g_{r_0} \in G$ and a constant $C>0$ such that for any vector $(v_1, \ldots, v_{r_0}) \in \mathbb{C}^{r_0}$, we have the following inequality:
\beq
\label{a.e3}
C\cdot \max\limits_{1 \leq s \leq r_0}\left|\sum_{r=1}^{r_0} v_r \cdot e^{ik_r \cdot g_s}\right| \geq \max\limits_{1 \leq r\leq r_0}|v_r|.
\eeq
Now, for any $(g,x) \in G \times \cF$ and $1 \leq s \leq r_0$, we apply \mref{a.e3} to the vector
\beq
(v_1, \ldots, v_{r_0}):=\left(\sum_{|j|=N_0}e_{k_r}(x)a_{j,r}(x)(g+g_s)^{j}\langle g+g_s \rangle^{-\mathcal{N}}e^{ik_r \cdot g}\right)_{1 \leq r \leq r_0}
\eeq
to
deduce that
\beq
\label{a.e4}
\begin{split}
&\max\limits_{1 \leq r \leq r_0}\left|\sum_{|j|=N_0}e_{k_r}(x)a_{j,r}(x)(g+g_s)^{j}\right|^p\langle g+g_s \rangle^{-p\mathcal{N}}\\
=
&\max\limits_{1 \leq r \leq r_0}\left|\sum_{|j|=N_0}e_{k_r}(x)a_{j,r}(x)(g+g_s)^{j}\langle g+g_s \rangle^{-\mathcal{N}}e^{ik_r \cdot g}\right|^p\\
\lesssim
&\max\limits_{1 \leq s \leq r_0} \left|\sum\limits_{r=1}^{r_0}\left(\sum_{|j|=N_0}e_{k_r}(x)a_{j,r}(x)(g+g_s)^{j}\langle g+g_s \rangle^{-\mathcal{N}} e^{ik_r \cdot g}\right)e^{ik_r \cdot g_s}\right|^p\\
\lesssim
&\sum\limits_{s=1}^{r_0} \left|\sum\limits_{r=1}^{r_0}\sum_{|j|=N_0}e_{k_r}(x)a_{j,r}(x)(g+g_s)^{j}\cdot e^{ik_r \cdot (g+g_s)}\right|^p\cdot \langle g+g_s \rangle^{-p\mathcal{N}}
\end{split}
\eeq
Summing the estimate \mref{a.e4} over $g \in G$, we derive
\beq
\label{a.e5}
\begin{split}
&\max\limits_{1 \leq r \leq r_0}\sup_{x \in \cF}\sum_{g \in G}\left|\sum_{|j|=N_0}e_{k_r}(x)a_{j,r}(x)g^{j}\right|^p\langle g \rangle^{-p\mathcal{N}}\\
=
&\max\limits_{1 \leq r,s \leq r_0}\sup_{x \in \cF}\sum_{g \in G}\left|\sum_{|j|=N_0}e_{k_r}(x)a_{j,r}(x)(g+g_s)^{j}\right|^p\langle g+g_s \rangle^{-p\mathcal{N}}\\
\lesssim
&\sup_{x \in \cF}\sum\limits_{s=1}^{r_0} \sum_{g \in G}\left|\sum\limits_{r=1}^{r_0}\sum_{|j|=N_0}e_{k_r}(x)a_{j,r}(x)(g+g_s)^{j}\cdot e^{ik_r \cdot (g+g_s)}\right|^p\cdot \langle g+g_s \rangle^{-p\mathcal{N}}
\\
\lesssim
&\sup_{x \in \cF}\sum_{g \in G}\left|\sum\limits_{r=1}^{r_0}\sum_{|j|=N_0}e_{k_r}(x)a_{j,r}(x)g^{j}\cdot e^{ik_r \cdot g}\right|^p\cdot \langle g \rangle^{-p\mathcal{N}}<\infty.
\end{split}
\eeq
From \mref{a.e2} and \mref{a.e5}, we get
\beq
\label{a.e6}
\sum_{g \in G}\left|\sum_{|j|=N_0}e_{k_r}(x)a_{j,r}(x)g^{j}\right|^p \cdot \langle g \rangle^{-p\mathcal{N}}<\infty,
\eeq
for any $1 \leq r \leq r_0$ and $x \in \cF$.

\textbf{Step 5.} At this step, we prove the following claim: If $P$ is a non-zero homogeneous polynomial of degree $N_0$ in $d$-variables such that $N_0<\mathcal{N}+1-d/p$ and
\beq
\label{a.e7}
\sum_{g \in \bZ^d} \left|P(g)\right|^p \cdot \langle g \rangle^{-p\mathcal{N}}<\infty,
\eeq
then $N_0<\mathcal{N}-d/p$.

Our idea is to approximate the series in \mref{a.e7} by the integral
$$\mathcal{I}:=\int\limits_{\mathbb{R}^d}|P(z)|^p\cdot \langle z \rangle^{-p\mathcal{N}}\di{z}.$$
In fact, for any $z \in [0,1)^d+g$, one can use the triangle inequality and the assumption that the order of $P$ is $N_0$ to achieve the following estimate
$$|P(z)|^p \leq 2^{p-1}\left(|P(g)|^p+|P(z)-P(g)|^p\right) \lesssim |P(g)|^p+\langle g \rangle^{p(N_0-1)}.$$
Integrating the above estimate over the cube
\beq
[0,1)^d+g
\eeq
and then summing over all $g \in \bZ^d$, we deduce
\beq
\begin{split}
\mathcal{I}&=\sum\limits_{g \in \bZ^d} \int\limits_{[0,1)^d+g} |P(z)|^p\cdot \langle z \rangle^{-p\mathcal{N}}\di{z} \\
&\lesssim \sum_{g \in \bZ^d} \left|P(g)\right|^p \cdot \langle g \rangle^{-p\mathcal{N}}+\sum_{g \in \bZ^d} \langle g \rangle^{p(N_0-1-\mathcal{N})}<\infty,
\end{split}
\eeq
where we have used \mref{a.e7} and the assumption $(N_0-1-\mathcal{N})p<-d$.

We now rewrite the integral $\mathcal{I}$ in polar coordinates:
\beq
\mathcal{I}=\int\limits_{0}^{\infty} \int\limits_{\mathbb{S}^{d-1}}|P(r\omega)|^p \langle r \rangle^{-p\mathcal{N}}r^{d-1}\di\omega \di{r}=\int\limits_{0}^{\infty}\langle r \rangle^{-p\mathcal{N}}r^{d-1+pN_0}\di{r}\cdot \int\limits_{\mathbb{S}^{d-1}}|P(\omega)|^p \di\omega.
\eeq
Suppose for contradiction that $(N_0-\mathcal{N})p \geq -d$. Then
\beq
\int\limits_{0}^{\infty}\langle r \rangle^{-p\mathcal{N}}r^{d-1+pN_0}\di{r}=\infty.
\eeq
Thus, the finiteness of $\mathcal{I}$ implies that
\beq
\int\limits_{\mathbb{S}^{d-1}}|P(\omega)|^p \di\omega=0.
\eeq
Hence, $P(\omega)=0$ for any $\omega \in \mathbb{S}^{d-1}$. By homogeneity, $P$ must be zero, which is a contradiction that proves our claim.

\textbf{Step 6}.
Since $u$ is non-zero, there are some $r \in \{1, \ldots, r_0\}$ and $x \in\cF$ such that the following homogeneous polynomial of degree $N_0$ in $\mathbb{R}^d$
\beq
P(z):=\sum_{|j|=N_0}e_{k_r}(x)a_{j,r}(x)z^{j}
\eeq
is non-zero. Due to \mref{a.e6} and the condition
\beq N_0 \leq \mathcal{M}<\mathcal{N}+1-d/p \mbox{ (see Step 1)},
\eeq
the inequality $N_0<\mathcal{N}-d/p$ must be satisfied according to Step 5. This finishes the proof of the first part of the lemma.
\item
Consider
\beq
u \in V^{p}_{\mathcal{N}}(A) \cap V^{\infty}_{\mathcal{N}+1-d/p}(A).
\eeq
In particular, for any $\varepsilon>0$,
\beq u \in V^{p}_{\mathcal{N}+\varepsilon}(A) \cap V^{\infty}_{\mathcal{N}+\varepsilon+1-d/p}(A).
\eeq
We repeat the proof of Lemma \ref{induction} (i) for $(\mathcal{N}+\varepsilon)$ (instead of $\mathcal{N}$ in part (i)). As in the Step 1 of the previous proof, we first decompose $u$ as a finite sum of Floquet solutions and let $N_0$ be the highest order among all the orders of the Floquet functions appearing in that decomposition.
Repeating all the steps of the part (i), we conclude that
\beq
N_0<\mathcal{N}+\varepsilon-d/p
\eeq
for any $\varepsilon>0$.
By letting $\varepsilon \rightarrow 0^+$, $N_0 \leq \mathcal{N}-d/p$. We conclude that $u \in V^{\infty}_{\mathcal{N}-d/p}(A)$. This yields the second part of the lemma.
\end{enumerate}

\subsection{Proof of Theorem \ref{UCinfty}}\indent
\begin{enumerate}[a.]
\item
Consider $u \in V^{p}_N(A)$.
Due to Theorem \ref{dimvp} and the condition that $N\leq d/p$,
\beq
V^{p}_N(A) \subseteq V^{p}_{d/p+1/2}(A)=V^{\infty}_0(A).
\eeq
Using Theorem \ref{Liouvillethm} (ii), we get
\beq
\label{repn}
u(x)=\sum_{r=1}^{\ell}e_{k_r}(x)a_r(x),
\eeq
for some periodic functions $a_r(x)$.

Using \mref{schauder} and the assumption that $u \in V^{p}_N(A)$, we derive
\beq
\label{a.e8}
\sup_{x \in \cF}\sum_{g \in G}\left|\sum_{r=1}^{\ell}e_{k_r}(x)a_r(x)e^{ik_r\cdot g}\right|^p\cdot\langle g \rangle^{-pN}<\infty.
\eeq

Now one can modify (from the estimate \mref{a.e6}instead of \mref{a.e2}) the argument in Step 4 of the proof of Lemma \ref{induction} to get
\beq
\label{a.e9}
\max_{1\leq r\leq\ell}\sup_{ x \in \cF}\left|e_{k_r}(x)a_r(x)\right|^p\cdot \sum_{g \in \bZ^d}\langle g \rangle^{-pN}<\infty.
\eeq
Hence, the assumption $-pN\geq -d$ implies that $\max\limits_{1\leq r\leq\ell}\sup\limits_{x \in \cF}\left|e_{k_r}(x)a_r(x)\right|=0$. Thus, $u$ must be zero.

\item
Let $u$ be an arbitrary element in $V^{\infty}_N(A)$. Since $N<0$, we can assume that $u$ has the form \mref{repn}.
To prove that $u=0$, it is enough to show that $e_{k_r}(x)a_r(x)=0$ for any $x \in \cF$ and $1 \leq r \leq \ell$. One can repeat the same argument of the previous part to prove this claim. However, we will provide a different proof using Fourier analysis on the torus $\mathbb{T}^d$.

For each $x \in \cF$, we introduce the following distribution on $\mathbb{T}^d$
\beq
\label{delta_f}
f(k):=\sum_{r=1}^{\ell}e_{k_r}(x)a_r(x)\delta(k-k_r),
\eeq
where $\delta(\cdot-k_r)$ is the Dirac delta distribution on the torus $\mathbb{T}^d$ at the quasimomentum $k_r$. In terms of Fourier series, we obtain
\beq
\hat{f}(g)=\sum_{r=1}^{\ell}e_{k_r}(x)a_r(x)e^{-ik_r\cdot g}.
\eeq
As in \mref{a.e8}, the assumption $u \in V^{\infty}_N(A)$ is equivalent to
\beq
\sup_{g \in \bZ^d}\left|\hat{f}(g)\right| \cdot\langle g \rangle^{-N}<\infty.
\eeq
Let $\phi$ be a smooth function on $\mathbb{T}^d$. Using Parseval's identity and H\"older's inequality, we obtain
\beq
\label{a.e10}
\left|\sum_{r=1}^{\ell}e_{k_r}(x)a_r(x)\phi(k_r)\right|=\left|\langle f, \phi \rangle \right|=\left|\sum_{g\in \bZ^d}\hat{f}(g)\hat{\phi}(-g)\right|
\lesssim\sum_{g \in \bZ^d}|\hat{\phi}(g)|\cdot \langle g \rangle^{N}.
\eeq

Let us now pick $\delta>0$ small enough such that $k_{s} \notin B(k_r, 2\delta)$ if $s \neq r$.
Then we choose a cut-off function $\phi_r$ such that $\supp{\phi_r} \subseteq B(k_r, 2\delta)$  and $\phi_r=1$ on $B(k_r, \delta)$.
For $1 \leq r \leq \ell$, we define functions in $C^{\infty}(\mathbb{T}^d)$ as follows:
\beq
\phi^{\varepsilon}_{r}(k):=\phi_{r}(\varepsilon^{-1} k), \quad (0<\varepsilon<1).
\eeq

To bound the Fourier coefficients of $\phi^{\varepsilon}_{r}$ in terms of $\varepsilon$, we use integration by parts. Indeed, for any nonnegative real number $s$,
\beq
\begin{split}
(2\pi)^d|\widehat{\phi^{\varepsilon}_{r}}(g)|&=
\left
|\int_{\mathbb{T}^d}\phi^{\varepsilon}_{r}(k) e^{-ik\cdot g}dk\right|
=\varepsilon^{d}\cdot\left
|\int_{B(k_r, 2\delta)}\phi_{r}(k) e^{-i\varepsilon k\cdot g}dk\right|
\\
&=\varepsilon^{d}\langle \varepsilon g\rangle^{-2[s]-2}\cdot\left
|\int_{B(k_r, 2\delta)}(1-\Delta)^{[s]+1} \phi_{r}(k)\cdot e^{-ik\cdot \varepsilon g}dk\right|
\\
&\leq  \varepsilon^{d}\langle \varepsilon g\rangle^{-2[s]-2}\cdot \sup_k \left|(1-\Delta)^{[s]+1} \phi_{r}(k)\right| \lesssim \varepsilon^{d-s}\langle g \rangle^{-s}.
\end{split}
\eeq
In the last inequality, we make use of the fact that
\beq
\langle \varepsilon g \rangle^{-2[s]-2} \leq\langle \varepsilon g \rangle^{-s} \leq \varepsilon^{-s} \langle g \rangle^{-s}
\eeq
whenever $\varepsilon \in (0,1)$.
In particular, by choosing any $s \in (\max(0,N+d), d)$, one has
\beq
\label{nsp}
|\widehat{\phi^{\varepsilon}_{r}}(g)|
\cdot \langle g \rangle^{N}\lesssim \varepsilon^{d-s}\cdot \langle g \rangle^{N-s}.
\eeq
Now we substitute $\phi:=\phi^{\varepsilon}_{r}$ in \mref{a.e10}, use \mref{nsp}, and then take $\varepsilon \rightarrow 0^+$ to derive
\beq
|e_{k_r}(x)a_r(x)| \lesssim \lim_{\varepsilon \rightarrow 0}\varepsilon^{d-s}\sum_{g \in \bZ^d}\langle g \rangle^{N-s}=0.
\eeq
\end{enumerate}
$\square$

\subsection{Proof of Corollary \ref{RReqmu}}
It suffices to prove that $\Image{\tilde{P}}=\tilde{\Gamma}_{\mu}(\mathcal{X},P)$, since according to \mref{codim}, \beq
\dim\Ker{\tilde{P^*}}=\codim{\Image{\tilde{P}}}=0
\eeq
and then the conclusion of Corollary \ref{RReqmu} holds true as we mentioned in Section \ref{GSRR}. Now, given any $f \in \tilde{\Gamma}_{\mu}(\mathcal{X},P)$, one has $\langle f, \tilde{L}^- \rangle=0$ and $f=Pu$ for some $u \in \Dom{P}$, since $\Image{P}=\Domp{P^*}$. According to the assumption, we can find a solution $w=u-v$ in $\Dom{P}$ of the equation $Pw=Pu-Pv=f$ such that $\langle w, L^- \rangle=0$. Let $w_0$ be the restriction of $w$ on $\mathcal{X} \setminus D^+$. Clearly, $w_0$ belongs to the space $\Gamma(\mathcal{X}, \mu, P)$.
Since $Pw$ is smooth on $\mathcal{X}$, $\tilde{P}w_0=Pw=f$ by the definition of the extension operator $\tilde{P}$. This shows that $f \in \Image{\tilde{P}}$, which finishes the proof.

\begin{remark}
In the special case when $D^-=\emptyset$, $L^-=\{0\}$, one can prove the Riemann-Roch equality \mref{RReq} directly, i.e., without referring to the extension operators $\tilde{P}$ and $\tilde{P^*}$. For reader's convenience, we present this short proof following \cite{GSadv}. We define the space
$\Gamma(\mu, P):=\{u \in \mathcal{D}'(\mathcal{X}) \mid u \in \Dom_{D^+}{P}, Pu \in L^+\}$. Then it is easy to check that the following sequences are exact:
$$0 \rightarrow \tilde{L}^+ \xrightarrow{i} \Gamma(\mu, P) \xrightarrow{r} L(\mu, P) \rightarrow 0$$
$$0 \rightarrow \Ker{P} \xrightarrow{i} \Gamma(\mu, P) \xrightarrow{P} L^+ \rightarrow 0,$$
where $i$ and $r$ are natural inclusion and restriction maps. Here the surjectivity of $P$ from $\Gamma(\mu, P)$ to $L^+$ is a consequence of the existence of a properly supported pseudodifferential parametrix of $P$ (modulo a properly supported smoothing operator) and $C^{\infty}_c(\mathcal{X}) \subseteq \Domp{P^*}=\Image{P}$. Note that $\Ker{P^*}=\{0\}$.   Hence, it follows that
\beq
\dim L(\mu,P)=\dim \Gamma(\mu, P)-\dim \tilde{L}^+=\dim \Ker{P}+\dim L^+-\dim \tilde{L}^+=\Index{P}+\deg_P(\mu).
\eeq
\end{remark}

\chapter[Remarks]{Final Remarks and Acknowledgments}\label{C:remarks}

\section{Remarks and conclusions}
\begin{itemize}
\item An interesting discussion of issues related to the difference operators Definition \ref{Floquetsol} of Floquet functions can be found in the recent papers \cite{KhaArx,MinhLin}. There, a study of polynomial-like elements in vector spaces equipped with group actions is provided. These elements are defined via iterated difference operators. In the case of a full rank lattice acting on an Euclidean space, they are exactly polynomials with periodic coefficients, and thus are closely related to solutions of periodic differential equations. The main theorem of that work confirms that if the space of polynomial-like elements of degree zero is of finite dimension then for any $n \in \mathbb{Z}_+$, the space consisting of all polynomial-like elements of degree at most $n$ is also finite dimensional. Non-abelian groups are considered, with the hope to transfer at least some of the Liouville theorems to the case of nilpotent co-compact coverings (compare with \cite{LinPinchover}), albeit this goal has not been achieved yet.
\item
The Remark \ref{rrl2d-example} shows that Assumption $(\mathcal{A}2)$ cannot be dropped in Theorem \ref{RRLineq}. Besides the example given in Remark \ref{rrl2d-example}, we provide a heuristic explanation here.
It is known (see, e.g. \cite{ZKKP}) that if $\{A_t\}$ is a family of Fredholm operators that is continuous with respect to the parameter $t$, the kernel dimension $\dim \Ker{A_t}$ is upper semicontinuous. The idea in combining Riemann-Roch and Liouville theorems by considering dimensions of spaces of solutions with polynomial growth as some Fredholm indices would imply that the upper-semicontinuity property should hold also for these dimensions (Corollary \ref{uppersemicont}). On the other hand, as shown in \cite{Shterenberg}, there exists a continuous family $\{M_t\}$ of periodic operators on $\mathbb{R}^2$ such that for each $N \geq 0$,
\beq
\dim V^{\infty}_N(M_t)=2\dim V^{\infty}_N(M_{2\sqrt{3}})
\eeq
if $2\sqrt{3}<t<2\sqrt{3}+\varepsilon$ for some $\varepsilon>0$ and thus, $\dim V_N(M_t)$ is not upper-semicontinuous at $t=2\sqrt{3}$ (see \cite{KP2}). In this example, the minimum $0$ of the lowest band $\lambda_1(k)$ of the operator $M_{2\sqrt{3}}$ is degenerate \cite{Shterenberg}. This explains why our approach requires the ``non-degeneracy'' type condition $(\mathcal{A}2)$ for avoiding some intractable cases like the previous example. Notice that in general, Assumption $\mathcal{A}$ and Liouville type results are not stable under small perturbations.
\item
Since both the Riemann-Roch type results of \cite{GSadv, GSinv} and the Liouville type results of \cite{KP2} hold for elliptic systems, the results of this work could be easily extended (at the expense of heavier notations) to linear elliptic matrix operators (e.g., of Maxwell type) acting between vector bundles.
\item
Our results do not cover the important case when $A=\bar{\partial}$ on an abelian covering of a compact complex manifold. We plan to study this case in a separate paper.
\item A natural question that arises is of having the divisor being also periodic, rather than compact (compare with \cite{shubinRR}), which would require measuring ``dimensions'' of some infinite dimensional spaces, which in turn would require using some special von Neumann algebras and the corresponding traces. This task does not seem to be too daunting, but the authors have not addressed it here.
\item
This work shows that adding a pole at infinity (Liouville theorems) to the Riemann-Roch type results is not automatic and not always works. Moreover, the choice of the $L_p$ space for measuring growth in Liouville theorems is very relevant for the results (in \cite{KP1,KP2} only $p=\infty$ was considered).
In particular, when $p=2$, the Liouville-Riemann-Roch equality we obtained in \mref{LRRL2} could be viewed as an analog of the $L^2$-Riemann-Roch theorem in \cite{shubinRR} for finite divisors.
\end{itemize}

\section{Acknowledgements}\label{S:ackn}
The work of the authors was partially supported by NSF DMS grants. The authors express their gratitude to NSF for the support. M. Kha also thanks AMS and the Simons foundation for their travel grant support. Thanks also go to Y.~Pinchover and M.~Shubin for the insightful discussions of the topic.\backmatter
\bibliographystyle{amsalpha}

\printindex
\end{document}